\documentclass[aos,preprint]{imsart}
\RequirePackage[OT1]{fontenc}
\RequirePackage{amsthm,amsmath}
\RequirePackage[numbers]{natbib}
\RequirePackage[colorlinks,citecolor=blue,urlcolor=blue]{hyperref}
\usepackage{float}
\usepackage{amssymb,amsmath, amsthm, mathabx,mathtools}
\usepackage{tikz}
\usepackage{color, enumerate}
\usepackage{pgfplots}
\usepackage{bm}
\usepackage{stmaryrd}
\usepackage{graphicx}
\usepackage{caption}
\usepackage{subcaption}
\usepackage{epstopdf}

\startlocaldefs
\numberwithin{equation}{section}
\theoremstyle{plain}

\endlocaldefs

\usepackage{enumerate}
\usepackage{mathrsfs}
\usepackage{wrapfig}

\usepackage{color}

\numberwithin{equation}{section}

\newcommand{\bzeta}{\bm{\zeta}}
\newcommand{\bxi}{\bm{\xi}}

\newcommand{\Sig}{\Sigma}

\newcommand{\ri}{\mathrm{i}}
\newcommand{\rI}{\mathrm{I}}
\newcommand{\Ub}{\mathbf{U}}

\newcommand{\beq}{\begin{equation}}
\newcommand{\bEq}{\end{equation}}

\newcommand{\bx}{{\bf{x}}}

\newcommand{\al}{\alpha}

\newcommand{\be}{\begin{equation}}
\newcommand{\ee}{\end{equation}}

\newcommand{\e}{{\varepsilon}}

\newcommand{\si}{\sigma}

\newcommand{\ctQ}{\widetilde{\mathcal{Q}}}
\newcommand{\CTG}{\widetilde{\mathcal{G}}}
\newcommand{\tsig}{\widetilde{\sigma}}

\newcommand{\fa}{{\mathfrak a}}
\newcommand{\fb}{{\mathfrak b}}

\newcommand{\bU}{ {\bf  U}}

\renewcommand{\cal}{\mathcal}

\newcommand{\wh}{\widehat}
\newcommand{\wt}{\widetilde}

\newcommand{\ii}{\mathrm{i}} 
\newcommand{\dd}{\mathrm{d}}

\renewcommand{\epsilon}{\varepsilon}
\renewcommand{\leq}{\leqslant}
\renewcommand{\geq}{\geqslant}

\renewcommand{\le}{\leq}
\renewcommand{\ge}{\geq}

\newcommand{\E}{\mathbb{E}}

\newcommand{\C}{\mathbb{C}}
\newcommand{\N}{\mathbb{N}}

\newcommand{\norm}[1]{\lVert #1 \rVert}

\DeclareMathOperator{\tr}{Tr}

\DeclareMathOperator{\re}{Re}
\DeclareMathOperator{\im}{Im}

\DeclareMathOperator{\OO}{O}
\DeclareMathOperator{\oo}{o}

\DeclareMathOperator{\bv}{\mathbf{v}}
\DeclareMathOperator{\bu}{\mathbf{u}}
\DeclareMathOperator{\bw}{\mathbf{w}}

\DeclareMathOperator{\bbN}{\mathbb{N}}
\DeclareMathOperator{\bbP}{\mathbb{P}}

\theoremstyle{plain} 
\newtheorem{theorem}{Theorem}[section]
\newtheorem*{theorem*}{Theorem}
\newtheorem{lemma}[theorem]{Lemma}
\newtheorem{assumption}[theorem]{Assumption}
\newtheorem*{lemma*}{Lemma}
\newtheorem{corollary}[theorem]{Corollary}
\newtheorem*{corollary*}{Corollary}
\newtheorem{proposition}[theorem]{Proposition}
\newtheorem*{proposition*}{Proposition}

\newtheorem{definition}[theorem]{Definition}
\newtheorem*{definition*}{Definition}
\theoremstyle{remark}
\newtheorem{example}[theorem]{Example}
\newtheorem*{example*}{Example}
\newtheorem{remark}[theorem]{Remark}

\newtheorem*{remark*}{Remark}
\newtheorem*{remarks*}{Remarks}

\renewcommand{\Im}{{\rm{Im}}}
\renewcommand{\Re}{{\rm{Re}}}

\newcommand{\nc}{\normalcolor}
\newcommand{\vb}{\mathbf{v}}
\newcommand{\wb}{\mathbf{w}}

\allowdisplaybreaks

\begin{document}

\begin{frontmatter}
\title{Spiked separable covariance matrices and principal components}
\runtitle{Spiked separable covariance matrices}

\begin{aug}
\author{\fnms{Xiucai} \snm{Ding}\thanksref{T1}\ead[label=e1]{xiuca.ding@duke.edu}}
\and 
\author{\fnms{Fan} \snm{Yang}\thanksref{T2}\ead[label=e2]{fyang75@wharton.upenn.edu}}

\runauthor{X. Ding and F. Yang}

\affiliation{Duke University \thanksmark{T1} and University of Pennsylvania \thanksmark{T2} }

\address{ Department of Mathematics \\
 Duke University \\
120 Science Drive\\
 Durham, NC, 27710\\
\printead{e1}}

\address{
 Department of Statistics \\
 University of Pennsylvania\\
 3730 Walnut Street\\ 
 Philadelphia, PA 19104\\
\printead{e2}
}
\end{aug}

\begin{abstract}
We study a class of separable sample covariance matrices of the form $\widetilde{\mathcal{Q}}_1:=\widetilde A^{1/2} X \widetilde B X^* \widetilde A^{1/2}.$ Here $\widetilde{A}$ and $\widetilde{B}$ are positive definite matrices whose spectrums consist of bulk spectrums plus several spikes, i.e. larger eigenvalues that are separated from the bulks. Conceptually, we call  $\widetilde{\mathcal{Q}}_1$ a \emph{spiked separable covariance matrix model}. On the one hand, this model includes the spiked covariance matrix as a special case with $\widetilde{B}=I$. On the other hand, it allows for more general correlations of datasets. In particular, for spatio-temporal dataset, $\widetilde{A}$ and $\widetilde{B}$ represent the spatial and temporal correlations, respectively. 

 In this paper, we study the outlier eigenvalues and eigenvectors, i.e., the principal components, of the spiked separable covariance model $\widetilde{\mathcal{Q}}_1$. We prove the convergence of the outlier eigenvalues $\wt\lambda_i$ and the generalized components (i.e. $\langle \mathbf v, \wt\bxi_i\rangle$ for any deterministic vector $\mathbf v$) of the outlier eigenvectors $\wt\bxi_i$ with optimal convergence rates. Moreover, we also prove the delocalization of the non-outlier eigenvectors.  We state our results in full generality, in the sense that they also hold near the so-called BBP transition and for degenerate outliers. Our results highlight both the similarity and difference between the spiked separable covariance matrix model and the spiked covariance matrix model in \cite{principal}. In particular, we show that the spikes of both $\wt A$ and $\wt B$ will cause outliers of the eigenvalue spectrum, and the eigenvectors can help to select the outliers that correspond to the spikes of $\wt A$ (or $\wt B$).
\end{abstract}

\begin{keyword}[class=MSC]
\kwd[Primary ]{15B52}
\kwd{62E20}
\kwd[; secondary ]{62H99}
\end{keyword}

\begin{keyword}
\kwd{Spiked separable covariance matrices}
\kwd{Principal components}
\kwd{BBP transition}
\kwd{Local laws}
\end{keyword}

\end{frontmatter}

\section{Introduction}

High-dimensional data obtained at space-time points has been increasingly employed in various
scientific fields, such as geophysical and environmental sciences \cite{st3, st2}, wireless communications \cite{MIMO,tulino2004random, mimo1}, medical imaging \cite{skup2010} and financial economics \cite{factor1, factor2, yp2016}.  The structural assumption of separability is a popular assumption in
the analysis of spatio-temporal data. Although this assumption does not allow for space-time interactions in the covariance matrix, in many real data applications (e.g., the study of Irish wind speed \cite{jasadata}), the covariance matrix can be well approximated    using separable covariance matrices by solving a nearest Kronecker product for a space-time covariance matrix problem (NKPST)  \cite{st1}.


Consider a $p \times n$ data matrix $Y$ of the form
\begin{equation}\label{eq_defny}
Y=\widetilde{A}^{1/2} X \widetilde{B}^{1/2},
\end{equation}
where $X=(x_{ij})$ is a $p \times n$ random matrix with independent entries such that $\mathbb{E} x_{ij}=0$ and $\mathbb{E} |x_{ij}|^2=n^{-1},$ and $\widetilde{A}$ and $\widetilde{B}$ are respectively $p \times p$ and $n \times n$ deterministic positive-definite matrices. We say $Y$ has a separable covariance structure because the joint spatio-temporal covariance of $Y$, viewed as an $(np)$-dimensional vector consisting of
the columns of $Y$ stacked on top of one another, is given by a separable form $\widetilde{A} \otimes \widetilde{B}$, where $\otimes$ denotes the Kronecker product. This model has different names and meanings in different fields. For example, in wireless communications \cite{MIMO,tulino2004random, mimo1}, especially for the multiple-input-multiple-output (MIMO) systems, the $\widetilde{A}$ and $\widetilde{B}$ represent the covariances between the receiver antennas and between the transmitter antennas, respectively. Also, $Y$ is called the doubly-heteroscedastic noise in \cite{leeb} for matrix denoising and the separable idiosyncratic part in factor model \cite{factor1}. However, as a convention, in this paper we always say that the row indices of $Y$ correspond to spatial locations while the column indices correspond to time points. Moreover, we shall call $\widetilde{A}$ and $\widetilde{B}$ as spatial and temporal covariance matrices, respectively. In this paper, we are mainly interested in the so-called {\it separable sample covariance matrix} $\wt{\mathcal Q}_1:= YY^*$ for the above separable data model $Y$.

One special case 
is the classic sample covariance matrix when $\widetilde{B}=I_n$, which has been a central object of study in multivariate statistics. 
In the null case with $\wt A= I_p$, it is well-known that the empirical spectral distribution (ESD) of $\wt{\mathcal Q}_1$ converges to the celebrated Marchenko-Pastur (MP) law \cite{MP}.
Later on the convergence result of the ESD is extended to various settings with general positive definite covariance matrices $\widetilde{A}$; we refer the readers to the monograph \cite{bai2009spectral} and the review paper \cite{RMTreview}. For the extremal eigenvalues, 
the Tracy-Widom distribution \cite{TW1,TW} of the extremal eigenvalue was first proved in \cite{spikedmodel} for sample covariance matrices with $\widetilde{A}=I_p$ and Gaussian $X$ (i.e. the entries of $X$ are i.i.d.\;Gaussian), and later proved for $X$ with generally distributed entries in \cite{PY}. When $\wt A$ is a general non-scalar matrix, the Tracy-Widom distribution was first proved for the case with $i.i.d.$ Gaussian $X$ in \cite{Karoui,Onatski} and later proved under various moment assumptions on the entries $x_{ij}$ \cite{BPZ1,dingyang, Knowles2017, lee2016}. 
Finally, for the (non-outlier) sample eigenvectors, the completely delocalization \cite{Knowles2017, OROURKE2016361}, quantum unique ergodicity \cite{principal}, distribution of the eigenvector components \cite{ding2016} and convergence of eigenvector empirical spectral distribution \cite{XYY_VESD} have been constructed.

In the statistical study of sample covariance matrices, a popular model is the \emph{Johnstone's spiked  covariance matrix model} \cite{spikedmodel}. In this model, a few spikes, i.e., eigenvalues detached from the bulk eigenvalue spectrum, are added to $\widetilde{A}$. Since the seminal work of Baik, Ben Arous and P\'{e}ch\'{e} \cite{BBP}, it is now well-understood that the extremal eigenvalues undergo a so-called BBP transition along with the change of the strength of the spikes. Roughly speaking, there is a critical value such that the following properties hold: if the strength of the spike is smaller than the critical value, then the extremal eigenvalue of the spiked sample covariance matrix will stick to the right endpoint of the bulk eigenvalue spectrum (and hence is not an outlier), and the corresponding sample eigenvector will be delocalized; otherwise, if the strength of the spike is larger than the critical value, then the associated eigenvalue will jump out of the bulk eigenvalue spectrum, and the outlier sample eigenvector will be concentrated on a cone with axis parallel to the population eigenvector with an (almost) deterministic aperture. 
For an extensive overview of such results, we refer the reader to \cite{principal, ding2017, DP_spike} . 

One purpose of this paper is to generalize some important results for sample and spiked covariance matrices to the more general {separable and spiked separable covariance matrices.} The convergence of the ESD of separable covariance matrices to a limiting law were shown in \cite{Separable, WANG2014, Zhang_thesis}. The edge universality and delocalization of eigenvectors have been proved by the second author \cite{yang2018} for separable covariance matrices {\it without} spikes on $\widetilde{A}$ and $\widetilde{B}$. The convergence of VESD of separable covariance matrices was proved in \cite{yang_thesis}, which is an extension of the result in \cite{XYY_VESD}. 
Then the main goal of this paper is to study the outlier eigenvalues and eigenvectors of separable covariance matrices with spikes on both the spatial and temporal covariance matrices $\widetilde{A}$ and $\widetilde{B}$, which we shall refer to as the {\it spiked separable covariance matrices}.
The precise definition is given in Section \ref{sec_defspike}. 


In this paper, we derive precise large deviation estimates on the outlier eigenvalues and the generalized components of the outlier eigenvectors. In particular, our results give both the first order limits and the (almost) optimal rates of convergence of the relevant quantities. We now describe them briefly. Let $\wt A= \sum_{i=1}^p \widetilde{\sigma}_i^a \vb_i^a(\vb_i^a)^*$ and $\wt B= \sum_{\mu=1}^n \widetilde{\sigma}_\mu^b \vb_\mu^b(\vb_\mu^b)^*$ be the eigendecomposition of $\widetilde{A}$ and $\widetilde{B}$, respectively, where we label the eigenvalues in descending order. We assume that the spiked eigenvalues are $\{\wt\sigma_i^a\}_{i=1}^r$ and $\{\wt\sigma_\mu^b\}_{\mu=1}^s$, where $r$ and $s$ are some fixed integers. Then there exists a threshold $\ell_a$ (or $\ell_b$) such that $\widetilde{\sigma}_i^a$ (or $\widetilde{\sigma}_\mu^b$) gives rise to outliers of $\wt{\mathcal Q}_1$ if and only if $\widetilde{\sigma}_i^a>\ell_a $ (or $\widetilde{\sigma}_\mu^b >\ell_b$). Moreover, the outlier lies around a fixed location determined by the spike $\widetilde{\sigma}_i^a$ (or $\widetilde{\sigma}_\mu^b$); see Theorem \ref{thm_outlier}.
If $\widetilde{\sigma}_i^a - \ell_a \gg n^{-1/3}$ or $\widetilde{\sigma}_j^b - \ell_b \gg n^{-1/3}$, i.e. the spike is supercritical, then the outlier will be well-separated from the bulk spectrum 
and can be detected readily. 
For $0< \widetilde{\sigma}_i^a - \ell_a \ll n^{-1/3}$ or $0< \widetilde{\sigma}_j^b - \ell_b \ll n^{-1/3}$, i.e. the spike is subcritical, the corresponding ``outlier" cannot be distinguished from the bulk spectrum and will instead stick to the right-most edge of the bulk spectrum up to some random fluctuation of order $\OO(n^{-2/3})$. 
Next for the sample eigenvector of $\wt{\mathcal Q}_1$ that is associated with the outlier caused by a {supercritical} spike $\wt\sigma_i^a$, we show that it is concentrated on a cone with axis parallel to the population eigenvector $\vb_i^a$ with an explicit aperture determined by $\widetilde{\sigma}_i^a$. On the other hand, the sample eigenvector of $\wt{\mathcal Q}_1$ that is associated with a {supercritical} spike $\widetilde{\sigma}_\mu^b$ is  delocalized. Similar results hold for the right singular vectors of $Y$, i.e. the eigenvectors of $\wt{\mathcal Q}_2:=\wt B^{1/2} X^* \wt A X \wt B^{1/2}$, by switching the roles of $\wt A$ and $\wt B$. Finally, for the non-outlier singular vectors, {i.e., singular vectors associated with subcritical and bulk eigenvalues}, we proved that they are delocalized. We point out that our results are in the same spirit as the ones for deformed Wigner matrix \cite{KY2013}, deformed rectangular matrix  \cite{bgn2012, ding20171} and spiked  covariance matrices \cite{principal, ding2017, DP_spike}. 
 
The information from sample singular vectors is very important in the estimation of spiked separable covariance matrices. For example, one important parameter to estimate is the number of spikes. For spiked separable covariance matrices, the outliers have two different origins from either $\widetilde{A}$ or $\widetilde{B}$. Hence we need to estimate the number of spikes for each of them. In the literature of spiked covariance matrices \cite{PY14}, the number of spikes is estimated using statistic constructed from eigenvalues only. However, this only gives an estimation of the total number of spikes. To distinguish the two types of spikes, we also need to utilize the information from singular vectors. This will be discussed in detail in Section \ref{sec:statapp}. 

Before concluding the introduction, we summarize the main contributions of our work. 
\begin{itemize}
\item We introduce the spiked separable covariance matrix model; see (\ref{eq_sepamodel}). It allows for more general covariance structure and is suitable for spatio-temporal data analysis with spikes in  both space and time. 

\item For both supercritical and subcritical spikes, we obtain the first order limits of the corresponding {eigenvalue} outliers and the generalized components of the associated eigenvectors. Moreover, our results provide a precise rate of convergence, which we believe to be optimal up to some $n^{\e}$ factor. They are presented in Theorems \ref{thm_outlier} and \ref{thm_eveout}.

\item We prove large deviation bounds for the non-outlier eigenvalues and eigenvectors. In particular, we prove that the non-outlier eigenvalues will stick with those of the reference matrix. Moreover, the non-outlier eigenvectors near the spectrum edge will be biased in the direction of the population eigenvectors of the subcritical spikes. These results are presented in Theorems \ref{thm_eigenvaluesticking} and \ref{thm_noneve}.   

\item We address two important issues in the estimation of spiked separable covariance matrices. First, we provide statistics to estimate the number of spikes for $\wt{A}$ and $\wt{B}$. In particular, we will show that the eigenvectors are important for us to separate the outliers from the spikes of $\wt A$ and those from the spikes of $\wt B$. Second, we obtain the optimal shrinkage for the eigenvalues, which is adaptive to the data matrix only. These are discussed in Section \ref{sec:statapp}.
\end{itemize}


This paper is organized as follows. In Section \ref{sec_defspike}, we define the spiked separable covariance matrix. In Section \ref{main_results}, we state our main results.
In Section \ref{sec:statapp}, we address two important issues regarding the statistical estimation of the proposed spiked separable covariance matrices.  We present the technical proofs in the supplementary material. 

\section{Definition of spiked separable covariance matrices}\label{sec_defspike}

\subsection{The model}

We first consider a class of separable sample covariance matrices of the form $\mathcal Q_1:=A^{1/2}XBX^*A^{1/2}$, where $A$ and $B$ are deterministic non-negative definite symmetric (or Hermitian) matrices. Note that $A$ and $B$ are not necessarily diagonal. We assume that $X=(x_{ij})$ is a $p\times n$ random matrix, where the entries $x_{ij}$, $1 \leq i \leq p$, $1 \leq j \leq n$, are real or complex independent random variables satisfying
\begin{equation}\label{eq_12moment} 
\mathbb{E} x_{ij} =0, \ \quad \ \mathbb{E} \vert x_{ij} \vert^2  = n^{-1}.
\end{equation}
For definiteness, in this paper we focus on the real case, that is, the random variables $x_{ij}$ are real. However, our proof can be applied to the complex case after minor modifications if we assume in addition that $\Re\, x_{ij}$ and $\Im\, x_{ij}$ are independent centered random variables with variance $(2n)^{-1}$. 
 We assume that the entries  $\sqrt{n}x_{ij}$ have bounded fourth moment:
\be\label{conditionA3} 
 \max_{i,j}\mathbb{E} | \sqrt{n}x_{ij} |^4  \le C_4, 
\ee
for some constant $C_4>0$. 
We will also use the $n \times n$ matrix $\mathcal Q_2:=B^{1/2}X^* A X B^{1/2}$. 
We denote the eigenvalues of $\mathcal Q_1$ and $\mathcal Q_2$ in descending order by $\lambda_1(\mathcal Q_1)\geq \ldots \geq \lambda_{p}(\mathcal Q_1)$ and $\lambda_1(\mathcal Q_2) \geq \ldots \geq \lambda_n(\mathcal Q_2)$. Since $\mathcal Q_1$ and $\mathcal Q_2$ share the same nonzero eigenvalues, we will simply write $\lambda_j$, $1\le j \le p \wedge n$, to denote the $j$-th eigenvalue of both $\mathcal Q_1$ and $\mathcal Q_2$ without causing any confusion.

We shall consider the high-dimensional setting in this paper. More precisely, we assume that there exists a constant $0<\tau <1$ such that the aspect ratio $d_n:= p/n$ satisfies 
\begin{equation}
 \tau \le d_n \le \tau^{-1} \ \ \text{ for all } n. \label{eq_ratio} 
 \end{equation}

We assume that $A$ and $B$ have eigendecompositions
\be\label{eigen}
A= V^a\Sig^a (V^a)^*, \quad  B= V^b \Sigma^b (V^b)^* ,
\ee
where
$$\Sig^a=\text{diag}(\si_1^a, \ldots, \si_p^a), \quad  \Sig^b=\text{diag}( \si_1^b, \ldots,  \si_n^b),$$
and
$$V^a= (\bv^a_1, \cdots, \bv^a_p),\quad V^b= (\bv^b_1, \cdots, \bv^b_n). $$ 
We denote the empirical spectral distributions (ESD) of $A$ and $B$ by
\begin{equation}\label{sigma_ESD}
\pi_A\equiv \pi_A^{(p)} := \frac{1}{p} \sum_{i=1}^p \delta_{\si_i^a} \ ,\quad \pi_B\equiv \pi_B^{(n)} := \frac{1}{n} \sum_{i=1}^n \delta_{\si_i^b}\ .
\end{equation}
We assume that there exists a small constant $0<\tau<1$ such that for all $n$ large enough,
\begin{equation}\label{assm3}
\max\{\si_1^a,  \sigma_1^b \} \le \tau^{-1}, \quad \max\left\{\pi_A^{(p)}([0,\tau]), \pi_B^{(n)}([0,\tau])\right\} \le 1 - \tau .
\end{equation}
Note the first condition means that the operator norms of $A$ and $B$ are bounded by $\tau^{-1}$, and the second condition means that the spectrums of $A$ and $B$ cannot concentrate at zero.

In this paper, we study spiked separable sample covariance matrices, which can be realized through a low rank perturbation of the non-spiked version. We shall assume that $\mathcal Q_1$ is a separable sample covariance matrix without spikes (see Assumption \ref{ass:unper} below).
To add spikes, we follow the setup in \cite{ding2017} and assume that there exist some fixed intergers $r, s\in \N$ and constants $d_i^a$, $1\le i \le r$, and $d_\mu^b$, $1\le \mu \le s$, such that
\begin{equation}\label{eq_defnsigmaa}
\begin{split}
\wt A= V^a\wt \Sig^a (V^a)^*, \quad & \wt B= V^b \wt \Sigma^b (V^b)^* ,\\
 \wt\Sig^a=\text{diag}(\wt\si_1^a, \ldots, \wt\si_p^a), \quad & \wt\Sig^b=\text{diag}( \wt\si_1^b, \ldots,  \wt\si_n^b),
 \end{split}
\end{equation}
where 
  \begin{equation}\label{eq_defnsigmab}
\widetilde{\sigma}_i^a=
 \begin{cases}
 \sigma_i^a(1+d^a_i),   & 1 \leq i \leq r \\
 \sigma_i^a,  & \text{otherwise}
 \end{cases},\qquad \widetilde{\sigma}_\mu^b=
 \begin{cases}
 \sigma_\mu^b(1+d^b_\mu),   & 1 \leq \mu \leq s \\
 \sigma_\mu^b,  & \text{otherwise}
 \end{cases}.
 \end{equation} 
 Without loss of generality, we assume that we have reordered indices such that 
 \be\label{reorder} \wt\si_1^a \ge \wt\si_2^a \ge \ldots \ge \wt\si_p^a \ge 0 \ , \quad  \wt\si_1^b \ge \wt\si_2^b \ge \ldots \ge  \wt\si^b_n \ge 0 \ .\ee
 Moreover, we assume that 
 \begin{equation}\label{assm33}
\max\{\widetilde\si_1^a,  \widetilde\sigma_1^b \} \le \tau^{-1} .
\end{equation}
 With (\ref{eq_defnsigmaa}) and (\ref{eq_defnsigmab}), we can write   
\be\label{AOBO} 
\begin{split}
& \widetilde A = A\Big(I_p+{V_o^a} {D}^{a} (V_o^a)^*\Big)=\Big(I_p+ V_o^{a} D^a(V_o^a)^*\Big)A, \\ 
& \widetilde B= B\Big(I_n+V_o^{b} {D}^{b} {(V_o^b)}^*\Big)=\Big(I_n+V_o^{b} {D}^{b} {(V_o^b)}^*\Big)B,
\end{split}
\ee
where 
$${D}^{a}=\text{diag}(d_1^{a},\cdots, d_r^a), \quad V_o^a=(\bv_1^a,\cdots, \bv^a_r),$$
and
$${D}^{b}=\text{diag}(d_1^{b},\cdots, d_s^b), \quad V_o^b=(\bv_1^b,\cdots, \bv^b_s).$$
Then we define the spiked separable sample covariance matrices as
\begin{equation}\label{eq_sepamodel}
\widetilde{\mathcal{Q}}_1=\widetilde A^{1/2} X \widetilde B X^* \widetilde A^{1/2}, \quad  \widetilde{\mathcal{Q}}_2=\widetilde B^{1/2} X^* \widetilde A X \widetilde B^{1/2}.
\end{equation} 

\begin{remark}\label{generaladdrmk}
In the above definition, we have assumed that the non-spiked covariance matrix $A$ (or $B$) and the spiked one $\wt A$ (or $\wt B$) share the same eigenvectors. Theoretically, the more general additive model actually can be reduced to our case as following:
consider the following model
\begin{equation*}
\widetilde{A}=A+\Delta_A,
\end{equation*}
where $A$ is the non-spiked part as above, and $\Delta_A$ is a finite rank perturbation. We can perform the eigendecomposition of $\wt A$ as
$$\wt A=\sum_{i=1}^p \wt\sigma_i^a \wt\bv_i^a (\wt\bv_i^a)^*,$$
where $\wt\bv_i^a$ are not necessarily the eigenvectors of $A$. Then we can decompose $\wt A$ in its eigenbasis as
\be\label{decomposewtA}\wt A=A'+\Delta_A', \quad A'=\sum_{i=1}^p  \sigma_i' \wt\bv_i^a (\wt\bv_i^a)^*,\ee
such that $A'$ is a non-spiked matrix and $\Delta_A'$ is a finite rank perturbation. This is reduced to our setting again. Similar discussion also applies to $\widetilde{B}$. 

In general, how the eigenvalues and eigenvectors of $\wt A$ are related to those of $A$ and $\Delta_A$ is unknown---we even do not know whether $\Delta'_A$ has the same rank as $\Delta_A$.  One possible assumption is that the eigenvalues (i.e. the signal strengths) of $\Delta_A$ are relatively large compared to those of $A$, then the largest few eigenvalues and the corresponding eigenvectors should be well approximated by those of $\Delta_A$, and our results can be applied again. However, the behaviors of the smaller eigenvalues can still be very interesting. For example, in Section \ref{additiveapp} of the supplement \cite{dysupple}, we construct an example such that $B=I$ and $\Delta_A$ is a rank-1 matrix with a large signal, but $\widetilde{Q}_1$ has two outlier eigenvalues. The behavior of the decomposition \eqref{decomposewtA} should depend strongly on the assumptions on $A$ and $\Delta_A$, and we will not pursue this direction in the current paper---it will be a subject for future study.
\end{remark}

 We summarize our basic assumptions here for future reference. For our purpose, we shall relax the assumption \eqref{eq_12moment} a little bit. 
\begin{assumption}\label{assm_big1}
 We assume that $X$ is a $p\times n$ random matrix with real entries satisfying \eqref{conditionA3} and that 
\begin{align}
\max_{i,j}\left|\mathbb{E} x_{ij}\right|   \le n^{-2-\tau}, \quad  \max_{i,j}\left|\mathbb{E} | x_{ij} |^2  - n^{-1}\right|   \le n^{-2-\tau}, \label{entry_assm1}
\end{align}
for some constant $\tau>0$. Note that (\ref{entry_assm1}) is slightly more general than (\ref{eq_12moment}). Moreover, we assume that both $A$ and $B$ are deterministic non-negative definite symmetric matrices satisfying \eqref{eigen} and (\ref{assm3}), $\wt A$ and $\wt B$ are deterministic non-negative definite symmetric matrices satisfying \eqref{eq_defnsigmaa}, \eqref{eq_defnsigmab}, \eqref{reorder} and \eqref{assm33}, and $d_n$ satisfies \eqref{eq_ratio}.
\end{assumption}

\subsection{Resolvents and limiting laws}
In this paper, we study the eigenvalue statistics of $\mathcal Q_{1}$, $\mathcal Q_{2}$ and $\wt {\mathcal Q}_{1}$, $\wt {\mathcal Q}_{2}$ through their {\it{resolvents}} (or  {\it{Green's functions}}). 
Throughout the paper, we shall denote the upper half complex plane and the right half real line by 
$$\mathbb C_+:=\{z\in \mathbb C: \im z>0\}, \quad \mathbb R^+:=[0,\infty).$$ 

\begin{definition}[Resolvents]\label{defn_resolvent}
For $z = E+ \ii \eta \in \mathbb C_+,$ we define the following resolvents for $\al=1,2$:  
\begin{equation}\label{def_green}
 \mathcal G_{\al}(X,z):=\left({\mathcal Q}_{\al}(X) -z\right)^{-1} , \ \ \ \wt{\mathcal G}_{\al} (X,z):=(\wt{\mathcal Q}_{\al}(X)-z)^{-1} .
\end{equation}
 We denote the ESD $\rho^{(p)}$ of ${\mathcal Q}_{1}$ and its Stieltjes transform as
\be\label{defn_m}
\rho^{(p)} := \frac{1}{p} \sum_{i=1}^p \delta_{\lambda_i({\mathcal Q}_1)},\quad m^{(n)}(z):=\int \frac{\rho^{(p)}(\dd x)}{x-z}=\frac{1}{p} \mathrm{Tr} \, \mathcal G_1(z).
\ee
\end{definition}

 It was shown in \cite{Separable} that if $d_n \to d \in (0,\infty)$ and $\pi_A^{(p)}$, $\pi_B^{(n)}$ converge to certain probability distributions, then almost surely $\rho^{(p)}$ converges to a deterministic distributions $ \rho_{\infty}$.
{We now give its definition.} 
For any finite $n$, $p=nd_n$ and $z\in \mathbb C_+$, we define $(m^{(n)}_{1c}(z),m^{(n)}_{2c}(z))\in \mathbb C_+^2$ as the unique solution to the following system of self-consistent equations
\begin{equation}\label{separa_m12}
\begin{split}
& {m^{(n)}_{1c}(z)} = d_n \int\frac{x}{-z\left[1+xm^{(n)}_{2c}(z) \right]} \pi_A^{(p)}(\dd x) ,\\ 
& {m^{(n)}_{2c}(z)} =  \int\frac{x}{-z\left[1+xm^{(n)}_{1c}(z) \right]} \pi_B^{(n)}(\dd x) .
\end{split}
\end{equation}
Then we define
\begin{equation}\label{def_mc}
m_c(z)\equiv m_c^{(n)}(z):= \int\frac{1}{-z\left[1+xm^{(n)}_{2c}(z) \right]} \pi_A^{(p)}(\dd x).
\end{equation}
It is easy to verify that $m_c(z)\in \mathbb C_+$ for $z\in \mathbb C_+$. Letting $\eta \downarrow 0$, we can obtain a probability measure $\rho_{c}^{(n)}$ with the inverse formula
\begin{equation}\label{ST_inverse}
\rho_{c}^{(n)}(E) = \lim_{\eta\downarrow 0} \frac{1}{\pi}\Im\, m^{(n)}_{c}(E+\ii \eta).
\end{equation}
If $d_n \to d \in (0,\infty)$ and $\pi_A^{(p)}$, $\pi_B^{(n)}$ converge to certain probability distributions, {then $\rho_{c}^{(n)}$ converges weakly as $n\to \infty$, and its weak limit is $\rho_\infty$.}


The above definitions of $m_c^{(n)}$, $\rho_c^{(n)}$ and $\rho_\infty$ make sense due to the following theorem. Throughout the rest of this paper, we often omit the super-indices $(p)$ and $(n)$ from our notations for simplicity.

\begin{theorem} [Existence, uniqueness, and continuous density]
For any $z\in \mathbb C_+$, there exists a unique solution $(m_{1c},m_{2c})\in \mathbb C_+^2$ to the systems of equations in (\ref{separa_m12}). The function $m_c$ in (\ref{def_mc}) is the Stieltjes transform of a probability measure $\mu_c$ supported on $\mathbb R^+$. Moreover, $\mu_c$ has a continuous derivative $\rho_c(x)$ on $(0,\infty)$.
\end{theorem}
\begin{proof}
See {\cite[Theorem 1.2.1]{Zhang_thesis}}, {\cite[Theorem 2.4]{Hachem2007}} and {\cite[Theorem 3.1]{Separable_solution}}.
\end{proof}

From (\ref{separa_m12}), it is easy to see that if we define the function
\begin{equation}\label{separable_MP}
f(z,m):=- m + \int\frac{x}{-z+xd_n \int\frac{t}{1+tm} \pi_A(\dd t)} \pi_B(\dd x) ,
\end{equation}
then $m_{2c}(z)$ can be characterized as the unique solution to the equation $f(z,m)=0$ that satisfies $\Im \, m> 0$ for $z\in \mathbb C_+$, and $m_{1c}(z)$ can be defined using the first equation in \eqref{separa_m12}.
Moreover,  {$m_{1c}(z)$ and $m_{2c}(z)$} are the Stieltjes transforms of the densities {$\rho_{1c}$ and $\rho_{2c}$}:
\begin{equation}\label{eq_inverse}
 \rho_{\al c}(E) = \lim_{\eta\downarrow 0} \frac{1}{\pi}\Im\, m_{\al c}(E+\ii \eta),\quad \al=1,2.
\end{equation}
Then we have the following result.

\begin{lemma}\label{lambdar}
The densities $\rho_{c}$, {$\rho_{1c}$ and $\rho_{2c}$} all have the same support on $(0,\infty)$, which is a union of intervals: {for $\al=1,2$,}
\begin{equation}\label{support_rho1c}
{\rm{supp}} \, \rho_{c} \cap (0,\infty) ={\rm{supp}} \, \rho_{\al c} \cap (0,\infty) = \bigcup_{k=1}^{L} [e_{2k}, e_{2k-1}] \cap (0,\infty),
\end{equation}
where $L\in \mathbb N$ depends only on $\pi_{A,B}$. Moreover, $(x,m)=(e_k, m_{2c}(e_k))$ are the real solutions to the equations
\begin{equation}
f(x,m)=0, \quad \text{and} \quad \frac{\partial f}{\partial m}(x,m) = 0. \label{equationEm2}
\end{equation}
Finally, we have $e_1 ={\rm O}(1)$, $m_{1c}(e_1) \in (-(\max_{\mu}\sigma_\mu^b)^{-1}, 0)$ and $m_{2c}(e_1) \in (-(\max_i \sigma_{i}^a)^{-1}, 0)$. 
\end{lemma}
\begin{proof}
See Section 3 of \cite{Separable_solution}.
\end{proof}

We shall call $e_k$ the spectral edges. In particular, we focus on the rightmost edge $\lambda_+ := e_1$. Now we make the following assumption. It guarantees a regular square-root behavior of the spectral densities {$\rho_{1c}$ and $\rho_{2c}$} near $\lambda_+$ and rules out the existence of outliers.

\begin{assumption} \label{ass:unper} 
There exists a constant $\tau>0$ such that 
\begin{equation}\label{assm_gap}
1+m_{1c}(\lambda_+) \max_\mu \sigma_\mu^b \geq \tau, \quad 1+m_{2c}(\lambda_+) \max_i \sigma_i^a \geq \tau.
\end{equation}
\end{assumption}

 \section{Main results}\label{main_results}

 In this section, we state the main results on the eigenvalues and eigenvectors of $\widetilde{\mathcal{Q}}_1$ and $\widetilde{\mathcal{Q}}_2$, together with some interpretations of these results. Their proof will be presented in the supplement.

Throughout this paper, we use the words {\it spikes} and {\it spiked eigenvectors} for those of the population matrices $\wt A$ and $\wt B.$ Meanwhile, we shall use the words {\it outlier eigenvalues} and {\it outlier eigenvectors} for those of the sample separable covariance matrices $\ctQ_1$ and $\ctQ_2.$

We will see that a spike $\wt\sigma_i^a$, $1\le i \le r$, or $\wt\sigma_\mu^b$, $1\le \mu \le s$, causes an outlier eigenvalue beyond $\lambda_+$, if 
\be\label{spikes}
 \tsig_{i}^a > - m_{2c}^{-1}(\lambda_+)  \quad \text{ or } \quad  \tsig_{\mu}^b > - m_{1c}^{-1}(\lambda_+),
 \ee
 where $m_{1c}(\cdot)$ and $m_{2c}(\cdot)$ are defined in (\ref{separa_m12}). Moreover, such an outlier is around a deterministic location
\be\label{g12c} 
\theta_1(\tsig_{i}^a): = g_{2c}\left(-(\tsig_i^a)^{-1}\right) \quad \text{or} \quad \theta_2(\tsig_{\mu}^b): = g_{1c}\left(-(\tsig_\mu^b)^{-1}\right),
\ee
where $g_{1c}$ and $g_{2c}$ are the inverse functions of $m_{1c}: (\lambda_+,\infty)\to (m_{1c}(\lambda_+),0)$ and $m_{2c}: (\lambda_+,\infty)\to (m_{2c}(\lambda_+),0)$, respectively. Note that the inverse functions exist because
\begin{equation}\label{real_stiel}
 m_{\al c}(x) = \int_0^{\lambda_+} \frac{\rho_{\al c}(t)}{t-x}\dd t, \quad x>\lambda_+, \quad \al=1,2,
 \end{equation}
are monotonically increasing functions of $x$ for $x> \lambda_+$. 


For $X$, we introduce the following bounded support condition.

\begin{definition}[Bounded support condition] \label{defn_support}
We say a random matrix $X$ satisfies the {\it{bounded support condition}} with $\phi_n$ if
\begin{equation}
\max_{i,j}\vert x_{ij}\vert \le \phi_n, \label{eq_support}
\end{equation}
where $ \phi_n$ is a deterministic parameter and usually satisfies $ n^{-{1}/{2}} \leq \phi_n \leq n^{- c_\phi} $ for some (small) constant $c_\phi>0$. Whenever (\ref{eq_support}) holds, we say that $X$ has support $\phi_n$. 
\end{definition}
The main reason for introducing this notation is as following: for a random matrix $X$ whose entries have at least $(4+\e)$-moments, it can be reduced to a random matrix with bounded support with probability $1-\oo(1)$ using a standard cut-off argument; see Corollary \ref{main_cor} below. 

\begin{assumption}\label{ass:spike} 
We assume that \eqref{spikes} holds for all $1\le i \le r$ and $1\le \mu \le s$. Otherwise, if \eqref{spikes} fails for some $\wt\sigma_i^a$ or $\wt\sigma_\mu^b$, we can simply redefine it as the unperturbed version $\sigma_i^a$ or $\sigma_\mu^b$. 
Moreover, we define the integers $0\le r^+ \leq r$ and $0\le s^+ \leq s$ such that 
\begin{equation}\label{eq_spikeassua}
\wt{\sigma}^a_{i} \geq - m_{2c}^{-1} (\lambda_{+}) + n^{-1/3}  + \phi_n  \quad \text{if and only if} \quad  1\le  i \le r^+,
\end{equation}
and 
\begin{equation} \label{eq_spikeassub}
\wt{\sigma}^b_{\mu}\ge - m_{1c}^{-1} (\lambda_{+}) + n^{-1/3} + \phi_n  \quad \text{if and only if} \quad 1\le  \mu \le s^+ .
\end{equation}
The lower bound $n^{-1/3}+\phi_n$ is chosen for definiteness, and {it can be replaced with any $n$-dependent parameter that is of the same order}.
\end{assumption}

\begin{remark}
 Consider the case where $\phi_n\le n^{-1/3}$ (this holds if we assume the existence of $12$-th moment). A spike $\wt\sigma_i^a$ or $\wt \sigma_\mu^b$ that does not satisfy \eqref{eq_spikeassua} or \eqref{eq_spikeassub} will give an outlier that lies within an $\OO(n^{-2/3})$ neighborhood of the rightmost edge $\lambda_+$. It is essentially indistinguishable from the extremal eigenvalue of $\mathcal{Q}_1$, which has typical fluctuation of order $n^{-2/3}$ around $\lambda_+$. Hence in \eqref{eq_spikeassua} and \eqref{eq_spikeassub}, we simply choose the ``real" spikes of $\wt A$ and $\wt B$.
\end{remark}

We will use the following notion of stochastic domination, which was first introduced in \cite{Average_fluc} and subsequently used in many works on random matrix theory, such as \cite{isotropic,principal,local_circular,Delocal,Semicircle,Knowles2017}. It simplifies the presentation of the results and their proofs by systematizing statements of the form ``$\xi$ is bounded by $\zeta$ with high probability up to a small power of $n$".

\begin{definition}[Stochastic domination]\label{stoch_domination}
(i) Let
\[\xi=\left(\xi^{(n)}(u):n\in\bbN, u\in U^{(n)}\right),\hskip 10pt \zeta=\left(\zeta^{(n)}(u):n\in\bbN, u\in U^{(n)}\right)\]
be two families of nonnegative random variables, where $U^{(n)}$ is a possibly $n$-dependent parameter set. We say $\xi$ is stochastically dominated by $\zeta$, uniformly in $u$, if for any fixed (small) $\epsilon>0$ and (large) $D>0$, 
\[\sup_{u\in U^{(n)}}\bbP\left(\xi^{(n)}(u)>n^\epsilon\zeta^{(n)}(u)\right)\le n^{-D}\]
for large enough $n \ge n_0(\epsilon, D)$, and we shall use the notation $\xi\prec\zeta$. Throughout this paper, the stochastic domination will always be uniform in all parameters that are not explicitly fixed (such as matrix indices, and $z$ that takes values in some compact set). Note that $n_0(\epsilon, D)$ may depend on quantities that are explicitly constant, such as $\tau$ in Assumption \ref{assm_big1} and \eqref{assm_gap}. If for some complex family $\xi$ we have $|\xi|\prec\zeta$, then we will also write $\xi \prec \zeta$ or $\xi=\OO_\prec(\zeta)$.

(ii) We extend the definition of $\OO_\prec(\cdot)$ to matrices in the weak operator norm sense as follows. Let $A$ be a family of random matrices and $\zeta$ be a family of nonnegative random variables. Then $A=\OO_\prec(\zeta)$ means that $\left|\left\langle\mathbf v, A\mathbf w\right\rangle\right|\prec\zeta \| \mathbf v\|_2 \|\mathbf w\|_2 $ uniformly in any deterministic vectors $\mathbf v$ and $\mathbf w$. Here and throughout the following, whenever we say ``uniformly in any deterministic vectors", we mean that ``uniformly in any deterministic vectors belonging to a set of cardinality $n^{\OO(1)}$".

(iii) We say an event $\Xi$ holds with high probability if for any constant $D>0$, $\mathbb P(\Xi)\ge 1- n^{-D}$ for large enough $n$.
\end{definition}

%
%
%
%

\subsection{Eigenvalue statistics} {In this subsection, we describe the results on the sample eigenvalues.} To state our result on the outlier eigenvalues, we first introduce the following labelling of such outliers.

\begin{definition}\label{defn_relabelling}
We define the labelling functions $\al:\{1,\cdots, p\}\to \N$ and $\beta:\{1,\cdots, n\}\to \N$ as follows. For any $1\le i \le r$, we assign to it a label $\alpha(i)\in \{1,\cdots, r+s\}$ if $\theta_1(\wt{\sigma}^a_{i})$ is the $\alpha(i)$-th largest element in $\{\theta_1(\wt{\sigma}^a_i)\}_{i=1}^r \cup  \{\theta_{2}(\wt{\sigma}^b_\mu)\}_{\mu=1}^s$. We also assign to any $1\le \mu \le s$ a label $\beta(\mu)\in \{1,\cdots, r+s\}$ in a similar way. Moreover, we define $\al(i)=i+s$ if $i>r$ and $\beta(\mu)=\mu + r$ if $\mu >s$.
We define the following sets of outlier indices:
\begin{align*}
& \mathcal O:= \{\al(i): 1\le i \le r\}\cup \{\beta(\mu): 1\le \mu \le s\}, 
\end{align*}
and
\begin{align*}
& \mathcal O^+:= \{\al(i): 1\le i \le r^+\}\cup \{\beta(\mu): 1\le \mu \le s^+\}.
\end{align*}
\end{definition}

We first state the results on the locations of the outlier and the first few non-outlier eigenvalues.  Denote the nontrivial eigenvalues of $\widetilde{\mathcal{Q}}_{1,2}$ by $\wt\lambda_1 \geq \wt\lambda_2 \geq \cdots \geq \wt\lambda_{n \wedge p}.$ For $1\le i \le r$ and $1\le \mu \le s$, we define 
\begin{align}\label{deltaimu}
& \Delta_1(\widetilde{\sigma}_i^a):= \left(\widetilde{\sigma}_i^a+m_{2c}^{-1}(\lambda_+)\right)^{1/2}, \quad \Delta_2(\widetilde{\sigma}_\mu^b):= \left(\widetilde{\sigma}_\mu^b+m_{1c}^{-1}(\lambda_+)\right)^{1/2}.
\end{align}

\begin{theorem}\label{thm_outlier}
 Suppose $X$ has bounded support $\phi_n$ such that $ n^{-{1}/{2}} \leq \phi_n \leq n^{- c_\phi} $ for some constant $c_\phi>0$. Suppose that Assumptions \ref{assm_big1}, \ref{ass:unper} and \ref{ass:spike} hold.  Then we have 
\begin{equation}\label{eq_spike}
\left|\wt\lambda_{\alpha(i)}-\theta_1(\wt\sigma_i^a)\right| \prec n^{-1/2}\Delta_1(\widetilde{\sigma}_i^a)  + \phi_n\Delta_1^2(\widetilde{\sigma}_i^a), \quad 1 \leq i \leq r^+,
\end{equation}
and 
\begin{equation}\label{eq_spike2}
\left|\wt\lambda_{\beta(\mu)}-\theta_2(\wt\sigma_\mu^b)\right| \prec n^{-1/2} \Delta_2(\widetilde{\sigma}_\mu^b) + \phi_n\Delta_2^2(\widetilde{\sigma}_\mu^b), \quad 1 \leq \mu \leq s^+ .
\end{equation}
Furthermore, for any fixed integer $\varpi>r+s$, we have 
\begin{equation}\label{eq_nonspike}
|\wt \lambda_{i}-{\lambda_{+}}| \prec n^{-2/3} + \phi_n^2, \quad  \text{for } \ i\notin \mathcal O^+  \text{ and }\ i \leq \varpi . 
\end{equation}
\end{theorem}

The above theorem gives the large deviation bounds for the locations of the outliers and the first few extremal non-outlier eigenvalues.  Again consider the case with $\phi_n\le n^{-1/3}$. Then Theorem \ref{thm_outlier} shows that the fluctuation of the outlier changes from the order $n^{-1/2}\Delta_1(\widetilde{\sigma}_i^a)$ to $n^{-2/3}$ when $\Delta_1(\widetilde{\sigma}_i^a)$ or $\Delta_2(\widetilde{\sigma}_\mu^b)$ crosses the scale $n^{-1/6}$. This implies the occurrence of the BBP transition \cite{BBP}. In a future work, we will show that under certain assumptions, the outlier eigenvalues are normally distributed, whereas the extremal non-outlier eigenvalues follow the Tracy-Widom law.

Next, we study the non-outlier eigenvalues of $\ctQ_1.$ We prove that the eigenvalues of $\ctQ_1$ for $i>r^++s^+$ are governed by \emph{eigenvalue sticking,} which states that the non-outlier eigenvalues of $\ctQ_1$ ``stick" with high probability to the eigenvalues of the reference matrix $\mathcal{Q}_1$.  Recall that we denote the eigenvalues of $\mathcal{Q}_1$  as $\lambda_1 \geq \lambda_2 \geq \cdots \geq \lambda_{p \wedge n}$. 

\begin{theorem}\label{thm_eigenvaluesticking}
 Suppose $X$ has bounded support $\phi_n$ such that $ n^{-{1}/{2}} \leq \phi_n \leq n^{- c_\phi} $ for some constant $c_\phi>0$.  Suppose that Assumptions \ref{assm_big1}, \ref{ass:unper} and \ref{ass:spike} hold. We define 
\begin{equation}\label{alpha+}
\alpha_+:=\min\left\{ \min_i \left|\tsig_i^a+m_{2c}^{-1}(\lambda_+)\right|, \min_\mu \left|\tsig_\mu^b+m_{1c}^{-1}(\lambda_+)\right|\right\}.
\end{equation}
 Assume that $\al_+ \ge n^{c_0} \phi_n$ for some constant $c_0>0$.  Fix any sufficiently small constant $\tau>0.$ We have that for $ 1 \leq i \leq \tau n$,
\begin{equation} \label{eq_stickingeq}
\left|\wt\lambda_{i+ r^+ + s^+ }-\lambda_i\right| \prec \frac{1}{n \alpha_+} + n^{-3/4} + i^{1/3}n^{-5/6} + n^{-1/2}\phi_n +  i^{-2/3}n^{-1/3}\phi_n^2 . 
\end{equation}
If either (a) the third moments of the entries of $X$ vanish in the sense that  
\be\label{assm_3rdmoment}
\mathbb E x_{ij}^3 = 0,\quad 1\le i \le p, \ \ 1\le j \le n,
\ee
or (b) either $A$ or $B$ is diagonal, then we have the stronger estimate
\begin{equation} \label{eq_stickingeq_strong}
\left|\wt\lambda_{i+ r^+ + s^+}-\lambda_i\right| \prec \frac{1}{n \alpha_+}, \quad 1 \leq i \leq \tau n. 
\end{equation}
\end{theorem}

Theorem \ref{thm_eigenvaluesticking} establishes the large deviation bounds for the non-outlier eigenvalues of $\ctQ_1$ with respect to the eigenvalues of $\mathcal{Q}_1$. In particular, when $\alpha_+ \gg n^{-1/3}$ and $\phi_n \ll n^{-1/6}$, the right-hand side of (\ref{eq_stickingeq}) or (\ref{eq_stickingeq_strong}) is much smaller than $n^{-2/3}$ for $i =\OO(1)$. In fact it was proved in \cite{yang2018} that the limiting joint distribution of the first few eigenvalues $\{\lambda_i\}_{1\le i \le k}$ of $\mathcal Q_1$ is universal under an $n^{2/3}$ scaling for any fixed $k\in \N$. Together with \eqref{eq_stickingeq}, this implies that the limiting distribution of the largest non-outlier eigenvalues of $\ctQ_1$ is also universal under an $n^{2/3}$ scaling as long as $\al_+\gg n^{-1/3}$ and $\phi_n \ll n^{-1/6}$. In a future paper, we will prove that $\{n^{2/3}(\lambda_i-\lambda_+)\}_{1\le i \le k}$ converges to the Tracy-Widom law for any fixed $k\in \mathbb N$, which immediately implies that the largest non-outlier eigenvalues of $\ctQ_1$ also satisfy the Tracy-Widom law.

\begin{remark} \label{remark_generalestimation}
The Theorems \ref{thm_outlier} and \ref{thm_eigenvaluesticking} can be combined to potentially estimate the spikes of $\wt A$ and $\wt B$ if {they are} low-rank perturbations of identity matrices. By Theorem \ref{thm_outlier}, the spike $\wt \sigma^a_i$ or $\wt \sigma^b_{\mu}$ can be effectively estimated using $-m^{-1}_{2c}(\wt\lambda_{\alpha(i)})$ or $-m^{-1}_{1c}(\wt\lambda_{\beta(\mu)})$. Although calculating $m_{1c}$ and $ m_{2c}$ needs the knowledge of the spectrums of $A$ and $B$, we will see that $m_{1c}$ and $m_{2c}$ can be well approximated using the eigenvalues of $\ctQ_1$ and $\ctQ_2$ only. We record such result in Theorem \ref{thm_adaptiveest}.

On the other hand, for the non-spiked eigenvalues, to our best knowledge there does not exist any literature on the estimation of the spectrums of general $A$ and $B$ using the eigenvalues of $\mathcal Q_1$ and $\mathcal Q_2$ only. However, for sample covariance matrices with $B=I$, the spectrum of $A$ can be estimated using the eigenvalues of $A^{1/2}XX^*A^{1/2}$ by solving a convex optimization problem involving the self-consistent equation for $m_{2c}$ in \cite{elkaroui2008,kong2017}. 
In the future work, we will try to generalize their results to the separable covariance matrices with more general $B.$ Note that although we cannot observe the eigenvalues of $\mathcal{Q}_1,$ Theorem \ref{thm_eigenvaluesticking} implies that the non-outlier eigenvalues of $\ctQ_1$ are close to those of $\mathcal{Q}_1.$      
\end{remark}

\begin{remark} 
{ We have seen from Theorem \ref{thm_outlier} that the locations of the outlier eigenvalues depend on the spikes and the spectrums of both $A$ and $B.$ Consider the case with $r=s=1$ and  supercitical spikes (c.f. Assumption \ref{assum_supercritical}). By (\ref{eq_spike}), we see that the outlier locations depend on  the 4-tuple $(\widetilde{\sigma}^a, \widetilde{\sigma}^b, \bm{\sigma}(A), \bm{\sigma}(B)),$ where $\widetilde{\sigma}^a$ and $\widetilde{\sigma}^b$ are the spikes associated with $A$ and $B$, respectively, and $\bm{\sigma}(A)$ and $\bm{\sigma}(B)$ denote the spectrums of $A$ and $B$. 
In general, the 4-tuple is not jointly identifiable. Indeed, even the pair $(\bm{\sigma}(A), \bm{\sigma}(B))$ is not jointly identifiable \cite{LU2005449}. 

To handle this issue,  one needs to impose some constraints. For instance, when $B=I_n,$ $\widetilde{\sigma}^a$ can be efficiently estimated using the eigenvalues of $\widetilde{\mathcal{Q}}_1$ by Theorem \ref{thm_adaptiveest}. Moreover, as mentioned in Remark \ref{remark_generalestimation}, the spectrum of $A$ can be estimated using the methods mentioned in \cite{elkaroui2008, kong2017, LW2015}. In this situation,  $(\widetilde{\sigma}^a, \bm{\sigma}(A))$ is identifiable. More generally, 
assume we know that the two triplets $(\widetilde{\sigma}_\alpha^a, \bm{\sigma}(A_\alpha), \bm{\sigma}(B))$ and $(\widetilde{\sigma}_\beta^a, \bm{\sigma}(A_\beta), \bm{\sigma}(B))$ share the same temporal covariance matrix $B$. Then using their sample eigenvalues $\{\widetilde{\lambda}^\alpha_k\}$ and $\{\widetilde{\lambda}_k^\beta\},$  we can employ the following two-step procedure to check whether they are identifiable. 

\vspace{5pt}
\noindent{\bf Step (i)}: Checking whether they have the same number of outliers and whether the outliers share the same values. 
More precisely 
given a threshold $\omega \rightarrow 0,$ we need to check whether $|\widetilde{\lambda}_k^{\alpha}-\widetilde{\lambda}_k^\beta| \leq \omega$, $1\le k \le r$, where $r$ is the number of outliers. If this does not hold true, then the two triples are different according to Theorem \ref{thm_outlier}. Otherwise, we continue with the second step. 

\vspace{5pt}

\noindent{\bf Step (ii)}: Checking whether the spectrums of $A_{\alpha}$ and $A_{\beta}$ are the same. In fact, the eigenvalues of $\mathcal{Q}_1$ are determined by the spectrums of $A$ and $B$; see the eigenvalues rigidity result, Theorem \ref{thm_largerigidity}, in the supplement \cite{dysupple}. Then with Theorem \ref{thm_eigenvaluesticking},  if $\bm{\sigma}(A_{\alpha})=\bm{\sigma}(A_\beta)$, we should have $|\widetilde{\lambda}_k^\alpha-\widetilde{\lambda}_k^\beta|\le \omega$, $k\ge r+1$, for the non-outliers. 
If this does not hold true, we claim that these two triplets are different. 

\vspace{5pt}

Finally, we mention that for a rigorous  statement of the above hypothesis testing on whether $(\widetilde{\sigma}_\alpha^a, \bm{\sigma}(A_\alpha), \bm{\sigma}(B))$ and $(\widetilde{\sigma}_\beta^a, \bm{\sigma}(A_\beta), \bm{\sigma}(B))$ are the same,  we need to derive the second order asymptotics of the eigenvalues. This will be our future work. }
\end{remark}

\subsection{Eigenvector statistics} {In this subsection, we state the results on the eigenvectors of $\ctQ_1$ and $\ctQ_2.$} We denote the eigenvectors of $\ctQ_1$ by $\wt{\bm\xi}_k$, $1\le k\le p$, and the eigenvectors of $\ctQ_2$ by $\wt{\bm \zeta}_\mu$, $1\le \mu\le  n$. To remove the arbitrariness in the definitions of eigenvectors, 
we shall consider instead the products of generalized components
$$\langle \bv,\wt{\bm\xi}_k\rangle\langle \wt{\bm\xi}_k,\bw\rangle ,\quad \langle \bv',\wt{\bm\zeta}_k\rangle\langle \wt{\bm\zeta}_k,\bw'\rangle ,$$
where $\vb, \wb, \vb'$ and $\wb'$ are some given deterministic vectors. Note that these products characterize the eigenvectors $\wt{\bm\xi}_k$ and $\wt{\bm\zeta}_k$ completely up to the ambiguity of a phase. More generally, if we consider degenerate  or near-degenerate outliers, then 
only eigenspace matters. {Here the degenerate (or near-degenerate) outliers refer to the outliers corresponding to identical (or near-degenerate) population spikes.} 
As in \cite{principal}, we shall consider the generalized components $\langle \bv,\cal P_S\bw\rangle$ of the random projection
$$ \cal P_S:=\sum_{k\in S} \wt{\bxi}_k\wt{\bxi}_k^*, \quad \text{ for } S\subset \mathcal O^+.$$
In particular, in the non-degenerate case $S=\{k\}$, the generalized components of $\cal P_S$ are the products of the generalized components of $\wt\bxi_k$.

 For $1\le i \le r^+$, $1\le  j \le p$ and $1\le \nu \le n$, we define
 \begin{equation}\label{eq_nu}
\delta_{\al(i), \al(j)}^{a}:=|\tsig^{a}_j-\tsig^{a}_i|, 
\quad \delta_{\al(i), \beta(\nu)}^{a}:=
\left| \tsig^{b}_\nu + m_{1c}^{-1}(\theta_1( \tsig^{a}_i)\right|.
\end{equation}
Similarly, for $1\le \mu \le s^+$, $1\le  j \le p$ and $1\le \nu \le n$, we define
 \begin{equation}\label{eq_nu2}
\delta_{\beta(\mu), \al(j)}^{b}:=|\tsig^{a}_j + m_{2c}^{-1}(\theta_2(\tsig^{b}_\mu))|, \quad \delta_{\beta(\mu), \beta(\nu)}^{b}:=|\tsig^{b}_\nu-\tsig^{b}_\mu|.
\ee
Given any $S \subset \mathcal O^+$, if $\mathfrak a\in S$, then we define
$$\delta_{\mathfrak a}(S):=\begin{cases}\left( \min_{ k:\al(k)\notin S}\delta^a_{\mathfrak a, \al(k)}\right)\wedge \left( \min_{\mu:\beta(\mu)\notin S}\delta^a_{a, \beta(\mu)}\right), & \ \text{if } \mathfrak a=\al(i) \in S\\
\left( \min_{k:\al(k)\notin S}\delta^b_{\mathfrak a, \al(k)}\right)\wedge \left( \min_{\mu:\beta(\mu)\notin S}\delta^b_{\mathfrak a,\beta(\mu)}\right), & \ \text{if } \mathfrak a=\beta(\mu) \in S
\end{cases};$$
if $\mathfrak a\notin S$, then we define
$$\delta_{\mathfrak a}(S):=\left( \min_{k:\al(k)\in S}\delta^a_{\al(k), \mathfrak a}\right)\wedge \left( \min_{\mu:\beta(\mu)\in S}\delta^b_{\beta(\mu), \mathfrak a}\right).$$
We now state the results on the left outlier singular vectors of ${\wt A}^{1/2}X{\wt B}^{1/2}, $ i.e., the outlier eigenvectors of $\ctQ_1.$ 


\begin{theorem}\label{thm_eveout} 
{Suppose $X$ has bounded support $\phi_n$ such that $ n^{-{1}/{2}} \leq \phi_n \leq n^{- c_\phi} $ for some constant $c_\phi>0$. Suppose that Assumptions \ref{assm_big1}, \ref{ass:unper} and \ref{ass:spike} hold. Fix any $S\subset \mathcal O^+$, we define the following deterministic positive quadratic form
\be\label{defn_ZA}
\langle \bv, \cal Z_S\bv\rangle := \sum_{i:\al(i)\in S}\frac{|v_i|^2}{\wt{\sigma}_i^a} \frac{g_{2c}'(-(\wt{\sigma}_i^a)^{-1})}{g_{2c}(-(\wt{\sigma}_i^a)^{-1})} , \ \ \text{for } \ \bv\in \mathbb C^p, \ \ v_i : =\langle \bv_i^a, \bv\rangle.
\ee
Then for any deterministic vector $\bv\in \mathbb C^p$, we have that
\begin{equation}\label{eq_spikePA}
\begin{split}
&\left| \langle \bv, \cal P_S\bv\rangle- \langle \bv, \cal Z_S\bv\rangle \right| \prec  \sum_{1\le i\le r:\al(i)\in S}|v_i|^2  \psi_{1}(\wt\sigma_i^a)  \\
&+\sum_{1\le i\le r:\al(i)\notin S}|v_i|^2  \frac{\phi_n^2}{\delta_{\al(i)}(S)} + \sum_{i=1}^p {|v_i|^2} \left(\frac{\psi_{1}^2(\wt\sigma_i^a)\Delta_1^2(\widetilde{\sigma}_i^a) }{\delta^2_{\al(i)}(S)} + \frac{\kappa_i}{n^{1/2}}\right) \\
&+ \langle \bv , \cal Z_S\bv \rangle^{1/2}\left[ \sum_{1\le i\le r:\al(i)\notin S}|v_i|^2  \frac{\phi_n^2}{\delta_{\al(i)}(S)}+ \sum_{1\le i \le p:\al(i)\notin S} {|v_i|^2} \left(\frac{\psi_{1}^2(\wt\sigma_i^a)\Delta_1^2(\widetilde{\sigma}_i^a) }{\delta^2_{\al(i)}(S)} +  \frac{\kappa_i}{n^{1/2}}\right)\right]^{1/2} ,
\end{split}
\end{equation}
where we denote
 $$\psi_{1}(\wt\sigma_i^a):= \phi_n +  n^{-1/2}\Delta_1^{-1}(\widetilde{\sigma}_i^a) .$$
 If we have (a) \eqref{assm_3rdmoment} holds, or (b) either $A$ or $B$ is diagonal, then the above estimate holds without the $ n^{-1/2}\kappa_i$ terms.}
\end{theorem}

\begin{remark}
For any deterministic vectors $\vb, \wb \in \mathbb{C}^p,$ we can state Theorem \ref{thm_eveout}  for more general quantities of the form $\langle \bv, \cal Z_S\bw\rangle$ using the polarization identity. Moreover, $\cal Z_S$ is a matrix that is uniquely determined by the quadratic form in \eqref{defn_ZA}. It can be  written as
$$\mathcal Z_S= \sum_{i:\alpha(i)\in S}\mathbf v_i^a(\mathbf v_i^a)^*\frac{1}{\wt{\sigma}_i^a} \frac{g_{2c}'(-(\wt{\sigma}_i^a)^{-1})}{g_{2c}(-(\wt{\sigma}_i^a)^{-1})}.$$
\end{remark}

The index set $S$ in Theorem \ref{thm_eveout} can be chosen according to user's goal. We now consider two typical cases to illustrate the idea. 

\begin{example}[Non-degenerate case]\label{exam_nondege} 
If all the outliers are well-separated, then we can choose $S=\{\alpha(i)\}$ or $S=\{\beta(\mu)\}$. For example, suppose $S=\{\alpha(i)\}$ and $\vb=\vb_i^a.$ Denote $\delta_{\alpha(i)}:= \delta_{\alpha(i)}(\{\alpha(i)\})$. Then we get from (\ref{eq_spikePA}) that 
\begin{equation}\nonumber 
|\langle \vb_i^a, { \wt{\bm \xi}_{\alpha(i)}} \rangle|^2=\frac{1}{\tsig_i^a} \frac{g_{2c}'(-(\tsig_i^a)^{-1})}{g_{2c}(-(\tsig_i^a)^{-1})}+\OO_{\prec} \left (\psi_{1}(\wt\sigma_i^a)+\frac{\psi_{1}^2(\wt\sigma_i^a)\Delta_1^2(\wt\sigma_i^a)}{n \delta_{\alpha(i)}^2}  \right ).
\end{equation}
Note that $\wt{\bm \xi_i}$ is concentrated on a cone with axis parallel to $\vb_i^a$ if the error term is much smaller than the first term, which is of order 
$$\frac{1}{\tsig_i^a} \frac{g_{2c}'(-(\tsig_i^a)^{-1})}{g_{2c}(-(\tsig_i^a)^{-1})} \sim \tsig_i^a+m_{2c}^{-1}(\lambda_+)$$
by Lemma \ref{lem_derivativeprop} in the supplement. This leads to the following conditions
\begin{equation}\label{eq_sepe}
{  \tsig_i^a+m_{2c}^{-1}(\lambda_+) \gg \phi_n + n^{-1/3}, \quad \delta_{\alpha(i)} \gg \phi_n +  n^{-1/2}\Delta_1^{-1}(\widetilde{\sigma}_i^a).}
\end{equation}
The first condition means that $\wt\lambda_{\al(i)}$ is truly an outlier (c.f. Theorem \ref{thm_outlier}), whereas the second condition is a \emph{non-overlapping condition}. 
In fact, by (\ref{eq_spike}), $\widetilde{\lambda}_{\al(i)}$ fluctuates around $\theta_1(\tsig_i^a)$ on the scale of order $  n^{-1/2}\Delta_1 (\widetilde{\sigma}_i^a) +\phi_n \Delta_1^2 (\widetilde{\sigma}_i^a)  $. Therefore, $\widetilde{\lambda}_{\al(i)}$ is well-separated from the other outlier eigenvalues if 
\begin{equation}\label{eq_connonover}
\begin{split}
\left( \min_{\al(j) \in \mathcal{O} \setminus \{\al(i)\}} |\theta_1(\tsig_i^a)-\theta_1(\tsig_j^a) | \right) \wedge  \left( \min_{\beta(\mu) \in \mathcal{O}} |\theta_1(\tsig_i^a)-\theta_2(\tsig_{\mu}^b)|  \right) \\
\gg   n^{-1/2}\Delta_1 (\widetilde{\sigma}_i^a) +\phi_n \Delta_1^2 (\widetilde{\sigma}_i^a) .  
\end{split}
\end{equation} 
Moreover, by Lemma \ref{lem_derivativeprop} in the supplement, the left-hand side of \eqref{eq_connonover} is of order $\delta_{\alpha(i)}\Delta_1^2 (\widetilde{\sigma}_i^a)$. This gives the second condition in \eqref{eq_sepe}. 

\end{example}

%

For degenerate or near-degenerate outliers, 
their indices should be included in the same set $S$. We now consider an example with multiple outliers that share exactly the same classical location.

\begin{example}[Degenerate case]\label{exam_dege} 
Suppose that we have an $|S|$-fold degenerate outlier, i.e., for some $\theta_0>\lambda_+$, 
\begin{equation*}
\theta_1(\tsig_i^a)=\theta_2(\tsig_\mu^b)= \theta_0, \quad \text{ for all } \ \al(i), \beta(\mu) \in S.
\end{equation*}
Suppose the outlier $\theta_0$ is well-separated from both the bulk and the other outliers (i.e., with distances of order 1). Then by (\ref{eq_spikePA}), we have that 
\begin{equation*}
\mathcal{P}_S =\sum_{\al(i)\in S} \frac1{\tsig_i^a} \frac{g_{2c}'(-(\tsig_i^a)^{-1})}{g_{2c}(-(\tsig_i^a)^{-1})}\bv_i^a(\bv_i^a)^*+ \mathcal E,
\end{equation*} 
where $\mathcal E$ is an error that is delocalized in the basis of $\bv_i^a$, i.e. $\langle \bv_i^a ,\mathcal E\bv_j^a\rangle \prec \phi_n  $. This can be regarded as a generalized cone concentration for the subspace spanned by $\{\wt\bxi_{\mathfrak a}\}_{\mathfrak a\in S}$.
\end{example}

\nc

Then we state the delocalization results on the non-outlier eigenvectors when $\al(i) \notin \mathcal{O}^+.$  Denote 
\begin{equation*}
 \eta_i:=n^{-3/4}+n^{-5/6} i^{1/3} +  n^{-1/2}\phi_n, \quad \kappa_i:=i^{2/3}n^{-2/3}. 
\end{equation*}
\begin{theorem}\label{thm_noneve}
 Suppose $X$ has bounded support $\phi_n$ such that $ n^{-{1}/{2}} \leq \phi_n \leq n^{- c_\phi} $ for some constant $c_\phi>0$.  Suppose that Assumptions \ref{assm_big1}, \ref{ass:unper} and \ref{ass:spike} hold. Fix any sufficiently small constant $\tau>0$. For  $\al( i)\notin \mathcal O^+$, $ i \leq {\tau}p$ and any deterministic vector $\bv \in \mathbb{C}^p$, we have 
\begin{equation}\label{eq_evebulka}
 |\langle \bv, \wt\bxi_{\al(i)} \rangle|^2 \prec \sum_{j=1}^p |v_j|^2 \frac{n^{-1}+  \eta_i \sqrt{\kappa_i} +\phi_n^3}{|\tsig_j^a+m_{2c}^{-1}(\lambda_+)|^2 + \phi_n^2+\kappa_i}.
\end{equation}
If we have (a) \eqref{assm_3rdmoment} holds, or (b) either $A$ or $B$ is diagonal, then the following stronger estimate holds:
\begin{equation}\label{eq_evebulka_strong}
 |\langle \bv, \wt\bxi_{\al(i)} \rangle|^2 \prec  \sum_{j=1}^p |v_j|^2 \frac{n^{-1} + \phi_n^3}{ |\tsig_j^a+m_{2c}^{-1}(\lambda_+)|^2+ \phi_n^2+\kappa_i }.
\end{equation}
\end{theorem}

\begin{remark} 
Note that for $\phi_n\le n^{-1/3}$ and $i\le n^{1/4}$, we have $\eta_i \sqrt{\kappa_i} +\phi_n^3=\OO(n^{-1})$. Hence \eqref{eq_evebulka} becomes the stronger estimate \eqref{eq_evebulka_strong} for the non-outlier eigenvalues with indices $i\le n^{1/4}$.
\end{remark}

\begin{example}
 Again we assume that $\phi_n\le n^{-1/3}$.  If $\tsig_j^a + m_{2c}^{-1}(\lambda_+)\gtrsim 1$, i.e. {$\tsig_j^a$} is well separated from the threshold, then $\wt{\bxi}_{\alpha(i)}$ is completely delocalized in the direction of $\vb_j^a$ for all $i\notin \mathcal O^+$ and $i\le n^{1/4}$. We next consider the outliers that are close to the threshold.  

Suppose that $i \le C$, i.e. $\wt\lambda_i$ is near the edge. Then (\ref{eq_evebulka_strong}) gives 
\begin{equation}\label{eq_partial}
|\langle \vb_j^a, \wt{\bxi}_{\alpha(i)} \rangle|^2 \prec \frac{1}{n(|\tsig_j^a+m_{2c}^{-1}(\lambda_+)|^2+n^{-2/3})}.
\end{equation}
 Therefore, the delocalization bound for the generalized component $|\langle \vb_j^a, \wt{\bxi}_{\alpha(i)} \rangle|$ 
changes from the optimal order $n^{-1/2}$ to $n^{-1/6}$ as $\tsig_j^a$ approaches the transition point $m_{2c}^{-1}(\lambda_+).$ This shows that the non-outlier eigenvectors near the edge are biased in the direction of $\vb_j^a$ provided that $\tsig_j^ a$ is near the transition point $m_{2c}^{-1}(\lambda_+).$ In particular, for $|\tsig_j^a+m_{2c}^{-1}(\lambda_+)| \le n^{-1/3}$, we have that   
\begin{equation}\label{eq_subcritical}
|\langle \vb_j^a, \wt{\bxi}_{\alpha(i)} \rangle|^2 \prec  n^{-1}|\tsig_j^a+m_{2c}^{-1}(\lambda_+)|^{-2}.
\end{equation}

In the literature, the $\tsig_j^a$ in this case is called a weak spike in statistics \cite{KZT} or subcritical spike in probability \cite{principal}. Thus (\ref{eq_subcritical}) shows that the non-outlier eigenvectors still retain information about the weak spikes of $\wt A$ in contrast to the non-outlier eigenvalues as seen from (\ref{eq_nonspike}). 
\end{example}

The Theorems \ref{thm_outlier}, \ref{thm_eigenvaluesticking}, \ref{thm_eveout} and \ref{thm_noneve} give the first order limits and convergent rates of the principal eigenvalues and eigenvectors of  $\widetilde{\mathcal{Q}}_{1}$. The second order asymptotics of the outlier eigenvalues and eigenvectors will be studied in another paper.

Note that for separable covariance matrices, $\wt{A}^{1/2}X \wt{B}^{1/2}$ and $\wt{B}^{1/2}X^*\wt{A}^{1/2}$ take exactly the same form. Hence by exchanging the roles of $(\wt{A},X)$ and $(\wt{B},X^*)$, one can immediately obtain from Theorems \ref{thm_eveout} and \ref{thm_noneve} the similar results for the eigenvectors $\wt{\bm\zeta}_k$ of $\wt Q_2$.  For reader's convenience, we state them in the following two theorems. Denote 
$$ \cal P'_S:=\sum_{k\in S} \wt{\bzeta}_k\wt{\bzeta}_k^*, \quad \text{ for } S\subset \mathcal O^+.$$

\begin{theorem} \label{thm_rightout}
Suppose $X$ has bounded support $\phi_n$ such that $ n^{-{1}/{2}} \leq \phi_n \leq n^{- c_\phi} $ for some constant $c_\phi>0$. Suppose that Assumptions \ref{assm_big1}, \ref{ass:unper} and \ref{ass:spike} hold. Fix any $S\subset \mathcal O^+$, we define the following deterministic positive quadratic form
\begin{equation*}
\langle \wb, \cal Z'_S\wb \rangle := \sum_{\mu:\beta(\mu)\in S}\frac{|w_\mu|^2}{\wt{\sigma}_\mu^b} \frac{g_{1c}'(-(\wt{\sigma}_\mu^b)^{-1})}{g_{1c}(-(\wt{\sigma}_\mu^b)^{-1})} , \quad  \text{for } \ \bw\in \mathbb C^n, \ \ w_\mu : =\langle \bv_\mu^b, \wb\rangle.
\end{equation*}
Then for any deterministic vector $\wb \in \mathbb C^n$, we have that
\begin{equation*}
\begin{split}
&\left| \langle \wb, \cal P'_S\wb\rangle- \langle \wb, \cal Z'_S\wb \rangle \right| \prec  \sum_{1\le \mu\le s:\beta(\mu)\in S}|w_\mu|^2  \psi_{2}(\wt\sigma_\mu^b)  \\
&+\sum_{1\le \mu\le s:\beta(\mu)\notin S}|w_\mu|^2  \frac{\phi_n^2}{\delta_{\beta(\mu)}(S)} + \sum_{\mu=1}^n {|w_\mu|^2} \left(\frac{\psi_{2}^2(\wt\sigma_\mu^b)\Delta_2^2(\widetilde{\sigma}_\mu^b) }{\delta^2_{\beta(\mu)}(S)} + \frac{\kappa_\mu}{n^{1/2}}\right) \\
&+ \langle \bw , \cal Z'_S\bw \rangle^{1/2}\left[\sum_{1\le \mu\le s:\beta(\mu)\notin S} \frac{|w_\mu|^2 \phi_n^2}{\delta_{\beta(\mu)}(S)} + \sum_{1\le \mu \le n:\beta(\mu)\notin S} {|w_\mu|^2} \left(\frac{\psi_{2}^2(\wt\sigma_\mu^b)\Delta_2^2(\widetilde{\sigma}_\mu^b) }{\delta^2_{\beta(\mu)}(S)} + \frac{\kappa_\mu}{n^{1/2}}\right)\right]^{1/2} ,
\end{split}
\end{equation*}
where we denote
 $$\psi_{2}(\wt\sigma_\mu^b):= \phi_n +  n^{-1/2}\Delta_2^{-1}(\widetilde{\sigma}_\mu^b) .$$
 If we have (a) \eqref{assm_3rdmoment} holds, or (b) either $A$ or $B$ is diagonal, then the above estimate holds without the $ n^{-1/2}\kappa_\mu$ terms.
\end{theorem}

\begin{theorem}\label{thm_rightbulk}
Suppose $X$ has bounded support $\phi_n$ such that $ n^{-{1}/{2}} \leq \phi_n \leq n^{- c_\phi} $ for some constant $c_\phi>0$. Suppose that Assumptions \ref{assm_big1}, \ref{ass:unper} and \ref{ass:spike} hold. Fix any sufficiently small constant $\tau>0$. For  $\beta(\mu)\notin \mathcal O^+$, $ \mu \leq {\tau}n$ and any deterministic vector $\wb \in \mathbb{C}^n$, we have 
\begin{equation*}
|\langle \wb, \wt\bzeta_{\beta(\mu)} \rangle|^2 \prec \sum_{\nu=1}^n |w_\nu|^2 \frac{n^{-1}+  \eta_\mu \sqrt{\kappa_\mu} +\phi_n^3}{|\tsig_\nu^b+m_{1c}^{-1}(\lambda_+)|^2+\phi_n^2 +\kappa_\mu}.
\end{equation*}
 If we have (a) \eqref{assm_3rdmoment} holds, or (b) either $A$ or $B$ is diagonal, then we have the stronger estimate
\begin{equation*}
|\langle \wb, \wt\bzeta_{\beta(\mu)} \rangle|^2 \prec  \sum_{\nu=1}^n |w_\nu|^2 \frac{n^{-1}+\phi_n^3}{ |\tsig_\nu^b+m_{1c}^{-1}(\lambda_+)|^2+\phi_n^2 + \kappa_\mu }.
\end{equation*}
\end{theorem}


Using a simple cutoff argument, it is easy to obtain the following corollary under certain moment assumptions. Since we do not assume the entries of $X$ are identically distributed, the means and variances of the truncated entries may be different. This is why we assume the slightly more general conditions in \eqref{entry_assm1}.

\begin{corollary}\label{main_cor}
Assume that $X=(x_{ij})$ is a real $p\times n$ matrix, whose entries are independent random variables that satisfy \eqref{eq_12moment} and  
\be\label{condition_4e} 
\max_{i,j}\mathbb{E}  |\sqrt{n} x_{ij} | ^{a}  \le C,  
\ee 
for some constants $C>0$ and $a>4$. Suppose $A$, $B$, $\wt A$, $\wt B$ and $d_n$ satisfy Assumptions \ref{assm_big1} and \ref{ass:unper}. Then Theorems \ref{thm_outlier}, \ref{thm_eigenvaluesticking}, \ref{thm_eveout}, \ref{thm_noneve}, \ref{thm_rightout} and \ref{thm_rightbulk} hold for $\phi_n= n^{2/a-1/2}$ on an event with probability $1-\oo(1)$.
\end{corollary}
Its proof is given in Section \ref{sec_supp_right} of the supplement. 

\begin{remark}
We remark that one can take $r=0$ or $s=0$ (i.e. either $\wt A$ or $\wt B$ has no spikes) in the statements of our main results, although some results will become trivial null results. As an example, we consider the case where $r \geq 1$ and $s=0$. 
In this case, the outlier eigenvalues only come from $\widetilde{A}.$ Consequently, in Definition \ref{defn_relabelling}, we have that $\mathcal O:= \{\al(i): 1\le i \le r\}$ and $\mathcal O^+:= \{\al(i): 1\le i \le r^+\}.$ Then Theorem \ref{thm_outlier} still holds, although \eqref{eq_spike2} becomes a null result since there is no $\mu$ such that $1\le \mu \le 0$;  Theorem \ref{thm_eigenvaluesticking} holds true with $s^+=0$ and $\alpha_+:= \min_i \left|\tsig_i^a+m_{2c}^{-1}(\lambda_+)\right|$; 
Theorems \ref{thm_eveout}, \ref{thm_noneve}, \ref{thm_rightout} and \ref{thm_rightbulk} still hold for the left and right singular vectors, although Theorem \ref{thm_rightout} actually can be  derived from  Theorem  \ref{thm_rightbulk} since there is no outlier coming from $\wt B$.


If $r=s=0,$ $\widetilde{Q}_1$ reduces to the non-spiked version $\mathcal{Q}_1=A^{1/2}XBX^*A^{1/2}.$ All of our main results are still valid, but better estimates actually hold  in this case as given in \cite{yang2018}, which studied non-spiked separable covariance matrices. Some of these results are also stated in Theorem \ref{thm_largerigidity} and Lemma \ref{delocal_rigidity} of our supplement \cite{dysupple}. 
\end{remark}

\subsection{Strategy for the proof} We conclude this section by describing briefly the main ideas and mathematical tools used in our proof. Using a linearization method (c.f. \eqref{linearize_block} of \cite{dysupple}), we can show that the outlier eigenvalues satisfy a master equation in terms of the resolvents in (\ref{def_green}) (c.f. Lemma \ref{lem_pertubation} of \cite{dysupple}). Moreover, the resolvents appear in the forms $(V_o^a)^*\cal G_1 V_o^a$ and  $(V_o^b)^*\cal G_2 V_o^b$, where we recall the notations in \eqref{AOBO}. These functionals of resolvents can be estimated using the anisotropic local law in \cite{yang2018}, which shows that they are close to certain deterministic matrices up to some small errors (c.f. Theorem \ref{LEM_SMALL} of \cite{dysupple}). By replacing  $(V_o^a)^*\cal G_1 V_o^a$ and  $(V_o^b)^*\cal G_2 V_o^b$ with their deterministic equivalents, we can solve the master equation to get the asymptotic locations $\theta_1(\wt\sigma_i^a)$ and $\theta_2(\wt\sigma_\mu^b)$ of the outliers. To obtain the convergence rates in Theorems \ref{thm_outlier} and \ref{thm_eigenvaluesticking}, we need to control the errors using the anisotropic local law and  a three-step proof strategy developed in \cite{KY2013}, which is summarized at the beginning of Section \ref{sec:ev} in supplement \cite{dysupple}.  

Once we know the asymptotic locations of the outliers, we can use Cauchy's integral formula to study the eigenvectors. For example, suppose the largest outlier $\wt\lambda_1$ is well separated from all the other eigenvalues. Then using the Cauchy's integral formula, we get
$$|\langle \bv,\wt{\bm \xi}_1\rangle|^2 = -\frac{1}{2\pi \ii} \oint_{\Gamma}\bv^* \sum_{k = 1}^{p} \frac{\wt{\bm \xi}_k(i) \wt{\bm \xi}_k^*(j)}{\wt\lambda_k-z} \bv \dd z= -\frac{1}{2\pi \ii} \oint_{\Gamma} \sum_{k = 1}^{p} \bv^* \wt {\cal G}_1(z) \bv \dd z$$
where $\Gamma$ is a small contour enclosing $\wt\lambda_1$ only. For a more general integral representation of $\langle \bv, \cal P_S\bv\rangle$, we refer the reader to  (\ref{eq_greenrepresent}) of \cite{dysupple}. 
Using the anisotropic local law, we can obtain the convergence limits and rates in Theorem \ref{thm_eveout}. The proof of Theorem \ref{thm_noneve} relies on the simple bound
$$ |\langle \bv,\wt{\bm \xi}_k\rangle|^2 \le \eta \cdot \left(\bv^* \sum_{k = 1}^{p} \frac{\eta \wt{\bm \xi}_k(i) \wt{\bm \xi}_k^*(j)}{|\wt\lambda_k-z_k|^2} \bv\right)  =\eta \im  \bv^* \wt {\cal G}_1(z_k) \bv,$$
where we take $z_k= \wt\lambda_k + \ii\eta$. Again we will use the anisotropic local law to establish the delocalization bounds. 

 \section{Statistical estimation for spiked separable covariance matrices}\label{sec:statapp}
 
 In this section, we consider the estimation of $\wt A$ and $\wt B$ from the data matrix $\wt{A}^{1/2}X \wt{B}^{1/2}.$ In particular, we address two fundamental issues:  
\begin{itemize}
\item[(1)]  estimating the number of spikes in $\widetilde{A}$ and $\widetilde{B};$  

\item[(2)]  adaptive optimal shrinkage of the eigenvalues of $\widetilde{A}$ and $\widetilde{B}.$ 
\end{itemize}

To ease our discussion, till the end of this section, we will replace Assumption \ref{ass:spike} with the following stronger \emph{super-critical}  condition. It is commonly used in the statistical literature, for instance \cite{bgn2012, DO2019, donoho2018, RRN2014}.   

\begin{assumption}\label{assum_supercritical} For some fixed constant $\tau>0,$  we assume that  there are $r$ spikes for $\widetilde{A}$ and $s$ spikes for $\widetilde{B}$, which satisfy
\begin{equation*} 
\wt{\sigma}^a_{i}+{m_{2c}^{-1}(\lambda_+)} > \tau, \ \ 1 \leq i \leq r, \quad 
\text{ and }
\quad \wt{\sigma}^b_{\mu}+m_{1c}^{-1}(\lambda_{+}) > \tau,  \  \ 1 \leq \mu \leq s.
\end{equation*}
\end{assumption}
For simplicity of presentation, we will also assume the following non-overlapping condition.
\begin{assumption}\label{big_gap}
Recall (\ref{eq_nu}) and (\ref{eq_nu2}).  For some fixed constant $\tau>0,$ we assume that
\begin{equation*}
 \min_{1\le j \le r}\delta_{\al(i),\al(j)}^a \wedge  \min_{1\le \mu \le s}\delta_{\al(i),\beta(\mu)}^a \ge \tau,  \quad 1\le i \le r,
\end{equation*}
and 
\begin{equation*}
 \min_{1\le \nu \le s}\delta_{\beta(\mu),\beta(\nu)}^b \wedge  \min_{1\le i \le r}\delta_{\beta(\mu),\al(i)}^b \ge \tau,  \quad 1\le \mu \le s.
\end{equation*}
\end{assumption}

 \subsection{ Estimating the number of spikes}\label{sec:estspike}
 
The number of spikes has important meaning in practice. For instance, it represents the number of factors in factor model \cite{factor1, factor2} and number of signals in signal processing \cite{signaldec}. Such a problem has been studied for spiked covariance matrix, see e.g. \cite{PY14}.  In this section, we extend the discussion to the more general spiked separable model (\ref{eq_sepamodel}). 
 
Different from the spiked covariance matrix model, we have two sources of spikes from either $\widetilde A$ or $\widetilde B$. For spiked covariance matrices, the statistic only involves sample eigenvalues. However, as we have seen from Theorem \ref{thm_outlier}, the sample eigenvalues only contain information of the total number of spikes, i.e. $r+s$.  
One way to deal with this issue is to use the information from the sample eigenvectors and apply Theorem \ref{thm_eveout}. In Figure \ref{fig_distinguish}, we use a numerical simulation to illustrate how the eigenvectors can help us to gather information of separable covariance matrices. We consider two different settings:
 \begin{equation}
 \widetilde{\Sigma}^a=\text{diag}(5, 1, \cdots, 1), \quad  \widetilde{\Sigma}^b=\text{diag}(5, 1, \cdots, 1)  \tag{Case I},
 \end{equation}
 and
  \begin{equation}
 \widetilde{\Sigma}^a=\text{diag}(3, 2,1,  \cdots, 1), \   \widetilde{\Sigma}^b=\text{diag}(1, 1, \cdots, 1) \tag{Case II}.
 \end{equation}
Figure \ref{fig_distinguish} (a) shows that there are two spikes in both cases. However, from Figure \ref{fig_distinguish} (b) and Figure \ref{fig_distinguish} (c), we can see that there are two parts of spikes in Case I, but only one part in Case II as expected.   It shows the necessity to take into consideration the information from the eigenvectors. Here we take $p=150, n=200.$ 

\begin{figure}[ht]
 \captionsetup{width=0.8\linewidth}
 \begin{center}
\begin{subfigure}{0.5\textwidth}
\includegraphics[width=6.5cm, height=5cm]{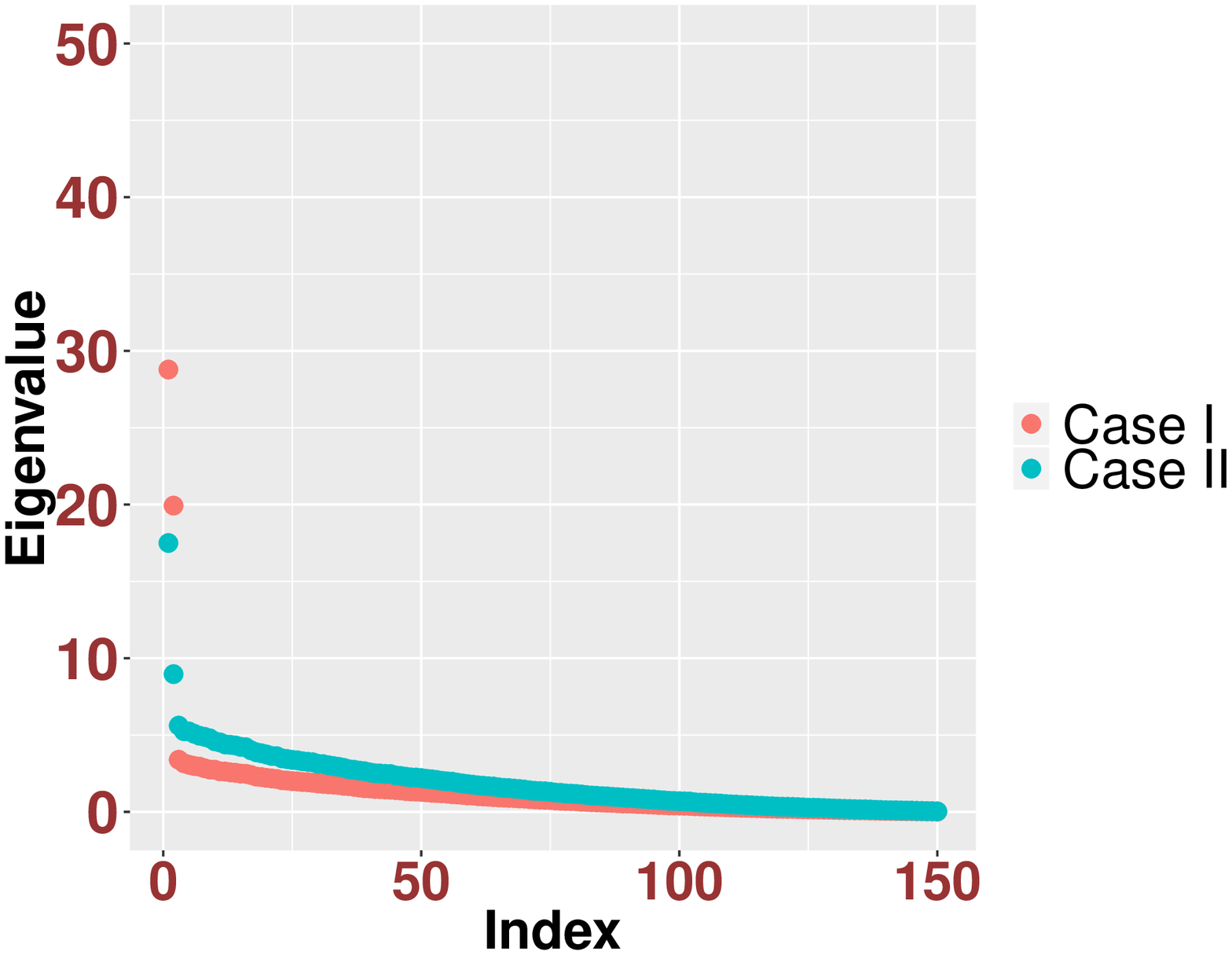}
\subcaption{Eigenvalues.  } \label{fig_evdis}
\end{subfigure}\\
\begin{subfigure}{0.4\textwidth}
\includegraphics[width=6.4cm,height=5cm]{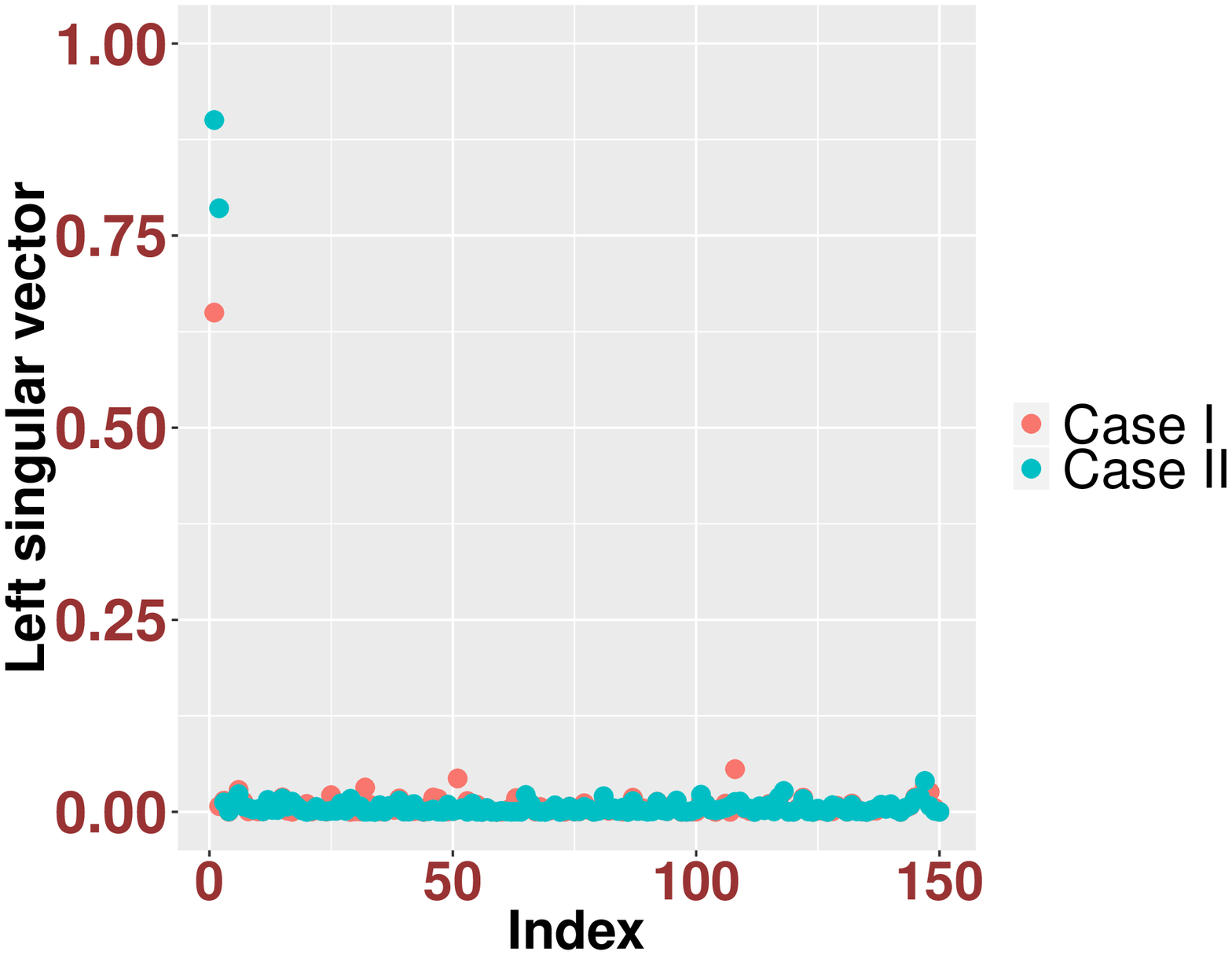}
\subcaption{
$|\langle \bv_i^a, \wt\bxi_{\al(i)} \rangle|^2$  }\label{fig_leve}
\end{subfigure}
\begin{subfigure}{0.4\textwidth}
\includegraphics[width=6.4cm,height=5cm]{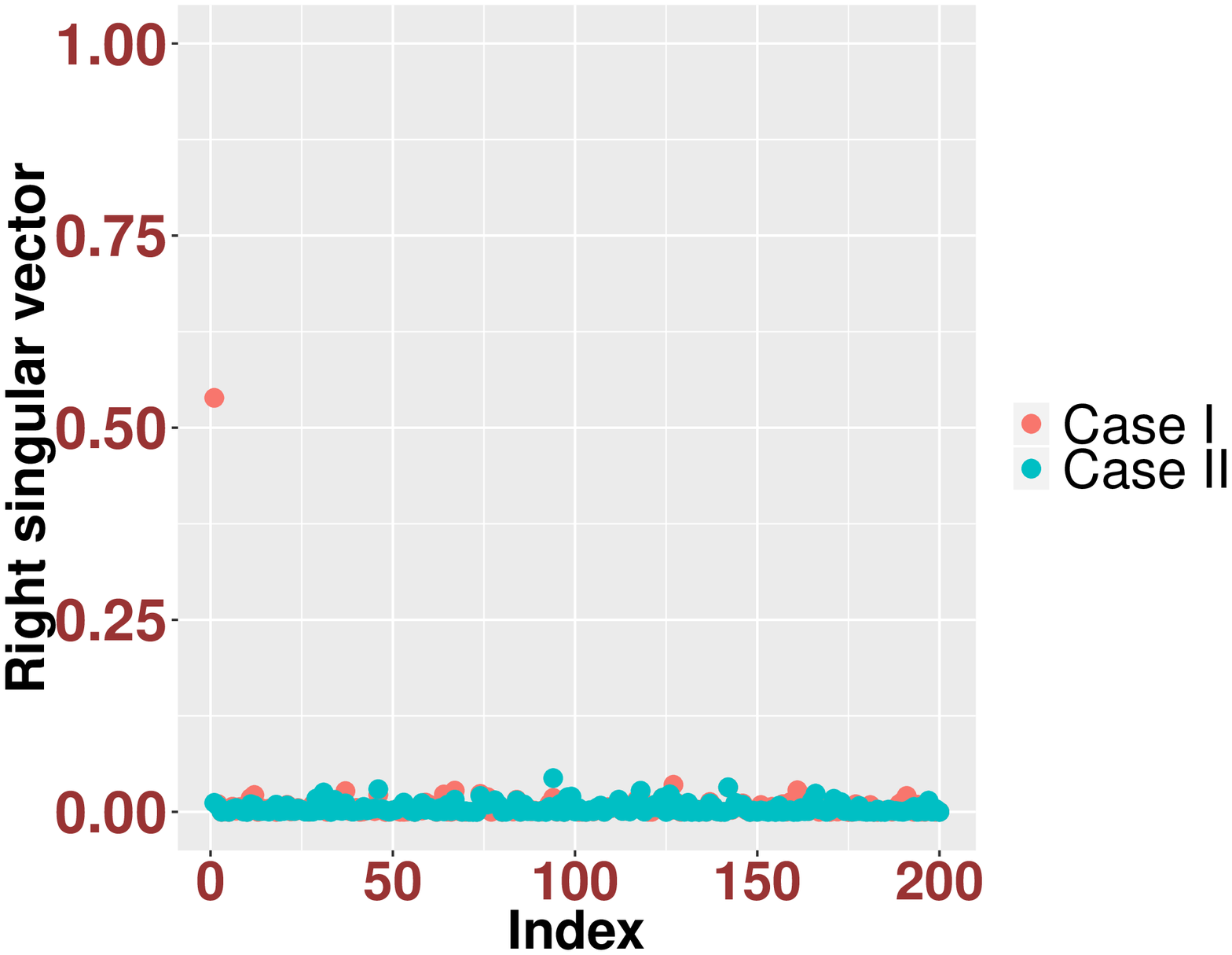}
\subcaption{
 $|\langle \bv_\mu^b, \wt{\bzeta}_{\beta(\mu)} \rangle|^2$ }\label{fig_reve}
\end{subfigure}
\caption[]{Eigenvalues and eigenvectors for spiked separable covariance matrices.} \label{fig_distinguish}
\end{center}
\end{figure} 

 In the following discussion, we assume that the population eigenvectors of $\widetilde A$ and $\widetilde B$ are known. For the more general case where such information is unavailable, we will study it somewhere else (see also Remark \ref{cone cond}).

We provide our statistic and start with a heuristic discussion. Under Assumptions \ref{assum_supercritical} and \ref{big_gap}, we get from Theorems \ref{thm_outlier}, \ref{thm_eveout} and \ref{thm_noneve} that 
\begin{align*}
& \wt\lambda_{\al(i)}= \theta_1(\wt\sigma_i^a) + \OO_\prec(\phi_n),
\end{align*}
and for $1\le i \le r$, 
\begin{align*}
& |\langle \bv_i^a, \wt\bxi_{k}\rangle|^2 =\mathbf 1(k=\al(i)) \left[\frac{1}{\wt\sigma_i^a}\frac{g_{2c}'(-(\wt{\sigma}_i^a)^{-1})}{\theta_{1}(\wt{\sigma}_i^a)} + \OO_\prec(\phi_n)\right] + \OO_\prec(\phi_n^2).
\end{align*}
\nc Hence, if all the spiked eigenvalues are well-separated, the ratio between $\wt\lambda_{\al(i)}$ and $\wt\lambda_{\al(i+1)}$ are strictly greater than 1. However, for the non-outlier eigenvalues, these ratios will converge to 1 at a rate $\OO_\prec(n^{-2/3}+\phi_n^2)$ by Theorem \ref{thm_eigenvaluesticking} and eigenvalue rigidity, Theorem \ref{thm_largerigidity} in the supplement. Moreover, the (cosine of) the angle $|\langle \bv_i^a, \wt\bxi_{k}\rangle|$ is of order $\OO_\prec(\phi_n)$ except when $k=\al(i)$, in which case we have that $|\langle \bv_i^a, \wt\bxi_{k}\rangle|$ is larger than a constant. Therefore, the ratios between consecutive eigenvalues and the angles will be used as our statistics. 

Formally,  for a given threshold $\omega>0$ and a properly chosen constant $c>0$, we define the statistic $q$ by
\begin{equation}\label{eq_defnqest}
q \equiv q(\omega):=\operatorname*{arg\,min}_ {1 \leq i \leq c(p\wedge n)} \left\{ \frac{\wt\lambda_{i+1}}{\wt\lambda_{i+2}}-1 \leq \omega \right\} , 
\end{equation}
and $q_{a,b} \equiv q_{a,b}(\omega)$ by 
\begin{align*}
& q_a(\omega):=\operatorname*{arg\,min}_ {1 \leq i \leq c(p\wedge n)} \left\{ \max_{1\le k\le c(p\wedge n)}\left| \langle \bv_{i+1}^a, \widetilde{\bxi}_k \rangle \right|^2  \leq \omega \right\}, \\
& q_b(\omega):=\operatorname*{arg\,min}_ {1 \leq \mu \leq c(p\wedge n)} \left\{ \max_{1\le \nu\le c(p\wedge n)}\left| \langle \bv_{\mu+1}^b, \widetilde{\bzeta}_\nu \rangle \right|^2  \leq \omega \right\}.
\end{align*}
As discussed above, $q$ is used to estimate the total number of spikes, whereas $q_a$ and $q_b$ are used to estimate the number of spikes for $\wt A$ and $\widetilde B$, respectively. With Theorems \ref{thm_outlier},  \ref{thm_eigenvaluesticking}, \ref{thm_eveout}, \ref{thm_noneve}, \ref{thm_rightout} and \ref{thm_rightbulk}, it is easy to show that they are consistent estimators for carefully chosen threshold $\omega$.  
Denote the event $\Omega \equiv \Omega(\omega)$ by
\begin{equation*}
\Omega:=\{q=r+s, q_a=r, q_b=s\}.
\end{equation*} 
\begin{theorem}\label{thm_estispike} 
Suppose $X$ has bounded support $\phi_n$ such that $ n^{-{1}/{2}} \leq \phi_n \leq n^{- c_\phi} $ for some constant $c_\phi>0$. Suppose that the Assumptions \ref{assm_big1}, \ref{ass:unper}, \ref{assum_supercritical} and \ref{big_gap} hold. Then if $\omega$ satisfies that for some constant $\e>0$,
\begin{equation}\label{eq_omegaassu}
{  \omega \rightarrow 0, \quad \frac{\omega}{ n^{ \e} (n^{-2/3}+\phi_n^2)} \rightarrow \infty,}
\end{equation}
then we have that $\Omega$ holds with high probability for large enough $n$. 
\end{theorem}
\begin{proof} This theorem is an easy consequence of Theorems \ref{thm_outlier},  \ref{thm_eigenvaluesticking}, \ref{thm_eveout}, \ref{thm_noneve}, \ref{thm_rightout} and \ref{thm_rightbulk}. 
\end{proof}

For the practical implementation, we employ a resampling procedure to choose the threshold $\omega$ for the statistic $q$ using a reference matrix.  Such procedure has been used in estimating the number of spikes for spiked covariance matrix \cite{PY14}.  We consider the case where the entries of $X$ have finite $(12+\e)$-th moments, such that we can take $\phi_n \ll n^{-1/3}$ by Corollary \ref{main_cor}. Then by Theorem \ref{thm_eigenvaluesticking}, the extreme non-outlier eigenvalues of $\wt {\cal Q}_1$ have the same limiting distribution as those of the non-spiked matrix $\cal Q_1$,  
which, by the edge universality result  \cite[Theorem 2.7]{yang2018}, fluctuate on the scale $n^{-2/3}$.  Since the edge eigenvalues of Wishart matrix 
satisfy the Tracy-Widom distribution up to an $n^{-2/3}$ rescaling, the edge eigenvalue ratios of $\mathcal Q_{1}$ should be close to those of the Wishart matrix. More precisely, we can use Wishart matrix as the reference matrix and take the following steps to choose $\omega$.

\vspace{5pt}

\noindent{\bf Step (i)}: Generate a sequence of $N$, say $N=10^4$, {$p \times p$} Wishart matrices $X_iX_i^*$ and  the associated  sequence of statistics $\{\mathcal{T}_i\}_{i=1}^N,$
\begin{equation*}
\mathcal{T}_i:=\max_{1\leq k \leq c (p\wedge n)} 
\left\{ {\lambda^{(i)}_k}/{\lambda^{(i)}_{k+1}} \right\},
\end{equation*}
where $\{\lambda_k^{(i)}\}_{k=1}^{p \wedge n}$ are the eigenvalues of $X_iX_i^*$ arranged in descending order.

\vspace{5pt}

\noindent{\bf Step (ii)}: Given the nominal level $\e$ (say $\e=0.05$), we choose $\omega$ such that 
\begin{equation*}
\frac{\#\{\mathcal{T}_i \leq 1+ \omega \}}{N}\ge 1-\e.
\end{equation*}



 In Figure \ref{fig_missrate}, we consider the estimation of the number of spikes of $\widetilde B$ and analyze the frequency (over $10^4$ simulations) of misestimation as a function of the value of $x$ under different combinations of $p$ and $n.$ We make use of the statistic $q_b$ and choose  $\omega$ according to the above steps (i) and (ii). Specifically, we report the frequency of misestimation of the setting
\begin{equation*}
\widetilde A=\text{diag}(4,1,\cdots, 1), \quad \widetilde B=\text{diag}(x+2, x,1,\cdots,1), \quad x \geq 1. 
\end{equation*}
We can see that our estimator performs quite well for $x$ above some threshold. 
\begin{figure}[ht]
 \captionsetup{width=1\linewidth}
 \begin{center}
\includegraphics[width=8cm,height=6cm]{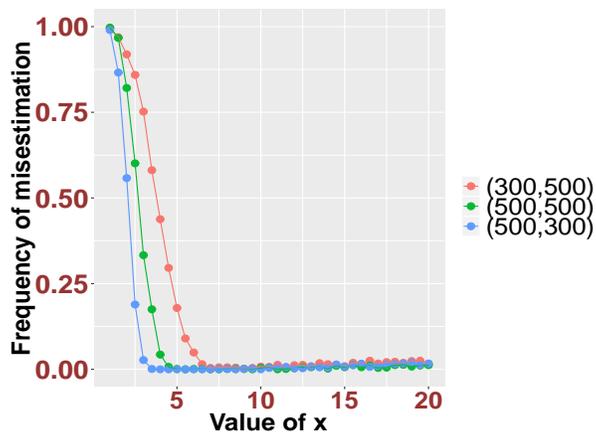} 
\caption{Frequency of misestimation for different values of $x$. }\label{fig_missrate}
\end{center}
\end{figure}


{ Before concluding this subsection, we provide some insights on the choices of $\omega.$ In general, the choice of $\omega$ should depend on both $A$ and $B,$ denoted as $\omega_{A,B}$.  Even though in the above procedure we have used $\omega_{I_p, I_n},$ such a simple choice is usually sufficient for our purpose. In Section  \ref{simu_threshold} of the supplement \cite{dysupple}, we show by simulations to verify our findings.  On one hand, as illustrated in Figure \ref{fig_thresholddiff},  the difference $|\omega_{I_p, I_n}-\omega_{A,B}|$ is already very small for $n=200$ and the difference decreases when  $n$ increases. Moreover,  empirically  we see from the simulations that $|\omega_{A,B}-\omega_{I_p,I_n}| \leq 0.008$ when $n \geq 300$ for a variety  of $d_n.$ On the other hand, for different choices of $A$ and $B,$ when the spiked eigenvalues are reasonably large, the frequency of misestimation will not be influenced if we simply use the threshold $\omega_{I_p, I_n}$. In Section \ref{simu_threshold} of our supplement \cite{dysupple}, we record such simulation results in Figure \ref{fig_estthresdiff}.

For smaller spikes, an accurate estimation of $A$ and $B$ can lead to more prudential choices of $\omega_{A,B}.$  As discussed in Remark \ref{remark_generalestimation},  there does not exist any method to estimate general $A$ and $B.$ Even though the construction of such estimators are out of the scope of this paper, when either $A$ or $B$ is identity, it reduces to estimating the spectrum of a sample covariance matrix. In this case, we can use many state-of-the-art algorithms to estimate the spectrum, for instance, \cite{elkaroui2008, kong2017, LW2015}. 
In \cite[Section \ref{simu_threshold}]{dysupple},  assuming that $B=I_n,$ we first use the numerical method as described in \cite{LW2017} to find an estimator of $A,$ denoted as $\widehat{A},$ and then use $\omega_{\widehat{A}, I_n}$ as our threshold. The results are recorded in Tables \ref{table_smallerspikes1}--\ref{table_smallerspikes3}. We see that it will reduce the frequency of misestimation for smaller spikes. 
}

\subsection{Adaptive optimal shrinkage for spiked separable covariance matrices}

In most of the real applications, we have no a priori information on the true eigenvectors of $\widetilde A$ or $\widetilde B$. Then the natural choice for us is to use the sample eigenvectors $\{\wt\bxi_i\}_{1\le i \le p}$  and $\{\wt{\bzeta}_\mu\}_{1\le \mu \le n}.$ Consider similar setting as in Johnstone's spiked  covariance model \cite{donoho2018, spikedmodel} with $A=I_p$ and $B=I_n$. Suppose we know the number of spikes $r+s$. Then we want to estimate 
$$\widetilde A= \sum_{i=1}^r  \wt\sigma_i^a \bv_i^a(\bv_i^a)^*+\sum_{i=r+1}^p \bv_i^a(\bv_i^a)^*,\quad \widetilde B= \sum_{\mu=1}^s  \wt\sigma_\mu^b \bv_i^a(\bv_i^a)^*+\sum_{\mu=s+1}^n \bv_\mu^b(\bv_\mu^b)^*,$$ 
using the estimators
\begin{equation}\label{eq_rieop}
\begin{split}
& \widehat A= \sum_{i=1}^{r+s}  \varrho_a (\wt\lambda_i) \wt{\bxi}_i \wt{\bxi}_i^*+\sum_{i=r+s+1}^p \wt{\bxi}_i \wt{\bxi}_i^*,   \\
& \widehat B= \sum_{\mu=1}^{r+s}  \varrho_b (\wt\lambda_i) \wt{\bzeta}_\mu \wt{\bzeta}_\mu^*+\sum_{i=r+s+1}^n \wt{\bzeta}_\mu \wt{\bzeta}_\mu^*,
\end{split}
\end{equation}  
where $\varrho^a(\cdot)$ and $\varrho^b(\cdot)$ are some shrinkage functions characterized by the minimizers  of certain loss functions:
 \begin{equation*}
\wh A:= \operatorname*{arg\,min}_{\mathcal A} \mathcal{L}_a(\mathcal A, \widetilde A), \quad  \wh B:=\operatorname*{arg\,min}_{\mathcal B}\mathcal{L}_b(\mathcal B, \widetilde B).
\end{equation*}   
In \cite{donoho2018}, the authors consider this problem for spiked covariance matrices for a variety of loss functions assuming that $r,s$ are known. In this section, we study this problem for spiked separable covariance matrices using the Frobenius norm as the loss functional. We will also prove the optimal convergent rate for such estimators. The other loss functions as discussed in \cite{donoho2018} can be studied in a similar way.  

We shall only consider $\varrho_a(\wt\lambda_i)$, while $\varrho_b(\wt\lambda_i)$ can be handled with the same argument by symmetry. 
We calculate that
\begin{equation}\label{eq_normerror}
\norm{\widehat A-\widetilde A}_F^2=\norm{T}_F^2, \quad T:=\sum_{i=1}^{r+s} \left[ (\varrho(\wt\lambda_i)-1) \wt\bxi_i \wt\bxi_i^*-(\wt{\sigma}_i^a-1)\bv_i^a (\bv_i^a)^* \right].
\end{equation}
We expand $T$ to get 
\begin{align*}
\norm{T}_F^2 =&\sum_{i=1}^{r+s} \left[ (\varrho_a(\wt\lambda_i)-1)^2+(\tsig_i^a-1)^2 - 2|\langle \bv_i^a, \wt\bxi_i \rangle|^2(\varrho_a(\wt\lambda_i)-1)(\tsig_i^a-1) \right] \\
& -2 \sum_{ i \ne j }^{r+s} (\varrho_a(\wt\lambda_i)-1)(\wt{\sigma}_j^a-1) |\langle \wt{\bv}_j^a,\wt\bxi_i  \rangle|^2.
\end{align*}
Therefore,  (\ref{eq_normerror}) is minimized if 
\begin{equation*}
\varrho_a(\wt\lambda_i)= 1+\sum_{ j=1}^{r+s} (\wt{\sigma}_j^a-1) |\langle \bv_j^a,\wt\bxi_i  \rangle|^2.
\end{equation*}
Under Assumptions \ref{assum_supercritical} and \ref{big_gap},  by Theorems \ref{thm_eveout} and \ref{thm_noneve} we find that for $\wt\sigma_k^a:= d_k^a +1$,
\begin{equation*}
\varrho_a(\wt\lambda_i)=\mathbf 1(i=\al(k) \text{ for some } k=1,\cdots, r)\frac{d_k^a}{\wt\sigma_k^a}\frac{g'_{2c}(-(\wt\sigma_k^a)^{-1})}{g_{2c}(-(\wt\sigma_k^a)^{-1})}+\OO_{\prec}(\phi_n).
\end{equation*} 
 
Under the setting with $A=I_p$ and $B=I_n$, $m_{2c}(z)$ is the Stieltjes transform of the standard Marchenko-Pastur (MP) law. Then it is known that $g_{2c}$ is given by \cite[Section 2.2]{Knowles2017}
\begin{equation*}
g_{2c}(x)=-\frac{1}{x}+d_n \frac{1}{x+1},
\end{equation*}
where we recall that $d_n=p/n$. Therefore, we can calculate that 
\begin{equation*}
\varrho_a(\wt\lambda_i)=\frac{(d_k^a)^2 - d_n}{d_k^a+d_n}+\OO_{\prec}(\phi_n), \quad i=\al(k).
\end{equation*}
For $d_k^a$, we can use Theorem \ref{thm_outlier} to get that $d_k^a = -m_{2c}^{-1}(\lambda_i) - 1 + \OO_\prec(\phi_n)$ for $i=\al(k)$. 
We have the following explicit form for $m_{2c}$ (see e.g. (4.10) of \cite{ding20171}):
\begin{equation*}
m_{2c}(x)=\frac{d_n-1-x+\sqrt{(x-\lambda_+)(x-\lambda_-)}}{2 x}, \quad \lambda_{\pm}=(1\pm d^{1/2}_n)^2,
\end{equation*} 
when $x>\lambda_+$. Thus we can define the following shrinkage function
\begin{equation*}
\widehat{\varrho}_a(\wt\lambda_{i})=\mathbf 1(i=\al(k) \text{ for } k\in \{1,\cdots, r\})\frac{(\widehat{d}_k^a)^2-d_n}{\widehat{d}_k^a + d_n},\quad \wh d_k^a = -m_{2c}^{-1}(\wt\lambda_{\al(k)}) - 1 ,
\end{equation*} 
which satisfies that
\begin{equation*}
\varrho_a(\wt\lambda_i)=\widehat{\varrho}_a(\wt\lambda_i)+\OO_{\prec}(\phi_n). 
\end{equation*}

\begin{remark}\label{cone cond}
Note that the definition of the shrinkage function depends on  a priori knowledge of the indices of the outliers caused by the spikes of $\wt A$, which may not be available in applications. Moreover, the methods in Section \ref{sec:estspike} cannot be used since we have no information on the eigenvectors of $\wt A$ and $\wt B$. However, this kind of information is still possible to obtain by exploring the ``cone condition" in Example \ref{exam_nondege}, that is, we can project the left and right outlier-singular vectors onto some suitably chosen directions and take average over many samples. To have a rigorous theory, it is necessary to establish the second order asymptotics of the outlier eigenvectors. Both of these topics will be explored elsewhere. 
\end{remark}

 We then present the results of some Monte-Carlo simulations designed to illustrate the finite-sample properties of the shrinkage estimator $\widehat A$. We study the improvement of $\widehat A$ over the separable covariance matrix $\widetilde{\mathcal{Q}}_1$, which also uses the sample eigenvectors. Denote $\overline{A}$ as in (\ref{eq_rieop}) but with $\varrho_a(\wt{\lambda}_i)$ replaced by $\widehat{\varrho}_a(\wt{\lambda}_i).$ In Figure \ref{fig_PRIAL}, we report the Percentage Relative Improvement in Average Loss (PRIAL) \cite[Section 1.3]{Ledoit2011} for $\overline{A}$: 
\begin{equation}\label{eq_prial}
\text{PRIAL}:=100 \times \left\{1-\frac{\mathbb{E} \norm{\overline{A}-\widehat{A}}_F^2}{\mathbb{E} \norm{\widetilde{\mathcal{Q}}_1-\widehat{A}}_F^2} \right\} \%,
\end{equation}
where $\mathbb{E}(\cdot)$ denotes the average over $10^4$ Monte-Carlo simulations. We can see that our estimators perform better than sample separable covariance matrix even for ``not so large" matrix dimensions. 

\begin{figure}[ht]
 \captionsetup{width=1\linewidth}
 \begin{center}
\includegraphics[width=8cm,height=6cm]{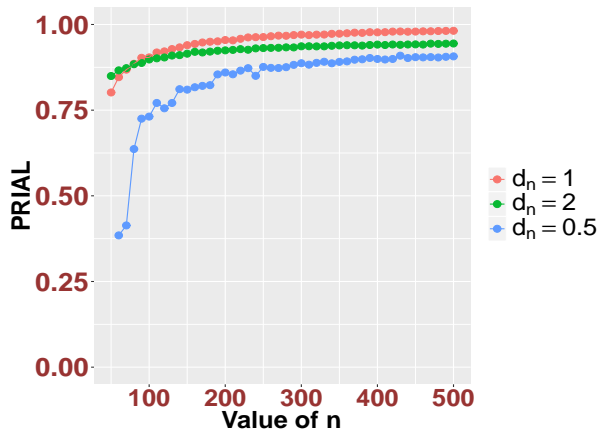}
\caption{PRIAL against matrix dimension $n$. We consider the setting $\wt{A}=\text{diag}(8,5,1,\cdots, 1)$ and $\wt{B}=\text{diag}(3,1,\cdots, 1).$ }\label{fig_PRIAL}
\end{center}
\end{figure} 


Before concluding this section,  we provide a useful result for the estimation of spikes. By Theorem \ref{thm_outlier}, we need to know the form of $m_{2c}$ in order to estimate the spikes of $\wt A$. However, thanks to 
the anisotropic local law in \cite{yang2018} (see also Theorem \ref{LEM_SMALL} and Theorem \ref{lem_localout} in the supplement), it is possible to have an adaptive estimator for the spikes of $\wt{A}$ based only on the data matrices $\widetilde{\mathcal Q}_{2}$ if $\wt{B}$ is a small-rank perturbation of the identity matrix. We define 
$$\widehat{\sigma}^a_i:=-\left( \frac{1}{n} \sum_{\nu=r+s+1}^{n} \frac{1}{\wt\lambda_\nu(\wt{\mathcal Q}_2)-\wt\lambda_{\al(i)}} \right)^{-1},\quad 1\le i \leq r+s.$$ 
Similarly, if $A$ is a small-rank perturbation of the identity matrix, then we have the following estimator for the spikes of $\wt{B}$:
$$\widehat{\sigma}^b_\mu:=-\left( \frac{1}{n} \sum_{k=r+s+1}^{p} \frac{1}{\wt\lambda_k(\wt{\mathcal Q}_1)-\wt\lambda_{\beta(\mu)}} \right)^{-1}, \quad 1\le \mu \leq r+s.$$ 
We claim the following result.
 
 \begin{theorem}\label{thm_adaptiveest}
Suppose that the Assumptions \ref{assm_big1}, \ref{ass:unper} and \ref{assum_supercritical} hold. 
Suppose $\wt B=I_n + \cal M_n$, where $\cal M_n$ is a matrix of rank $l_n$. 
Then we have that for $1\le  i \leq r$,
\begin{equation}
\tsig_i^a=\widehat{\sigma}^a_i+\OO_{\prec}(n^{-1}l_n + \phi_n). \label{finite rankB}
\end{equation}
Similarly, if $\wt A$ is an $l_n$-rank perturbation of the identity matrix, then for $1\le \mu \le s$,
 \begin{equation}
\tsig_\mu^b=\widehat{\sigma}_\mu^b+\OO_{\prec}(n^{-1}l_n + \phi_n). \label{finite rankA}
\end{equation}
 \end{theorem}

The proof of Theorem \ref{thm_adaptiveest} will be given in the supplement. Here we use some Monte-Carlo simulations to illustrate the accuracy of the above estimators. We set 
\begin{equation*}
\widetilde A=\text{diag}( \widetilde{\sigma}^a, 1,\cdots, 1),  \quad \widetilde B=\text{diag}(3, 1,\cdots, 1). 
\end{equation*} 
In Table \ref{table_estperformance}, we give the estimation of $\widetilde{\sigma}^a$ using $\widehat{\sigma}^a$ for various combinations of $p$ and $n.$ Each value is recorded by taking an average over 2,000 simulations. We find that our estimator is quite accurate even for a small sample size.

\begin{table}[ht]
\def\arraystretch{1.3}
\begin{center}
\begin{tabular}{ccccccc}

$\widetilde{\sigma}^a/(p,n)$ & $(100, 200)$  & $(200, 400)$  & $(300,400)$  & $(400,300)$  & $(500, 400)$  \\ \hline
$4$ & 3.67 & 3.58 & 3.83  & 4.61  &  4.43\\
$5$ & 4.78 & 4.65 & 4.84 &5.49  & 5.37 \\ 
$8$ & 7.75 & 7.62 & 7.86 & 8.47 & 8.33  \\ 
$10$ & 9.83 &  9.65 & 9.88  & 10.51 &  10.37\\ 
$15$ & 14.95 &  14.86 & 14.93  & 15.56 &  15.42\\ \hline
\end{tabular}
\caption{The value of $\widehat{\sigma}^a$. We record the average of $\widehat{\sigma}^a$ over 2,000 simulations. }\label{table_estperformance}
\end{center}
\end{table}

 \section*{Acknowledgements}
 The authors would like to thank Zhou Fan and Edgar Dobriban for helpful discussions. We also want to thank the editor, the associated editor and two anonymous referees for their helpful comments, which have improved the paper significantly.

\begin{center}
{\Large Supplementary material}
\end{center}
This supplementary material contains further explanation,  auxiliary lemmas and technical proofs and additional simulations for the main results of the paper.

\begin{appendix}

\renewcommand{\theequation}{S.\arabic{equation}}
\renewcommand{\thetable}{S.\arabic{table}}
\renewcommand{\thefigure}{S.\arabic{figure}}
\renewcommand{\thesection}{S.\arabic{section}}
\renewcommand{\thelemma}{S.\arabic{lemma}}


\section{Numerical simulations}\label{simu_extra} In this section, we report additional results of the  numerical simulations of the paper. 


\subsection{Discussion on the choices of $\omega$} \label{simu_threshold} In this subsection, we report the empirical results on the choices of $\omega.$ Recall that $\omega_{I_p, I_n}$ is the value of $\omega$ generated by the two-step procedure described in Section \ref{sec:estspike} and $\omega_{A,B}$ is generated by replacing $I_p$ and $I_n$ with $A$ and $B.$ We consider the setting 
\begin{equation*}
A=\operatorname{diag}(\underbracket{1, \cdots,1}_{p/2 \  \text{times}}, \underbracket{2, \cdots,2}_{p/2 \ \text{times}}), \  B=\operatorname{diag}(\underbracket{3, \cdots,3}_{n/2 \  \text{times}}, \underbracket{4, \cdots,4}_{n/2 \ \text{times}}).
\end{equation*} 
In Figure \ref{fig_thresholddiff}, we record the differences between $\omega_{I_p, I_n}$ and $\omega_{A,B}$, i.e., $|\omega_{I_p, I_n}-\omega_{A,B}|$ for different values of $n$ and $d_n.$ We find that the difference is small even for not so large $n.$ It also decreases when $n$ increases. 

\begin{figure}[H]
 \captionsetup{width=1\linewidth}
 \begin{center}
\includegraphics[width=8cm,height=6cm]{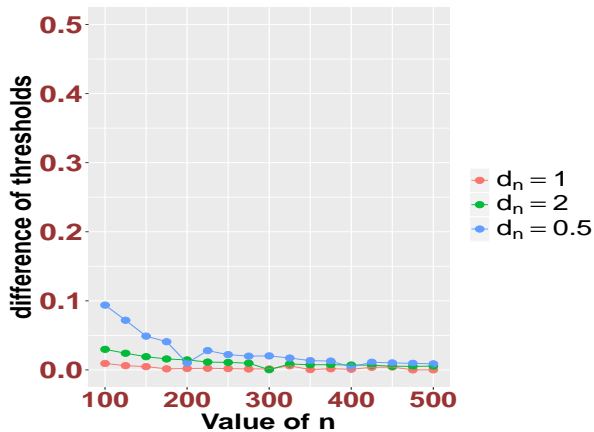} 
\caption{Threshold difference $|\omega_{I_p, I_n}-\omega_{A,B}|$ under the nomial level $0.95$ with $10^4$ simulations.  }\label{fig_thresholddiff}
\end{center}
\end{figure} 

Moreover, this simple choice of $\omega_{I_p, I_n}$ will not influence the frequency of misestimation especially when the spikes are reasonable large. In Figure \ref{fig_estthresdiff}, we record the frequency of misestimation of the setting
\begin{equation*}
\widetilde{A}=\operatorname{diag}(x, \underbracket{1, \cdots,1}_{p/2-1 \  \text{times}}, \underbracket{2, \cdots,2}_{p/2 \ \text{times}}), \  \widetilde{B}=\operatorname{diag}(5,\underbracket{3, \cdots,3}_{n/2-1 \  \text{times}}, \underbracket{4, \cdots,4}_{n/2 \ \text{times}}).
\end{equation*} 
We conclude that when $x$ is above some level, the frequencies of misestimation stay the same no matter we use $\omega_{I_p, I_n}$ or $\omega_{A,B}$. 

\begin{figure}
 \captionsetup{width=0.8\linewidth}
 \begin{center}
\begin{subfigure}{0.4\textwidth}
\includegraphics[width=6.7cm,height=5cm]{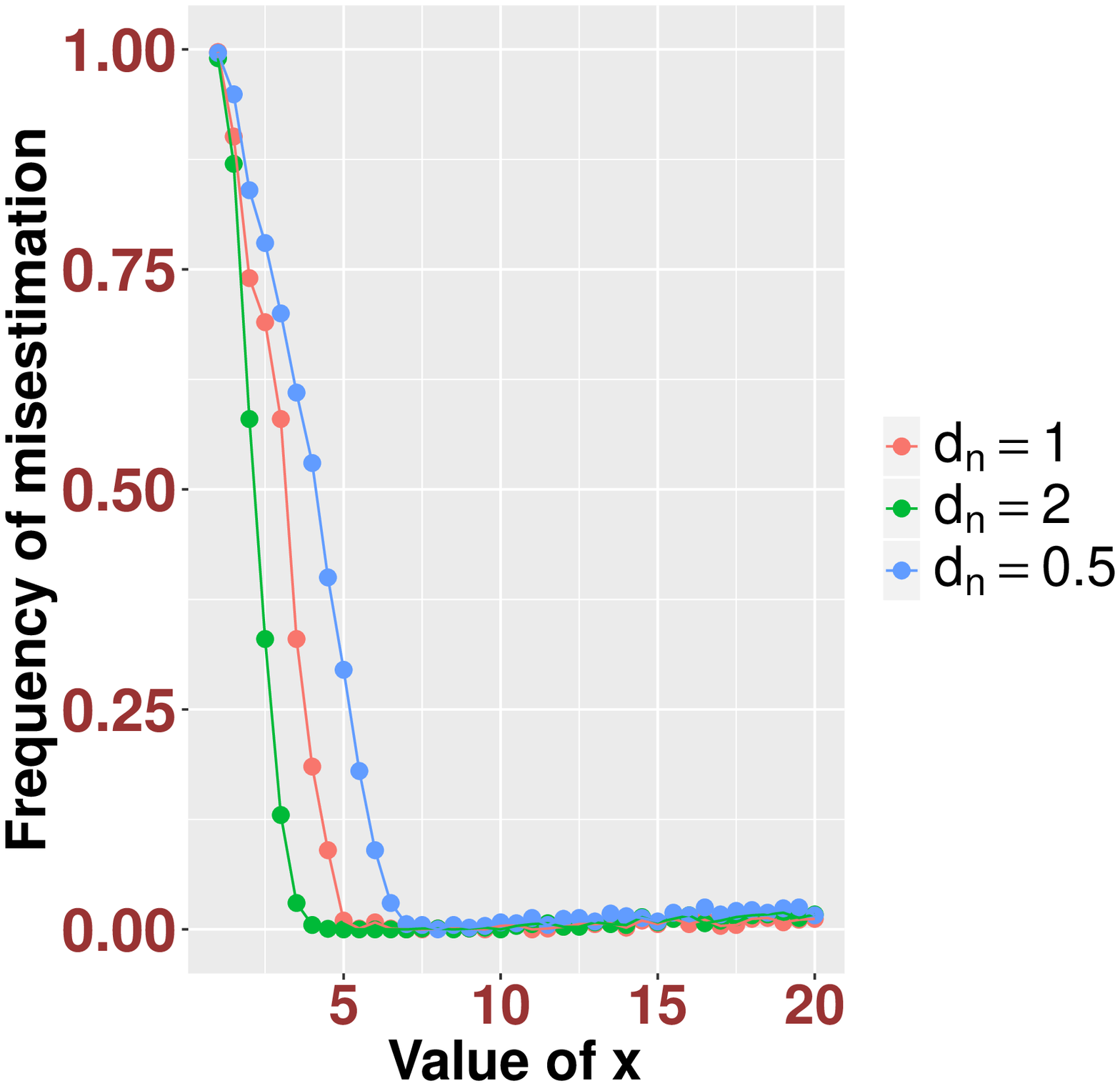}
\subcaption{
Frequency using $\omega_{A,B}.$  }\label{fig_true}
\end{subfigure}
\begin{subfigure}{0.4\textwidth}
\includegraphics[width=6.7cm,height=5cm]{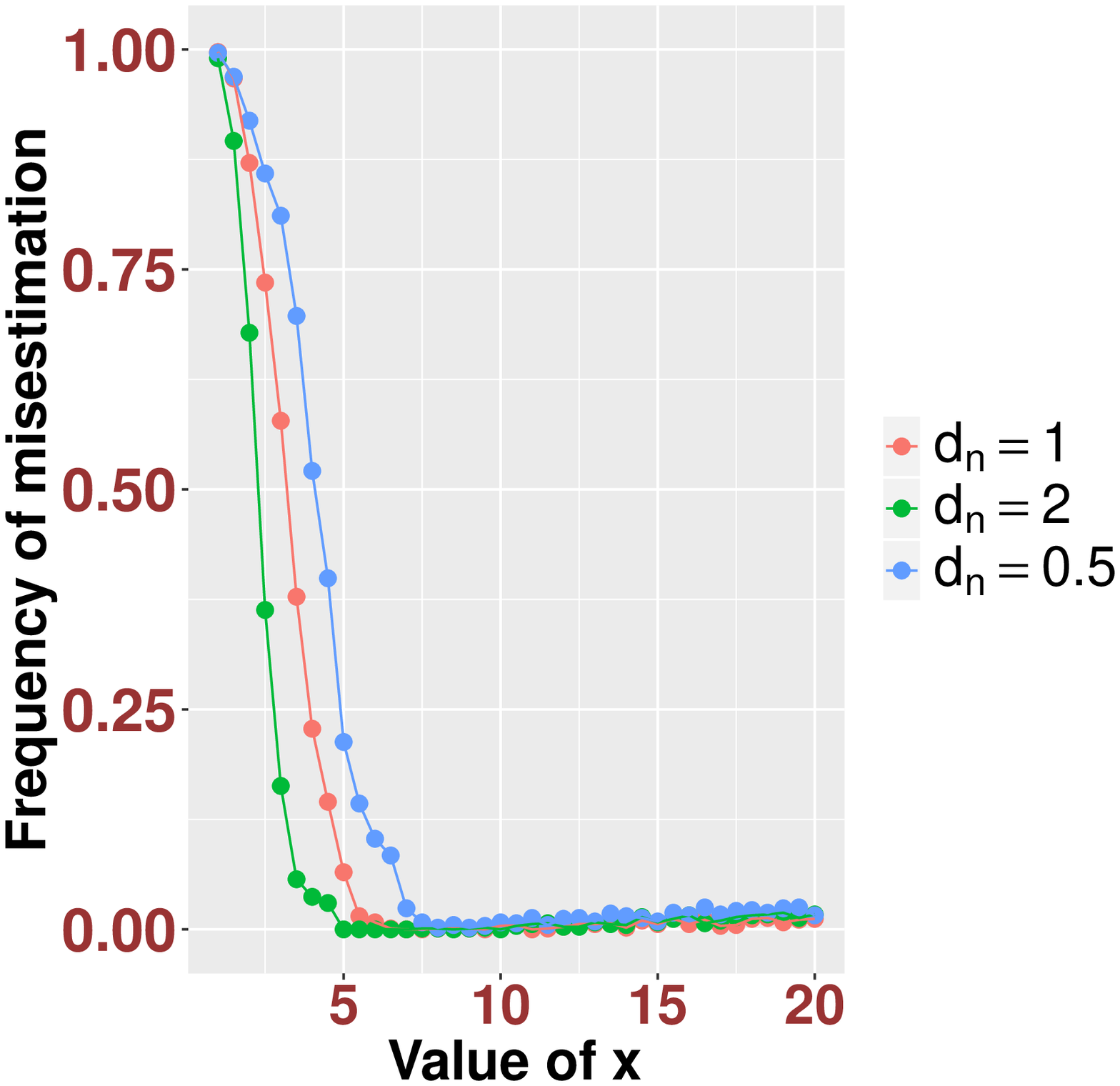}
\subcaption{
Frequency using $\omega_{I_p, I_n}.$ }\label{fig_simple}
\end{subfigure}
\caption[]{Frequencey of misestimation for different values of $x$ using $\omega_{A, B}$ and $\omega_{I_p, I_n}$ respectively. We choose the nominal level 0.95 and report the results for 2,000 simulations. Here $p=300.$ } \label{fig_estthresdiff}
\end{center}
\end{figure} 

Finally, we find that for smaller values of $x,$ an accurate estimation of $A$ and $B$ could potentially reduce the frequency of misestimation. In the literature, there exist some efficient algorithms on estimating $A$ and $B$ when one of them is identity, for instance,  \cite{elkaroui2008, kong2017, LW2015}. In the following numerical simulations, we take $B=I_n$ and use the algorithm developed in  \cite{LW2017}, which is essentially the implementation of \cite{LW2015}. We make use of the R package \texttt{nlshrink}.  We consider the setting
\begin{equation*}
\widetilde{A}=\operatorname{diag}(x, \underbracket{1, \cdots,1}_{p/2-1 \  \text{times}}, \underbracket{2, \cdots,2}_{p/2 \ \text{times}}), \  \widetilde{B}=\operatorname{diag}(3,\underbracket{1, \cdots,1}_{n/2-1 \  \text{times}}, \underbracket{1, \cdots,1}_{n/2 \ \text{times}}).
\end{equation*} 
We first use  the numerical method as described in \cite{LW2017} to find an estimator $\widehat{A}$ of $A,$  and then use $\omega_{\widehat{A}, I_n}$ as our threshold. We conclude that it will reduce the frequencies of misestimation for smaller spikes compared to the case which simply uses $\omega_{I_p, I_n}.$ In Tables \ref{table_smallerspikes1}--\ref{table_smallerspikes3}, uner the nominal level $0.95,$ we record the frequencies of misestimation using 2,000 simulations with the values $\omega_{I_p, I_n}, \omega_{\widehat{A}, I_n}$ and $\omega_{A, I_n}$ for $d_n=0.5,\ 1,\ 2$. Based on these numerical results, instead of simply using $\omega_{I_p, I_n},$  we suggest the use of $\omega_{\widehat{A}, I_n}$ for smaller $x.$

\begin{table}[H]
\def\arraystretch{1.2}
\begin{center}
\begin{tabular}{cccccccccccc}

$x$ & $1$  & $1.5$  & $2$  & $2.5$  & $3$  & $3.5$ & $4$ & $4.5$ & $5$ & $5.5$ & $6$ \\ \hline
$\omega_{I_p, I_n}$ & 0.998 & 0.935 & 0.885  & 0.731  &  0.63 & 0.51 & 0.421 & 0.31 & 0.19 & 0.06 & 0.009\\
$\omega_{\widehat{A}, I_n}$ & 0.998 & 0.92 & 0.83 &0.71  & 0.625 & 0.492 & 0.395 & 0.3 & 0.19 & 0.06 & 0.008 \\ 
$\omega_{A,I_n}$ & 0.997 & 0.915 & 0.813 & 0.694 & 0.596 & 0. 478 & 0.39 & 0.291 & 0.17 & 0.03 & 0.008\\  \hline
\end{tabular}
\caption{Frequency of misestimation using different values of thresholds. Here $n=300, d_n=0.5.$  }\label{table_smallerspikes1}
\end{center}
\end{table}

\begin{table}[H]
\def\arraystretch{1.2}
\begin{center}
\begin{tabular}{cccccccccccc}

$x$ & $1$  & $1.5$  & $2$  & $2.5$  & $3$  & $3.5$ & $4$ & $4.5$ & $5$ & $5.5$ & $6$ \\ \hline
$\omega_{I_p, I_n}$ & 0.997 & 0.856 & 0.784 & 0.693 & 0.523 & 0.371 & 0.231 & 0.11 & 0.02 & 0.007 & 0.005\\
$\omega_{\widehat{A}, I_n}$ & 0.998  & 0.85 & 0.74 & 0.654 & 0.5 & 0.351 & 0.187 & 0.1 & 0.009 & 0.007& 0.005\\ 
$\omega_{A,I_n}$ & 0.997 & 0.837 & 0.721 & 0.65 & 0.5 & 0.33 & 0.18 & 0.087 & 0.007 & 0.005 & 0.005\\  \hline
\end{tabular}
\caption{Frequency of misestimation using different values of thresholds. Here $n=300, d_n=1.$  }\label{table_smallerspikes2}
\end{center}
\end{table}

\begin{table}[H]
\def\arraystretch{1.2}
\begin{center}
\begin{tabular}{cccccccccccc}

$x$ & $1$  & $1.5$  & $2$  & $2.5$  & $3$  & $3.5$ & $4$ & $4.5$ & $5$ & $5.5$ & $6$ \\ \hline
$\omega_{I_p, I_n}$ & 0.997 & 0.81 & 0.67 & 0.48 & 0.286 & 0.11 & 0.06 & 0.008 & 0.005 & 0.005 & 0.006 \\
$\omega_{\widehat{A}, I_n}$ & 0.997 & 0.81 & 0.62 & 0.42 & 0.27 & 0.1 & 0.05 & 0.007 & 0.006 & 0.005 & 0.005  \\ 
$\omega_{A,I_n}$ & 0.997 &0.793 & 0.62 & 0.417 & 0.24 & 0.1 & 0.02 & 0.005 & 0.005 & 0.006 & 0.004 \\  \hline
\end{tabular}
\caption{Frequency of misestimation using different values of thresholds. Here $n=300, d_n=2.$  }\label{table_smallerspikes3}
\end{center}
\end{table}


\subsection{Additive spiked model} \label{additiveapp}
We consider the following example:
\begin{equation}\label{eq_additive}
A=U\Sigma^A U^*, \quad \Delta=x \bm{u} \bm{u}^*, \quad B=I_n,
\end{equation}
where $\bm{u}=p^{-1/2}\mathbf{1}_p$ and 
\begin{equation*}
\Sigma^A=\operatorname{diag}( \underbracket{30, \cdots,30}_{p/2 \  \text{times}}, \underbracket{1, \cdots,1}_{p/2 \ \text{times}}). \ 
\end{equation*}
Here we generate $U$ as orthogonal matrix from the R package \texttt{pracma} and set $x=35, d_n=1/3.$ In terms of eigenvalues, $\widetilde{A}=A+\Delta$ is a rank-one additive spiked model (recall Remark \ref{generaladdrmk}). 
However, 
we find that it actually generates two outlier eigenvalues as recorded in Figure \ref{fig_general}.

\begin{figure}[ht]
 \captionsetup{width=1\linewidth}
 \begin{center}
\includegraphics[width=8cm,height=6cm]{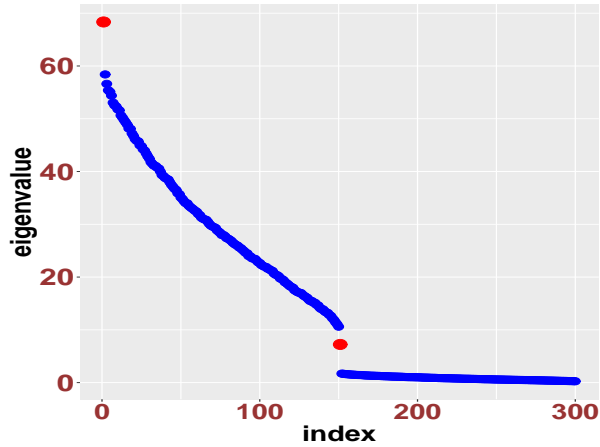}
\caption{General additive model (\ref{eq_additive}). Here $p=300.$ We can see that there exist two outlier eigenvalues associated with each bulk component.  }\label{fig_general}
\end{center}
\end{figure}

\section{Proof of Corollary 3.19}\label{sec_supp_right} 
Fix any sufficiently small constant $\e>0$. We then choose $\phi_n= n^{-c_\phi+\e}$ with $c_\phi=1/2-2/a $. Then we introduce the following truncation  
$$\wt X :=\mathbf 1_{\Omega} X, \quad \Omega :=\left\{\max_{i,j} |x_{ij}|\le \phi_n \right\}.$$
By the moment conditions (\ref{condition_4e}) and a simple union bound, we have
\begin{equation}\label{XneX}
\mathbb P(\wt X \ne X ) =\OO ( n^{-a\e}).
\end{equation}
Using (\ref{condition_4e}) and integration by parts, it is easy to verify that 
\begin{align*}
\mathbb E  \left|x_{ij}\right|1_{|x_{ij}|> \phi_n} =\OO(n^{-2-\e}), \quad \mathbb E \left|x_{ij}\right|^2 1_{|x_{ij}|> \phi_n} =\OO(n^{-2-\e}),
\end{align*}
which imply that
$$|\mathbb E  \tilde x_{ij}| =\OO(n^{-2-\e}), \quad  \mathbb E |\tilde x_{ij}|^2 = n^{-1} + \OO(n^{-2-\e}).$$
Moreover, we trivially have
$$\mathbb E  |\tilde x_{ij}|^4 \le \mathbb E  |x_{ij}|^4 =\OO(n^{-2}).$$
Hence $\wt X$ satisfies Assumptions \ref{assm_big1}, and we can apply Theorems \ref{thm_outlier}, \ref{thm_eigenvaluesticking}, \ref{thm_eveout}, \ref{thm_noneve}, \ref{thm_rightout} and \ref{thm_rightbulk}  to it with $\phi_n= n^{2/a-1/2-\e}$. Since $\e$ can be arbitrarily small, we conclude the proof. 


\section{Basic tools and proof of Theorem 4.5}\label{sec:tools}
In this section, we collect some tools that will be used in the proof. We introduce the following quantities:
\begin{equation}\label{defn_m1m2}
 m_1^{(n)}(z):= \frac{1}n\tr \left(A \mathcal G_1(z)\right) ,\quad m_2^{(n)}(z):=\frac{1}{n}\tr\left( B\mathcal G_2(z)\right). 
\end{equation}
First, the following lemma collects some basic properties of stochastic domination (Definition \ref{stoch_domination} of the paper), which will be used tacitly in the proof.

\begin{lemma}[Lemma 3.2 in \cite{isotropic}]\label{lem_stodomin}
Let $\xi$ and $\zeta$ be families of nonnegative random variables.

(i) Suppose that $\xi (u,v)\prec \zeta(u,v)$ uniformly in $u\in U$ and $v\in V$. If $|V|\le n^C$ for some constant $C$, then $\sum_{v\in V} \xi(u,v) \prec \sum_{v\in V} \zeta(u,v)$ uniformly in $u$.

(ii) If $\xi_1 (u)\prec \zeta_1(u)$ and $\xi_2 (u)\prec \zeta_2(u)$ uniformly in $u\in U$, then $\xi_1(u)\xi_2(u) \prec \zeta_1(u)\zeta_2(u)$ uniformly in $u$.

(iii) Suppose that $\Psi(u)\ge n^{-C}$ is deterministic and $\xi(u)$ satisfies $\mathbb E\xi(u)^2 \le n^C$ for all $u$. Then if $\xi(u)\prec \Psi(u)$ uniformly in $u$, we have $\mathbb E\xi(u) \prec \Psi(u)$ uniformly in $u$.
\end{lemma}

Till the end of this supplement, we will make use of the following conventions.
The fundamental large parameter is $n$ and we always assume that $p$ is comparable to and depends on $n$. We use $C$ to denote a generic large positive constant, whose value may change from one line to the next. Similarly, we use $\epsilon$, $\tau$, $c$, etc. to denote generic small positive constants. If a constant depend on a quantity $a$, we use $C(a)$ or $C_a$ to indicate this dependence. For two quantities $a_n$ and $b_n$ depending on $n$, the notation $a_n = \OO(b_n)$ means that $|a_n| \le C|b_n|$ for some constant $C>0$, and $a_n=\oo(b_n)$ means that $|a_n| \le c_n |b_n|$ for some positive sequence $c_n\downarrow 0$ as $n\to \infty$. We also use the notations $a_n \lesssim b_n$ if $a_n = \OO(b_n)$, and $a_n \sim b_n$ if $a_n = \OO(b_n)$ and $b_n = \OO(a_n)$. For a matrix $A$, we use $\|A\|:=\|A\|_{l^2 \to l^2}$ to denote the operator norm; 
for a vector $\mathbf v=(v_i)_{i=1}^n$, $\|\mathbf v\|\equiv \|\mathbf v\|_2$ stands for the Euclidean norm. 
For a matrix $A$ and a number $a>0$, we write $A=\OO(a)$ if $\|A\|=\OO(a)$. In this paper, we often write an identity matrix of any dimension as $I$ or $1$ without causing any confusions.

We record the following lemma for matrix perturbation, which follows from a simple algebraic calculation.

\begin{lemma} [Woodbury matrix identity] \label{lem_woodbury} For $\mathcal{A},S,\mathcal{B},T$ of conformable dimensions, we have 
\begin{equation*}
(\mathcal A+S\mathcal BT)^{-1}=\mathcal A^{-1}-\mathcal A^{-1}S(\mathcal B^{-1}+T\mathcal A^{-1}S)^{-1}T\mathcal A^{-1}.
\end{equation*}
as long as all the operations are legitimate. As a special case, we have the following Hua's identity:
\begin{equation}\label{Huaineq}
\mathcal A-\mathcal A(\mathcal A+\mathcal B)^{-1}\mathcal A=\mathcal B-\mathcal B(\mathcal A+\mathcal B)^{-1}\mathcal B
\end{equation}
if $\mathcal A+\mathcal B$ is non-singular.
\end{lemma}

We also need the following eigenvalue interlacing result for our spiked separable covariance model (\ref{eq_sepamodel}) of the paper. It is an analog of Corollary 4.2 in \cite{principal} for spiked covariance matrices. 

\begin{lemma} [Eigenvalue interlacing]  \label{lem_weylmodi} 
Recall that the eigenvalues of $\ctQ_1$ and $\mathcal{Q}_1$ are denoted by $\{\wt\lambda_i\}$ and $\{\lambda_i\}$, respectively. Then we have 
\begin{equation}\label{interlacing_eq0}
\wt\lambda_i \in [\lambda_{i}, \lambda_{i-r-s}],
\end{equation}
where we adopt the convention that  $\lambda_{i}=\infty$ if $i<1$ and $\lambda_i = 0$ if $i>p$. 
\end{lemma}
\begin{proof}
We first consider the rank one deformation with $r=1$ and $s=0$: $\wt A= (1+d^a\bv^a(\bv^a)^*)A$ with $d^a> 0$ and $\bv^a$ being an eigenvector of $A$. Then we have
\be\label{slG1}
\begin{split}
\CTG_1&= \left(\cal P^{1/2}A^{1/2}XBX^*A^{1/2}\cal P^{1/2} - z\right)^{-1} \\
&= \cal P^{-1/2} \left[\mathcal G_1^{-1}  + \bv^a  \frac{zd_a}{d_a+1} (\bv^a)^*\right]^{-1}  \cal P^{-1/2} ,
\end{split}
\ee
where $\cal P:=1+d^a\bv^a(\bv^a)^*$. 
Then applying Lemma \ref{lem_woodbury} to \eqref{slG1}, we obtain that
\begin{equation*}
(\CTG_1)_{\bv^a\bv^a}=\frac{(\mathcal G_1)_{\bv^a\bv^a}}{d^a+1}-\frac{(\mathcal G_1)_{\bv^a\bv^a}^2}{d^a+1}  \frac{z}{(d^a)^{-1}+1+z(\mathcal G_1)_{\bv^a\bv^a}},
\end{equation*}
where we used the following short-hand notations
\begin{equation}\label{eq_shorhand}
(\CTG_1)_{\bv^a \bv^a}=\langle \bv^a, \CTG_1 \bv^a \rangle, \quad \left( \mathcal G_1\right)_{\bv^a \bv^a}=\langle \bv^a, \mathcal G_1 \bv^a \rangle.
\end{equation}
Thus we get
\begin{equation}\label{eq_pf_interlacing}
\frac{1}{(\CTG_1)_{\bv^a \bv^a}}=\frac{1+d^a}{(\mathcal G_1)_{\bv^a \bv^a}}+zd^a .
\end{equation}
We denote the eigenvectors of $\mathcal{Q}_1$ and $\wt{\mathcal{Q}}_1$ as $\{{\bm \xi}_k\}_{k=1}^p$ and $\{\wt{\bm \xi}_k\}_{k=1}^p$, respectively. Then writing \eqref{eq_pf_interlacing} in spectral decomposition gives
\begin{equation}\label{interlacing_eq}
(d^a+1)\left(\sum_k\frac{|\langle \mathbf{v}^a, {\bm \xi}_k \rangle|^2}{\lambda_k-z}\right)^{-1} =\left(\sum_k \frac{|\langle \mathbf{v}^a, \wt{\bm \xi}_k \rangle|^2}{\wt\lambda_k-z}\right)^{-1}- zd^a . 
\end{equation}

By adding a small perturbation to $\mathcal Q_1$, we may assume without loss of generality that (i) $\lambda_1,\cdots, \lambda_p$ are all positive and distinct, and (ii) all $\langle \mathbf{v}^a,{\bm \xi}_k \rangle$ and $\langle \mathbf{v}^a, \wt{\bm \xi}_k \rangle$ are nonzero. Note that since eigenvalues and eigenvectors depend continuously on the matrix entries, we can remove the arbitrarily small perturbation and obtain the corresponding result for the original matrices $\mathcal Q_1$ and $\wt{\mathcal Q}_1$. Moreover, it is always possible to choose such perturbation. For example, we can add a matrix $\epsilon H$, where the entries of $H$ are bounded and have absolutely continuous densities. Then (i) and (ii) hold with probability 1 for any $\epsilon>0$. Thus there must exist a realization of $H$ such that (i) and (ii) hold for $\mathcal Q_1+\e H$ and $\wt{\mathcal Q}_1+\e H$.

By (i) and (ii), the left-hand side of (\ref{interlacing_eq}) defines a function of $z \in (0, \infty)$ with $(p-1)$ poles and $p$ zeros. The function is smooth and decreasing away from the singularities, and its zeros are $\lambda_1,\cdots,\lambda_p$. Now using the fact that $z$ is an eigenvalue of $\ctQ_1$ if and only if the left-hand side of \eqref{interlacing_eq} is equal to $-zd^a<0$, we obtain the interlacing property \eqref{interlacing_eq0} for $r=1$ and $s=0$.

Next, for the case $r=0$ and $s=1$, we conclude the proof easily by applying \eqref{interlacing_eq0} to $\wt{\mathcal Q}_2$ and using the fact that $\wt{\mathcal Q}_2$ have the same nonzero eigenvalues as $\wt{\mathcal Q}_1$. Note that the above arguments are purely deterministic. They work for any non-negative definite matrix $A^{1/2}XBX^*A^{1/2}$ and any rank one deformation of the form $\wt A^{1/2}XBX^* \wt A^{1/2}$ or $A^{1/2}X\wt BX^* A^{1/2}$, where 
$$\widetilde A = A\Big(1+d^a\bv^a(\bv^a)^*\Big)\quad \text{or} \quad \widetilde B= B\Big(1+d^b \bv^b(\bv^b)^*\Big),
$$
with $d^a> 0$, $d^b> 0$, and $\bv^a$ and $\bv^b$ being eigenvectors of $A$ and $B$, respectively. Then the general case \eqref{interlacing_eq0} with any finite $r, s=\OO(1)$ follows from a simple induction argument. \end{proof}

\subsection{Properties of limiting laws}
First of all, we report the properties of the limiting spectral distribution.
\begin{lemma}[Lemma 2.6 of \cite{yang2018}]\label{lambdar_sqrt}
Under the assumptions \eqref{eq_ratio}, \eqref{assm3} and \eqref{assm_gap}, there exist constants $a_{1,2}>0$ such that
\be\label{sqroot3}
\rho_{1,2c}(\lambda_+ - x) = a_{1,2} x^{1/2} + \OO(x), \quad x\downarrow 0,
\ee
and
\be\label{sqroot4}
\quad m_{1,2c}(z) = m_{1,2c}(\lambda_+) + \pi a_{1,2}(z-\lambda_+)^{1/2} + \OO(|z-\lambda_+|), \quad z\to \lambda_+  .
\ee
The estimates \eqref{sqroot3} and \eqref{sqroot4} also hold for $\rho_c$ and $m_c$ with different constants. 
\end{lemma}

For any constants $\varsigma_1, \varsigma_2>0,$ we denote a domain of the spectral parameter $z$ as 
\begin{equation}\label{eq_paraset}
S(\varsigma_1,\varsigma_2):=\{z=E+\mathrm{i} \eta: \lambda_+-\varsigma_1 \leq E \leq \varsigma_2 \lambda_+, \ 0< \eta \leq 1\}.
\end{equation} 
 For $ z=E+\mathrm{i} \eta,$ we define the distance to the rightmost edge as
 \begin{equation}\label{KAPPA}
 \kappa \equiv \kappa_E:=|E-\lambda_+|.
 \end{equation}
Then we have the following lemma, which summarizes some basic properties of $m_{1,2c}$ and $\rho_{1,2c}$.

\begin{lemma}\label{lem_mplaw} 
Suppose Assumptions \ref{assm_big1} and \ref{ass:unper} of the paper hold. Then there exists sufficiently small constant $\varsigma_1>0$ such that the following estimates hold: 
\begin{itemize}
\item[(i)]
\begin{equation}
\rho_{1,2c}(x) \sim \sqrt{\lambda_+-x}, \quad \ \ \text{ for } x \in \left[\lambda_+ - 2\varsigma_1,\lambda_+ \right];\label{SQUAREROOT}
\end{equation}
\item[(ii)] for $z =E+\ii \eta\in S(\varsigma_1,\varsigma_2)$, 
\begin{equation} \label{eq_estimm} 
\vert m_{1,2c}(z) \vert \sim 1,  \quad  \im m_{1,2c}(z) \sim \begin{cases}
    \frac{\eta}{\sqrt{\kappa+\eta}}, & \text{ if } E\geq \lambda_+ \\
    \sqrt{\kappa+\eta}, & \text{ if } E \le \lambda_+\\
  \end{cases},
\end{equation}
and
\begin{equation}\label{eq_realestimate}
|\operatorname{Re} m_{1,2c}(z)-m_{1,2c}(\lambda_+) | 
\sim
\begin{cases}
\sqrt{\kappa+\eta}, &  \ \text{if} \ E \geq \lambda_+ \\
\frac{\eta}{\sqrt{\kappa+\eta}}+\kappa, & \ \text{if} \ E \le \lambda_+
\end{cases};
\end{equation}

\item[(iii)] there exists constant $\tau'>0$ such that
\begin{equation}\label{Piii}
\min_{ \mu } \vert 1 + m_{1c}(z)\sigma_\mu^b \vert \ge \tau', \quad \min_{i } \vert 1 + m_{2c}(z)\sigma_i^a  \vert \ge \tau',
\end{equation}
for any $z \in S(\varsigma_1,\varsigma_2)$.
\end{itemize}
The above estimates (i)-(iii) also hold for $z$ on the real axis, i.e., $z \in \overline{S(\varsigma_1, \varsigma_2)} $. Finally, the estimates \eqref{SQUAREROOT}-\eqref{eq_realestimate} also hold for $\rho_c$ and $m_c$. 
\end{lemma}
\begin{proof}
The estimates \eqref{SQUAREROOT}, \eqref{eq_estimm} and \eqref{Piii} have been proved in \cite[Lemma 3.4]{yang2018}. The estimate \eqref{eq_realestimate} follows directly from (\ref{sqroot4}). 
\end{proof}

The next lemma contains some basic estimates for $\theta_{1,2}$ in (\ref{g12c})  and the derivatives of $m_{1,2c}$ and $g_{1,2c}$.

\begin{lemma} \label{lem_derivativeprop}

Suppose that Assumptions \ref{assm_big1} and \ref{ass:unper}  of the paper hold. For $\sigma_1 \ge -m_{1c}^{-1}(\lambda_+)$ and $\sigma_2 \ge -m_{2c}^{-1}(\lambda_+)$, we have  
\be\label{eq_derivativebound}
\begin{split}
& \theta_1(\sigma_2) - \lambda_+ =g_{2c}(-\sigma_2^{-1}) - \lambda_+  \sim (\sigma_2+m_{2c}^{-1}(\lambda_+))^2, \\
&  \theta_2(\sigma_1) - \lambda_+ = g_{1c}(-\sigma_1^{-1}) - \lambda_+   \sim (\sigma_1+m_{1c}^{-1}(\lambda_+))^2.
\end{split}
\ee
For $x > \lambda_+$ and $m_{1,2} > m_{1,2c}(\lambda_+)$, we have 
\begin{align}
m_{2c}'(x) \sim \kappa_x^{-1/2},\quad  & m_{1c}'(x) \sim \kappa_x^{-1/2},  \label{eq_mderivative0}\\
g_{2c}'(m_2) \sim (m_2 - m_{2c}(\lambda_+)), \quad & g_{1c}'(m_1) \sim (m_1 - m_{1c}(\lambda_+)).\label{eq_gderivative0}
\end{align}
Moreover, the above estimates imply that
\begin{align}
m_{2c}'(\theta_1(\sigma_2)) \sim \frac{1}{\sigma_2+m_{2c}^{-1}(\lambda_+)},\quad & m_{1c}'(\theta_2(\sigma_1)) \sim \frac1{\sigma_1+m_{1c}^{-1}(\lambda_+) },\label{eq_mderivative}\\
g_{2c}'(-\sigma_2^{-1}) \sim \sigma_2+m_{2c}^{-1}(\lambda_+), \quad & g_{1c}'(-\sigma_1^{-1}) \sim \sigma_1+m_{1c}^{-1}(\lambda_+).\label{eq_gderivative}
\end{align}
\end{lemma}
\begin{proof}
With the definitions (\ref{g12c}) and (\ref{sqroot4}) of the paper, we can obtain that 
$$- \sigma_2^{-1}=m_{2c}(\theta_1(\sigma_2)) = m_{2c}(\lambda_+) + \pi a_{2}\sqrt{\theta_1(\sigma_2)-\lambda_+} + \OO(|\theta_1(\sigma_2)-\lambda_+|) $$
if $\theta_1(\sigma_2)-\lambda_+ \le \varsigma_1$ for some sufficiently small constant $0<\varsigma_1<1$, and 
$$- \sigma_2^{-1}=m_{2c}(\theta_1(\sigma_2)) \ge  m_{2c}(\lambda_+ + \varsigma_1) = m_{2c}(\lambda_+) + \pi a_{2}\sqrt{\varsigma_1} + \OO(\varsigma_1)$$
if $\theta_1(\sigma_2)-\lambda_+ \ge \varsigma_1$, where in the second inequality we use the fact that $m_{2c}(x)$ is monotone increasing when $x>\lambda_+.$ The above two estimates imply the first estimate in \eqref{eq_derivativebound}. The second estimate in \eqref{eq_derivativebound} can be proved in the same way.

Differentiating the equation $f(z,m)=0$ in (\ref{separable_MP}) of the paper with respect to $m$, we can get that 
$$z'(m_+)=0 \quad \text{and}\quad z''(m_+)=-{\partial_m^2 f (\lambda_+,m_+)}/{\partial_z f(\lambda_+,m_+)},$$ 
where $m_+:=m_{2c}(\lambda_+)$. It was proved in \cite[Lemma 2.6]{yang2018} that $z''(m_+) \sim 1$ under the assumptions (\ref{assm3}) and (\ref{assm_gap}). Moreover, using implicit differentiation of the equation $f(z,m)=0$ and \eqref{Piii}, it is easy to show that $z^{(3)}(m) = \OO(1)$ if $m_+ - c \le m \le 0$ for some sufficiently small constant $c>0$. Hence we conclude that 
\be\label{z_m}
z'(m) = \OO(|m-m_+|),\quad \text{for }\  m_+ - c \le m \le 0. 
\ee
This implies the first estimate in \eqref{eq_gderivative0}. Since $m_{2c}$ is the inverse function of $g_{2c}$, we get  from the inverse function theorem that
$$m_{2c}'(x) = \frac{1}{g_{2c}'(m_{2c}(x))} \sim (m_{2c}(x)-m_{2c}(\lambda_+))^{-1} \sim \kappa_x^{-1/2},$$
where we used (\ref{sqroot4}) of the paper in the last step. This implies the first estimate in \eqref{eq_mderivative0}. Now taking $x=\theta_1(\sigma_2)$ and $m_2=-\sigma_2^{-1}$ in the first estimates in \eqref{eq_mderivative0} and \eqref{eq_gderivative0}, respectively, and using \eqref{eq_derivativebound}, we obtain the first estimates in  \eqref{eq_mderivative} and \eqref{eq_gderivative}.

Exchanging the roles of $(A, m_{1c}, g_{1c})$ and $(B, m_{2c}, g_{2c})$, one can prove the second estimates in \eqref{eq_mderivative0}-\eqref{eq_gderivative} in the same way. 
\end{proof}

In the proof, it is important to extend the real functions $g_{1c}$ and $g_{2c}$ to the complex plane. The following lemma can be proved with a simple complex analytical argument. 

\begin{lemma} \label{lem_complexderivative}
Suppose the assumptions of Lemma \ref{lem_derivativeprop} hold. Then for any constant $\varsigma>0$, there exist constants $\tau_0, \tau_1, \tau_2>0$ such that the following statements hold.
\begin{itemize}
\item[(i)] $m_{1c}$ and $m_{2c}$ are holomorphic homeomorphisms on the spectral domain
$$\mathbf D(\tau_0,\varsigma):=\{z=E+\mathrm{i} \eta: \lambda_+< E < \varsigma, \ -\tau_0< \eta < \tau_0\}.$$
 As a consequence, the inverse functions of $m_{1c}$ and $m_{2c}$ exist and we again denote them by $g_{1c}$ and $g_{2c}$, respectively.

\item[(ii)] We have $\mathbf D_1(\tau_1, \varsigma)\subset m_{1c} (\mathbf D(\tau_0,\varsigma))$ and $\mathbf D_2(\tau_2, \varsigma)\subset m_{2c} (\mathbf D(\tau_0,\varsigma))$, where 
$$\mathbf D_1(\tau_1,\varsigma):=\{\xi=E+\mathrm{i} \eta: m_{1c}(\lambda_+)< E < m_{1c}(\varsigma), \ -\tau_1< \eta < \tau_1\},$$
and
$$\mathbf D_2(\tau_2,\varsigma):=\{\zeta=E+\mathrm{i} \eta: m_{2c}(\lambda_+)< E < m_{2c}(\varsigma), \ -\tau_2< \eta < \tau_2\}.$$
In other words, $g_{1c}$ and $g_{2c}$ are holomorphic homeomorphisms on $\mathbf D_1(\tau_1,\varsigma)$ and $\mathbf D_2(\tau_2,\varsigma)$, respectively.

\item[(iii)] For $z\in \mathbf D(\tau_0,\varsigma)$, we have
\begin{equation}\label{eq_mcomplex}
|m_{1c}(z) - m_{1c}(\lambda_+)| \sim |z-\lambda_+|^{\frac12}, \quad  |m_{2c}(z) - m_{2c}(\lambda_+)|   \sim |z-\lambda_+|^{\frac12},
\end{equation}
and
\begin{equation}\label{eq_mcomplexd}
|m_{1c}'(z) | \sim |z-\lambda_+|^{-\frac12}, \quad  |m_{2c}'(z)|   \sim |z-\lambda_+|^{-\frac12}.
\end{equation}

\item[(iv)] For $\xi \in \mathbf D_1(\tau_1,\varsigma)$ and $\zeta\in \mathbf D_2(\tau_2,\varsigma)$, we have
\begin{equation}\label{eq_gcomplex}
|g_{1c}(\xi) - \lambda_+| \sim |\xi-m_{1c}(\lambda_+)|^2, \quad  |g_{2c}(\zeta) - \lambda_+|   \sim |\zeta-m_{2c}(\lambda_+)|^2,
\end{equation}
and
\begin{equation}\label{eq_gcomplexd}
|g_{1c}'(\xi) | \sim |\xi-m_{1c}(\lambda_+)|, \quad  |g_{2c}'(\zeta) |   \sim |\zeta-m_{2c}(\lambda_+)|.
\end{equation}

\item[(v)] For $z_1, z_2\in \mathbf D(\tau_0,\varsigma)$, $\xi_1,\xi_2 \in \mathbf D_1(\tau_1,\varsigma)$ and $\zeta_1,\zeta_2\in \mathbf D_2(\tau_2,\varsigma)$, we have 
\begin{equation}\label{eq_mdiff}
\begin{split}
|m_{1c}(z_1) - m_{1c}(z_2)| \sim |m_{2c}(z_1) - m_{2c}(z_2)|  \\
\sim \frac{|z_1-z_2|}{\max_{i=1,2}|z_i-\lambda_+|^{1/2}},
\end{split}
\end{equation}
and 
\begin{equation}\label{eq_gdiff}
\begin{split}
& |g_{1c}(\xi_1) - g_{1c}(\xi_2)| \sim |\xi_1-\xi_2|\cdot\max_{i=1,2}|\xi_i-m_{1c}(\lambda_+)|, \\ 
& |g_{2c}(\zeta_1) - g_{2c}(\zeta_2)| \sim |\zeta_1-\zeta_2|\cdot\max_{i=1,2}|\zeta_i-m_{2c}(\lambda_+)|.
 \end{split}
\end{equation}

\end{itemize}
\end{lemma}
\begin{proof}
For the proof, we choose a sufficiently small constant $\omega>0$ such that (\ref{sqroot4}) of the paper can be applied to $z\in D_\omega:=\{z=E+\ii \eta: 0<E-\lambda_+< 2\omega, -\omega < \eta < \omega\}$. We also define the spectral domain $\wt D_\omega:=\{z=E+\ii \eta: 0<E-\lambda_+ < \omega, -\omega< \eta <\omega\}$. Then the constants $\tau_0, \tau_1, \tau_2>0$ will be chosen such that they are much smaller than $\omega$. Without loss of generality, we only prove the relevant statements for $m_{2c}$ and $g_{2c}$.

Note that $m_{2c}$ is holomorphic on $\mathbb C\setminus [0,\lambda_+]$. By (\ref{sqroot4}), we see that $m_{2c}$ is a holomorphic homeomorphism for $z\in D_\omega$ as long as $\omega$ is sufficiently small. Moreover, we have
\be\label{close_to} 
g_{2c}(\xi) = \frac{1}{\pi^2a_2^2}(\xi-m_{2c}(\lambda_+))^2 + \OO\left(|\xi-m_{2c}(\lambda_+)|^3\right), \quad \xi\in m_{2c}(D_\omega).
\ee
On the other hand, with (\ref{real_stiel}) of the paper it is easy to see that there exists a constant $c_{\omega,\varsigma}>0$ such that $m_{2c}'(x) \ge c_{\omega,\varsigma}$ for all $\lambda_+ + \omega < x <\varsigma$. Then combining the implicit function theorem, analytic continuation and a compactness argument, we can conclude statement (i). 
The statement (ii) follows immediately from that
$$\im m_{2c}(E+\ii\eta) =\eta \int_0^{\lambda_+} \frac{\rho_{2c}(x)\dd x}{(x-E)^2 + \eta^2} \gtrsim \eta.$$
The estimates in (iii) and (iv) can be proved using (\ref{sqroot4}), \eqref{close_to}, and implicit differentiation of the equation $f(z,m)=0$ as in the proof for Lemma \ref{lem_derivativeprop}. We omit the details. Finally, notice that \eqref{eq_gdiff} follows directly from \eqref{eq_mdiff} together with \eqref{eq_gcomplex}.
Thus it only remains to prove \eqref{eq_mdiff}. 

The upper bound in \eqref{eq_mdiff} is given by \eqref{eq_mcomplexd}. We only need to show the lower bound. Without loss of generality, we assume that $|z_1-\lambda_+|\ge |z_2-\lambda_+|$. We consider the following three cases: (i) $z_1,z_2\in D_\omega$; (ii) $z_1,z_2\in \mathbf D(\tau_0,\varsigma)\setminus \wt D_\omega$; (iii) $z_1\in \mathbf D(\tau_0,\varsigma)\setminus D_\omega$ and $z_2\in \mathbf D(\tau_0,\varsigma)\cap \wt D_\omega$. 

In case (i), first suppose that $|z_1-z_2|\le |z_1-\lambda_+|/2$. Then \eqref{eq_mdiff} follows from the mean value theorem by using \eqref{eq_mcomplexd} and the fact that $|\xi-\lambda_+| \sim |z_1-\lambda_+|$ for any $\xi$ on the line between $z_1$ and $z_2$.
Now for $|z_1-z_2|\ge  |z_1-\lambda_+|/2$, then by (\ref{sqroot4}) we get
$$|m_{2c}(z_1)-m_{2c}(z_2)|\ge \pi a_{2}\left(\sqrt{z_1-\lambda_+} - \sqrt{z_2-\lambda_+} \right)- C |z_1-\lambda_+| \ge c\frac{|z_1-z_2|}{|z_1-\lambda_+|^{1/2}}$$
as long as we take $\omega$ to be sufficiently small. 

In case (ii), by mean value theorem and \eqref{eq_mcomplexd}, we have 
$$\left|m_{2c}(z_1) - m_{2c}(z_2) \right| \sim |z_1 - z_2| \sim \frac{|z_1-z_2|}{|z_1-\lambda_+|^{1/2}}.$$

Finally, in case (iii), we have 
\be\label{simple}\left|m_{2c}(z_1) - m_{2c}(z_2) \right| \ge \left|\re m_{2c}(z_1) - \re m_{2c}(z_2) \right|. 
\ee
Denote $z=E_1+\ii \eta_1$ and $z=E_2+\ii \eta_2$. Then applying (\ref{sqroot4}) of the paper to $m_{2c}(z_2)$ and the Stieltjes transform formula to $m_{2c}(z_1)$, we obtain that
$$|\re m_{2c}(z_1) - m_{2c}(E_1)| \le C\sqrt{\eta_1}, \quad |\re m_{2c}(z_2) - m_{2c}(E_2)|\le C_\omega \eta_2.$$
Together with \eqref{simple}, we get that 
$$\left|m_{2c}(z_1) - m_{2c}(z_2) \right| \ge |m_{2c}(E_1)-m_{2c}(E_2)| - C\sqrt{\eta_1} - C_\omega \eta_2 \ge c_\omega $$
as long as we take $\tau_0$ to be small enough. Here we used that $|m_{2c}(E_1)-m_{2c}(E_2)|\sim 1$ since $m_{2c}(x)$ is strictly decreasing. 

Combining the above three cases, we get the lower bound in \eqref{eq_mdiff}.
\end{proof}

\begin{remark}
As a corollary of \eqref{eq_mdiff} and \eqref{eq_gdiff}, we see that the following approximate isometry properties hold:
\be\label{isometry}
\begin{split}
& \left|g_{1c} (m_{2c}(z_1))-g_{1c} (m_{2c}(z_2))\right| \sim |z_1 - z_2|,\\
& \left|g_{2c} (m_{1c}(z_1))-g_{2c} (m_{1c}(z_2))\right| \sim |z_1 - z_2|,
\end{split}
\ee
and
\be\label{isometry2}
\begin{split}
&\left|m_{1c} (g_{2c}(\zeta_1))-m_{1c} (g_{2c}(\zeta_2))\right| \sim |\zeta_1 - \zeta_2|,\\ 
&\left|m_{2c} (g_{1c}(\xi_1))-m_{2c} (g_{1c}(\xi_2))\right| \sim |\xi_1 - \xi_2|,
\end{split}
\ee
for $z_1, z_2\in \mathbf D(\tau,\varsigma)$, $\xi_1,\xi_2 \in \mathbf D_1(\tau,\varsigma)$ and $\zeta_1,\zeta_2\in \mathbf D_2(\tau,\varsigma)$ for sufficiently small constant $\tau>0$.
\end{remark}

 \subsection{Local law}\label{sec_local law}

 We first introduce a convenient self-adjoint linearization trick, which has been proved to be useful in studying the local laws of random matrices of the Gram type \cite{Alt_Gram, AEK_Gram, Knowles2017, XYY_circular,yang2018}. We define the following $(p+n)\times (p+n)$ self-adjoint block matrix, which is a linear function of $X$:
 \begin{equation}\label{linearize_block}
   H \equiv H(X,z): = z^{1/2} \left( {\begin{array}{*{20}c}
   { 0 } &A^{1/2} X B^{1/2}   \\
   {B^{1/2} X^* A^{1/2} } & {0}  \\
   \end{array}} \right),  \quad z\in \mathbb C_+ .
 \end{equation}
where $z^{1/2}$ is taken to be the branch cut with positive imaginary part. Then we define its resolvent (Green's function) as
 \begin{equation}\label{eq_gz} 
 G \equiv G (X,z):= \left(H(X,z)-z\right)^{-1} .
 \end{equation}
By Schur complement formula, we can verify that (recall (\ref{def_green}) of the paper)
\begin{align} 
G(z) = \left( {\begin{array}{*{20}c}
   { \mathcal G_1} & z^{-1/2}\mathcal G_1 Y \\
   {z^{-1/2}Y^* \mathcal G_1} & { \mathcal G_2 }  \\
\end{array}} \right) = \left( {\begin{array}{*{20}c}
   { \mathcal G_1} & z^{-1/2}Y \mathcal G_2   \\
   {z^{-1/2} \mathcal G_2 Y^*} & { \mathcal G_2 }  \\ 
\end{array}} \right),\label{green2} 
\end{align}
where $Y:= A^{1/2}X B^{1/2}$. 
Thus a control of $G$ yields directly a control of the resolvents $\mathcal G_{1,2}$. Similarly, we can define $\wt H$ and $\wt G$  by replacing $A$ and $B$ with $\wt A$ and $\wt B$.

For simplicity of notations, we define the index sets
\[\mathcal I_1:=\{1,...,p\}, \quad \mathcal I_2:=\{p+1,...,p+n\}, \quad \mathcal I:=\mathcal I_1\cup\mathcal I_2.\]
Then we label the indices of the matrices according to 
$$X= (X_{i\mu}:i\in \mathcal I_1, \mu \in \mathcal I_2), \quad A=(A_{ij}: i,j\in \mathcal I_1),\quad B=(B_{\mu\nu}: \mu,\nu\in \mathcal I_2).$$  
In the rest of this paper, 
we will consistently use the latin letters $i,j\in\mathcal I_1$ and greek letters $\mu,\nu\in\mathcal I_2$. Note that for the index $1\le \mu \le n$ used in previous sections, it can be translated into an index in $\mathcal I_2$ by taking $\mu \to \mu +p$.

Next we introduce the spectral decomposition of $G$. Let
$$A^{1/2}X B^{1/2}  = \sum_{k = 1}^{p \wedge n} {\sqrt {\lambda_k} {\bm \xi}_k } {\bm\zeta} _{k}^* ,$$
be a singular value decomposition of $A^{1/2}X B^{1/2}$, where
$$\lambda_1\ge \lambda_2 \ge \ldots \ge \lambda_{p\wedge n} \ge 0 = \lambda_{p\wedge n+1} = \ldots = \lambda_{p\vee n}$$
are the eigenvalues of $\ctQ_1,$ and
$\{\bm \xi_{k}\}_{k=1}^{p}$ and $\{\bm \zeta_{k}\}_{k=1}^{n}$ are the left and right singular vectors of $A^{1/2}XB^{1/2},$ respectively. 
Then using (\ref{green2}), we can get that for $i,j\in \mathcal I_1$ and $\mu,\nu\in \mathcal I_2$,
\begin{align}
G_{ij} = \sum_{k = 1}^{p} \frac{{\bm \xi}_k(i) {\bm \xi}_k^*(j)}{\lambda_k-z},\ \  \ &G_{\mu\nu} = \sum_{k = 1}^{n} \frac{{\bm \zeta}_k(\mu) {\bm \zeta}_k^*(\nu)}{\lambda_k-z}, \label{spectral1}\\
G_{i\mu} = z^{-1/2}\sum_{k = 1}^{p\wedge n} \frac{\sqrt{\lambda_k}{\bm \xi}_k(i) {\bm \zeta}_k^*(\mu)}{\lambda_k-z}, \ \ \ &G_{\mu i} =  z^{-1/2}\sum_{k = 1}^{p\wedge n} \frac{\sqrt{\lambda_k}{\bm \zeta}_k(\mu) {\bm \xi}_k^*(i)}{\lambda_k-z}.\label{spectral2}
\end{align}

We define the deterministic limit $\Pi$ of the resolvent $G$ in (\ref{eq_gz}) as
\begin{equation}\label{defn_pi}
\Pi (z):=\left( {\begin{array}{*{20}c}
   { \Pi_1} & 0  \\
   0 & { \Pi_2}  \\
\end{array}} \right), 
\end{equation}
where
$$ \Pi_1:  =-z^{-1}\left(1+m_{2c}(z)A \right)^{-1},\quad \Pi_2:=- z^{-1} (1+m_{1c}(z)B )^{-1}.$$
Note that by (\ref{separa_m12}) and (\ref{def_mc}) we have
\be\label{mcPi}
\frac1{n}\tr \Pi_{1} =m_c, \quad  \frac1{n}\tr \left(A\Pi_{1}\right) =m_{1c}, \quad \frac1{n}\tr \left(B\Pi_{2}\right) =m_{2c}.
\ee
Define the control parameter
\begin{equation}\label{eq_defpsi}
\Psi (z):= \sqrt {\frac{\Im \, m_{2c}(z)}{{n\eta }} } + \frac{1}{n\eta}.
\end{equation}
Note that by (\ref{eq_estimm}) and (\ref{Piii}), we have
\begin{equation}\label{psi12}
\|\Pi\|=\OO(1), \ \ \Psi \gtrsim n^{-1/2} , \ \ \Psi^2 \lesssim (n\eta)^{-1}, \ \ \Psi(z) \sim  \sqrt {\frac{\Im \, m_{1c}(z)}{{n\eta }} } + \frac{1}{n\eta},
\end{equation}
for $z\in S(\varsigma_1,\varsigma_2)$. Now we state the local laws for $G(z)$, which are the main tools for our proof. Given any constant $\epsilon>0$, we define the spectral domains
\be \label{tildeS}
S_0(\varsigma_1,\varsigma_2,\e):= S(\varsigma_1,\varsigma_2) \cap \left\{z = E+ \ii \eta: \eta\ge n^{-1+\epsilon}\right\},
\ee
and
\be \label{tildeS}
\wt S(\varsigma_1,\varsigma_2,\e):= S_0(\varsigma_1,\varsigma_2,\e) \cap \left\{z = E+ \ii \eta:  n^{1/2}\left( \Psi^2(z)+\frac{\phi_n}{n\eta}\right) \le n^{-\e/2}\right\}.
\ee


\begin{theorem} [Local laws]\label{LEM_SMALL} 
 Suppose $X$ has bounded support $\phi_n$ such that $ n^{-{1}/{2}} \leq \phi_n \leq n^{- c_\phi} $ for some (small) constant $c_\phi>0$. Suppose that Assumptions \ref{assm_big1} and \ref{ass:unper} hold. Fix constants $\varsigma_1 $ and $\varsigma_2>0$ as in Lemma \ref{lem_mplaw}.  Then for any fixed $\e>0$, the following estimates hold. 
\begin{itemize}
\item[(1)] {\bf Anisotropic local law}: For any $z\in \wt S(\varsigma_1,\varsigma_2,\e)$ and deterministic unit vectors $\mathbf u, \mathbf v \in \mathbb C^{\mathcal I}$,
\begin{equation}\label{aniso_law}
\left| \langle \mathbf u, G(X,z) \mathbf v\rangle - \langle \mathbf u, \Pi (z)\mathbf v\rangle \right| \prec   \phi_n + \Psi(z) .
\end{equation}

\item[(2)] {\bf Averaged local law}: For any $z \in \wt S(\varsigma_1,\varsigma_2,\e)$,  we have
\begin{equation}
 \vert m(z)-m_{c}(z) \vert + \vert m_1(z)-m_{1c}(z) \vert  + \vert m_2(z)-m_{2c}(z) \vert \prec  (n \eta)^{-1}, \label{aver_in1} 
\end{equation}
where $m$ is defined in (\ref{defn_m}) of the paper and $m_{1,2}$ are defined in (\ref{defn_m1m2}). Moreover, outside of the spectrum we have the following stronger estimate
\begin{equation}\label{aver_out1}
\begin{split}
  \vert m(z)-m_{c}(z) \vert + \vert m_1(z)-m_{1c}(z) \vert  + \vert m_2(z)-m_{2c}(z) \\
  \prec \frac{n^{-\e/4}}{n\eta}+\frac{1}{n(\kappa +\eta)} + \frac{1}{(n\eta)^2\sqrt{\kappa +\eta}},
  \end{split}
\end{equation}
uniformly in $z\in \wt S(\varsigma_1,\varsigma_2,\e)\cap \{z=E+\ii\eta: E\ge \lambda_+, n\eta\sqrt{\kappa + \eta} \ge n^\epsilon\}$, where $\kappa$ is defined in \eqref{KAPPA}. 

\item[(3)] If we have (a) (\ref{assm_3rdmoment}) of the paper holds, or (b) either $A$ or $B$ is diagonal, then the estimate \eqref{aniso_law}-\eqref{aver_out1} hold for $z\in S_0(\varsigma_1,\varsigma_2,\e)$. 
\end{itemize}
The above estimates are uniform in $z$ and any set of deterministic unit vectors of cardinality $N^{\OO(1)}$. 
\end{theorem}

\begin{proof}
 This Theorem essentially has been proved as Theorem 3.6 and Theorem 3.8 in \cite{yang2018}. But the results there are under the assumption 
$$\mathbb{E} x_{ij} =0, \ \quad \ \mathbb{E} \vert x_{ij} \vert^2  = n^{-1},$$
instead of 
\begin{align}
\max_{i,j}\left|\mathbb{E} x_{ij}\right|   \le n^{-2-\tau}, \quad  \max_{i,j}\left|\mathbb{E} | x_{ij} |^2  - n^{-1}\right|   \le n^{-2-\tau}. \label{entry_assm1}
\end{align}
assumed in Assumption \ref{assm_big1}. The second variance condition is easy to deal with: one can check that 
replacing the variance $n^{-1}$ with $n^{-1}+\OO(n^{-2-\tau})$ leads to a negligible error in each step of the proof in \cite{yang2018}. The relaxation of the mean zero assumption to the first condition in (\ref{entry_assm1}) can be handled with the centralization below.  

We decompose $X= X_1 +\E X$, where $X_1=X-\mathbb EX$ is a random matrix satisfying Assumption \ref{assm_big1} but with all entries having zero means, and $ \mathbb EX$ is a deterministic matrix with $|\E X_{ij}|\le n^{-2-\tau}$. By the above arguments, we know that \eqref{aniso_law} holds for $\wh G(z)\equiv G(X_1,z)$, where 
 $$G(X_1,z) =  \left( {\begin{array}{*{20}c}
   { - z}  & z^{1/2}A^{1/2}X_1B^{1/2}  \\
   {z^{1/2}B^{1/2}X_1^*A^{1/2}} & {- z }  \\
   \end{array}} \right)^{-1} .$$
 Then we can write
\begin{equation}\nonumber
G(X,z) =  \left(\wh G^{-1} + V\right)^{-1}, \quad V := z^{1/2}\left( {\begin{array}{*{20}c}
   {0} & A^{1/2}\E X  B^{1/2} \\
   B^{1/2} \E X^*A^{1/2} & {0}  \\
   \end{array}} \right).
 \end{equation}
Then we expand $G$ using the resolvent expansion
\begin{equation}\label{rsexp1}
G = \wh G - \wh G V \wh G + (\wh G V )^2 \wh G - (\wh G V )^3 G.
\end{equation}
We need to estimate the last three terms of the right-hand side. Using the spectral decompositions \eqref{spectral1}-\eqref{spectral2}, it is easy to verify the following estimates
\begin{equation}\label{sum_bound0}
 \sum_{a\in \cal I}\left|\wh G_{\mathbf v a}\right|^2 \prec \frac{\im \wh G_{\bv_1\bv_1}+\im \wh G_{\bv_2\bv_2}}{\eta},
\end{equation}
for any $\mathbf v =\begin{pmatrix}\bv_1 \\ \bv_2 \end{pmatrix}\in \mathbb C^{\mathcal I}$ with $\bv_1\in \C^{\cal I_1}$ and $\bv_2\in \C^{\cal I_2}$.

For any deterministic unit vectors $\mathbf u, \mathbf v \in \mathbb C^{\mathcal I}$, we have
\begin{equation}\label{exp1}
\begin{split}
& \left|\langle \mathbf u, \wh G V \wh G\mathbf v\rangle \right|  \le \sum_{b\in \mathcal I} \Big| \sum_{a\in \mathcal I}\wh G_{\mathbf u a} V_{ab}\Big| |\wh G_{b\mathbf v}|  \\
&\prec \max_{b} \Big( \sum_{a\in \mathcal I} |V_{ab}|^2 \Big)^{1/2} \sum_{b\in \mathcal I}  |\wh G_{b\mathbf v}| \\
& \prec n^{-1-\tau} \Big(\sum_{b\in \mathcal I}  |\wh G_{b\mathbf v}|^2\Big)^{1/2} \prec n^{-1-\tau}\eta^{-1/2},
\end{split}
\end{equation}
where in the second step we used (\ref{aniso_law}) for $\wh G$, and in the last step (\ref{sum_bound0}). With a similar argument, we obtain that
\begin{equation}\label{exp2}
\begin{split}
\left|\langle \mathbf u, (\wh G V )^2 \wh G  \mathbf v\rangle \right| \prec n^{-2-2\tau}\eta^{-1}.
\end{split}
\end{equation}
Combining (\ref{exp2}) with the rough bound $\|G\|=\OO(\eta^{-1})$, we get that
\begin{equation}\label{exp3}
\begin{split}
& \left|\langle \mathbf u, (\wh G V)^3 G\mathbf v\rangle \right| =\Big| \sum_{a,b}\left((\wh G V )^2 \wh G \right)_{\mathbf u a} V_{ab} G_{b\mathbf v}\Big| \\
&\prec \left(n^{-2-2\tau}\eta^{-1}\right) \eta^{-1} \sum_a\Big(\sum_{b} |V_{ab}|^2\Big)^{1/2} \le Cn^{-3/2-3\tau}\eta^{-1},\end{split}
\end{equation}
where we used $\eta \ge n^{-1}$ for $z$ in the domain $\wt S(\varsigma_1,\varsigma_2,\e)$ or $ S_0(\varsigma_1,\varsigma_2,\e)$. Plugging the estimates (\ref{exp1})-(\ref{exp3}) into (\ref{rsexp1}), we conclude that
\begin{equation}\label{truncate_compare}
\left|\langle \mathbf u, G \mathbf v\rangle - \langle \mathbf u, \wh G \mathbf v\rangle\right| \prec n^{-1-\tau}\eta^{-1/2} .
\end{equation}
for all deterministic unit vectors $\mathbf u,\mathbf v \in \mathbb C^{\mathcal I}$. This shows that \eqref{aniso_law}-\eqref{aver_out1} hold for $G$, as long as they hold for $\wh G$.
\end{proof}

As a corollary of the averaged local law, the so-called eigenvalue rigidity holds for $\mathcal Q_1$. We first define the classical locations of eigenvalues.

\begin{definition} [Classical locations of eigenvalues]
The classical location $\gamma_j$ of the $j$-th eigenvalue of $\mathcal Q_1$ is defined as
\begin{equation}\label{gammaj}
\gamma_j:=\sup_{x}\left\{\int_{x}^{+\infty} \rho_{c}(x)dx > \frac{j-1}{n}\right\}.
\end{equation}
In particular, we have $\gamma_1 = \lambda_+$.
\end{definition}

Note that for any fixed $E\le \lambda_+$, $\Psi^2(E+\ii\eta)$ is monotonically decreasing with respect to $\eta$. Hence there is a unique $\eta_l(E)$ such that 
$$n^{1/2} \left[\Psi^2(E+\ii\eta_l(E))+ \frac{\phi_n}{n\eta_l(E)}\right] =1.$$ 
 Note that by \eqref{eq_estimm} and \eqref{eq_defpsi}, we have 
\be\label{etalE}
 \eta_l(E)\sim n^{-3/4} + n^{-1/2} \left(\sqrt{\kappa_E}+\phi_n\right) ,  \quad \text{ for } \ \ E\le \lambda_+. 
\ee
For $E>\lambda_+$, we define $\eta_l(E):=\eta_l(\lambda_+)=\OO(n^{-3/4} +  n^{-1/2}\phi_n)$. 

\begin{theorem}[Rigidity of eigenvalues] \label{thm_largerigidity}
Suppose that (\ref{aver_in1}) and \eqref{aver_out1} hold. Then we have the following estimates
for any fixed constant $0<\varsigma<\varsigma_1$.
\begin{itemize}
\item[(1)] For any $E\ge \lambda_+ - \varsigma $, we have
\begin{equation}
 \vert \mathfrak n(E)-\mathfrak n_{c}(E) \vert \prec n^{-1} + (\eta_l(E))^{3/2} + \eta_l(E)\sqrt{\kappa_E} ,  \label{Kdist}
\end{equation}
where 
\begin{equation}\label{ncE}
\mathfrak n(E):=\frac{1}{N} \# \{ \lambda_j \ge E\}, \quad \mathfrak n_{c}(E):=\int^{+\infty}_E \rho_{2c}(x)dx.
\end{equation}

\item[(2)] 
For any $j$ such that $\lambda_+ - \varsigma \le \gamma_j \le \lambda_+$, we have for any fixed $\e>0$,  
\begin{equation}\label{rigidity}
\begin{split}
\vert \lambda_j - \gamma_j \vert &\prec n^{-2/3}\left( j^{-1/3}+\mathbf 1(j \le n^{1/4} \phi_n^{3/2})\right)  \\
&+ \eta_{l}(\gamma_j) + n^{2/3} j^{-2/3} \eta_l^2(\gamma_j),
\end{split}
\end{equation}
where $\eta_{l}(\gamma_j)=\OO( n^{-3/4} + n^{-5/6}j^{1/3}+\phi_n n^{-1/2})$.  

\item[(3)] If we have (a) (\ref{assm_3rdmoment}) of the paper holds, or (b) either $A$ or $B$ is diagonal, then 
\begin{equation}\label{rigidity2}
\vert \lambda_j - \gamma_j \vert \prec n^{-2/3} j^{-1/3} .
\end{equation}
\end{itemize}
\end{theorem}
\begin{proof}
The bounds \eqref{Kdist} and \eqref{rigidity2} were proved in Theorem 3.8 of \cite{yang2018}. With \eqref{Kdist}, we follow the proof of Theorem 2.13 in \cite{EKYY1} to get that 
\be\label{complicated}
\begin{split}
\vert \lambda_j - \gamma_j \vert \prec n^{-2/3}\left[ j^{-1/3} + \mathbf 1\left( j\le n^\e \left(1+n\eta_l^{3/2}(\gamma_j)\right)\right) \right] \\
+ n^{2/3}j^{-2/3} \eta_l^2(\gamma_j)+ \eta_l(\gamma_j) .
\end{split}
\ee
 With \eqref{etalE} and $\kappa_{\gamma_j} \sim (j/n)^{2/3}$, it is easy to show that
$$ n\eta_l^{3/2}(\gamma_j)\lesssim n^{-1/8} + j^{1/2}n^{-1/4} + n^{1/4} \phi_n^{3/2}.$$
 Together with \eqref{complicated}, we get \eqref{rigidity} since $\e$ can be arbitrarily small.
\end{proof}

Away from the support of $\rho_{c}$, i.e. for $\re z>\lambda_+$, the anisotropic local law can be strengthened as follows. 

\begin{theorem}[Anisotropic local law outside of the spectrum]\label{lem_localout}  
Suppose that Assumptions  \ref{assm_big1} and \ref{ass:unper} hold. Fix any $\epsilon>0$. Then for any 
\begin{equation}\label{eq_paraout}
z\in S_{out}(\varsigma_2,\epsilon):=\left\{ E+ \ii\eta: \lambda_+ + n^{-2/3+\e} + n^{-1/3+\e}  \phi_n^2  \le E \le \varsigma_2 \lambda_+,  \eta\in [0,1]\right\},
\end{equation}
and any deterministic unit vectors $\bu, \bv \in \mathbb{C}^{\mathcal I}$, we have the anisotropic local law 
\begin{equation}\label{aniso_outstrong}
\begin{split}
\left| \langle \mathbf u, G(X,z) \mathbf v\rangle - \langle \mathbf u, \Pi (z)\mathbf v\rangle \right| & \prec \phi_n + \sqrt{\frac{\im m_{2c}(z)}{n \eta}} \\
&\asymp \phi_n + n^{-1/2}(\kappa+\eta)^{-1/4} .
\end{split}
\end{equation}
\end{theorem}
\begin{proof}
The second step of \eqref{aniso_outstrong} follows from \eqref{eq_estimm}. Moreover, for $\eta\ge \eta_0:=n^{-1/2}\kappa^{1/4} + n^{-1/2+\e}\phi_n$ and $\kappa \ge n^{-2/3+\e}+ n^{-1/3+\e}  \phi_n^2 $, it is easy to verify that
$$n^{1/2} \left( \Psi^2(z)+\frac{\phi_n}{n\eta}\right)\le  n^{-\e}, \quad \frac{1}{n\eta}\lesssim \sqrt{\frac{\im m_{2c}(z)}{n \eta}} .$$
Then by \eqref{aniso_law}, we see that \eqref{aniso_outstrong} holds for $z\in S_{out}(\varsigma_2,\epsilon)$ with $\eta \ge \eta_0$. Hence it remains to prove that for $z\in S_{out}(\varsigma_2,\epsilon)$ with $0\le \eta \le \eta_0$, we have 
\begin{equation}\label{aniso_outstrong2}
\left| \langle \mathbf v, G(X,z) \mathbf v\rangle - \langle \mathbf v, \Pi (z)\mathbf v\rangle \right|  \prec \phi_n + n^{-1/2}\kappa^{-1/4},
\end{equation}
for any deterministic unit vector $\bv \in \mathbb{C}^{p+n}$. Note that \eqref{aniso_outstrong2} implies \eqref{aniso_outstrong} by polarization identity.

Now fix any $z = E+\ii \eta \in S_{out}(\varsigma_2,\epsilon)$ with $\eta \le \eta_0$. We denote $z_0:= E+ \ii \eta_0$. With \eqref{aniso_outstrong2} at $z_0$, it suffices to prove that
\be\label{prof_m}
\langle \mathbf v, \left(\Pi (z)-\Pi(z_0)\right)\mathbf v\rangle \prec \phi_n + n^{-1/2}\kappa^{-1/4},
\ee
and 
\be\label{prof_G}
\langle \mathbf v, \left(G (z)-G(z_0)\right)\mathbf v\rangle \prec \phi_n + n^{-1/2}\kappa^{-1/4}.
\ee
With \eqref{Piii}, to prove \eqref{prof_m} it is enough to show that 
\be\label{prof_m2}
|m_{1c}(z)-m_{1c}(z_0)|+|m_{2c}(z)-m_{2c}(z_0)| \prec \phi_n + n^{-1/2}\kappa^{-1/4}.
\ee
Using \eqref{eq_mdiff}, 
we obtain that
$$ |m_{1c}(z) - m_{1c}(z_0)| \lesssim \frac{z-z_0}{|z_0-\lambda_+|^{1/2}} \le  \frac{n^{-1/2+\e}}{\sqrt{\kappa}}\phi_n  + \frac{n^{-1/2}}{\kappa^{1/4}} \le \phi_n  + n^{-1/2}\kappa^{-1/4} .$$
We can deal with the $m_{2c}$ term in the same way. This proves \eqref{prof_m}.

For \eqref{prof_G}, we write $\mathbf v=\begin{pmatrix} \bv_1 \\ \bv_2\end{pmatrix}$ and use \eqref{spectral1}-\eqref{spectral2}. The upper left block gives that 
\be\label{zz0}
\begin{split}
&\left|\langle \bv_1, \left(G(z) - G(z_0)\right) \bv_1\rangle\right| \\
&\le \sum_{k = 1}^{p} \frac{\eta_0  |\langle \bv_1,{\bm \xi}_k\rangle|^2}{\left[(E-\lambda_k)^2 + \eta^2 \right]^{1/2}\left[(E-\lambda_k)^2 + \eta_0^2 \right]^{1/2}}.
\end{split}
\ee
Here and throughout the rest of this paper, we will always identify vectors $\bv_1$ and $\bv_2$ with their embeddings $\begin{pmatrix} \bv_1 \\ 0\end{pmatrix}$ and $\begin{pmatrix} 0 \\ \bv_2\end{pmatrix}$, respectively.  By \eqref{rigidity}, we have for any $k$, $E-\lambda_k\ge E - \lambda_1 \gg \eta_0$ with high probability. Using the notations in (\ref{eq_shorhand}), we can bound \eqref{zz0} by
\begin{align*}
\left|\langle \bv_1, \left(G(z) - G(z_0)\right) \bv_1\rangle\right|  \lesssim \sum_{k = 1}^{p} \frac{\eta_0  |\langle \bv_1,{\bm \xi}_k\rangle|^2}{(E-\lambda_k)^2 + \eta_0^2} = \im \mathcal G_{\mathbf v_1\mathbf v_1} (z_0)  \\
 \prec \phi_n + n^{-1/2}\kappa^{-1/4} + \im \Pi_{\mathbf v_1\mathbf v_1}(z_0)  \lesssim \phi_n + n^{-1/2}\kappa^{-1/4},
\end{align*}
where in the third step we used \eqref{aniso_law}, and in the last step we used \eqref{defn_pi}, \eqref{Piii} and \eqref{eq_estimm} to get
$$\im \Pi_{\mathbf v_1\mathbf v_1}(z_0) \lesssim \frac{\eta_0}{\sqrt{\kappa+\eta_0}}  \lesssim \phi_n +  n^{-1/2}\kappa^{-1/4}.$$
 Similarly, for the upper right block we have
\begin{align*}
& \left|\langle \bv_1, \left(G(z) - G(z_0)\right) \bv_2\rangle\right| \\
&\prec  \left|1- (zz_0^{-1})^{1/2}\right| \left|\langle \bv_1, G(z)\bv_2\rangle\right| + \sum_{k = 1}^{p\wedge n} \frac{\eta_0\left|\langle \bv_1 ,{\bm \xi}_k\rangle \langle {\bm \zeta}_k,\bv_2\rangle\right|}{|\lambda_k-z||\lambda_k-z_0|}\\
& \prec \eta_0 + \sum_{k = 1}^{p\wedge n} \frac{\eta_0\left|\langle \bv_1 ,{\bm \xi}_k\rangle \right|^2 }{|\lambda_k-z_0|^2} +   \sum_{k = 1}^{p\wedge n} \frac{\eta_0\left| \bv_2,{\bm \zeta}_k\rangle\right|^2}{|\lambda_k-z_0|^2} \\
&=\eta_0+ \im G_{\mathbf v_1\mathbf v_1}(z_0) +\im G_{\mathbf v_2\mathbf v_2}(z_0) \prec  \phi_n + n^{-1/2}\kappa^{-1/4}.
\end{align*}
The lower left and lower right blocks can be handled in the same way. This proves \eqref{prof_G}, which completes the proof. 
\end{proof}


The anisotropic local law (\ref{aniso_law}) implies the following delocalization properties of eigenvectors. 

\begin{lemma}[Isotropic delocalization of eigenvectors] \label{delocal_rigidity}
Suppose \eqref{aniso_law} and \eqref{rigidity} hold. Then we have the following estimates for any fixed constant $0<\varsigma<\varsigma_1$.
\begin{itemize}
\item[(1)] For any deterministic unit vectors $\mathbf u \in \mathbb C^{\mathcal I_1}$ and $\mathbf v  \in \mathbb C^{\mathcal I_2}$, we have 
\begin{equation}\label{delocal_claim}
 \left|\langle \mathbf u,{\bm \xi}_k\rangle \right|^2+\left|\langle \mathbf v ,{\bm \zeta}_k\rangle \right|^2\prec  n^{-1} + \eta_l(\gamma_k) \left( \frac{k}{n}\right)^{1/3} +\eta_l(\gamma_k) \phi_n, 
\end{equation}
for all $k$ such that $\lambda_+ - \varsigma \le \gamma_k \le \lambda_+$, where $\eta_l(\gamma_k)=\OO(n^{-3/4} + n^{-5/6}k^{1/3}+\phi_n n^{-1/2})$. 


\item[(2)] If we have (a) (\ref{assm_3rdmoment}) of the paper holds, or (b) either $A$ or $B$ is diagonal, then we have
\begin{equation}
\max_{k: \lambda_+ - \varsigma \le \gamma_k \le \lambda_+} \left\{ \left|\langle \mathbf u,{\bm \xi}_k\rangle \right|^2+\left|\langle \mathbf v ,{\bm \zeta}_k\rangle \right|^2\right\} \prec n^{-1}.\label{delocals}
\end{equation}
\end{itemize}
\end{lemma}

\begin{proof}
 Fix any $k$ such that $\lambda_+ - \varsigma \le \gamma_k \le \lambda_+$. By (\ref{rigidity}), we have 
\be\label{kappalambdak}\kappa_{\lambda_k} \le \kappa_{\gamma_k}+\OO_\prec\left(n^{-2/3} + \eta_{l}(\gamma_k) + n^{2/3} k^{-2/3} \eta_l^2(\gamma_k) \right). \ee
Together with \eqref{etalE}, we can verify that
\begin{align} \label{etallambdak}
\eta_l(\lambda_k) \lesssim n^{-3/4} + n^{-1/2} \left(\sqrt{\kappa_{\lambda_k}}+\phi_n\right)  \lesssim \eta_l(\gamma_k).
\end{align}
For simplicity, we denote $\eta_k:= \eta_l(\gamma_k)$. Then $z_k:=\lambda_k + \ii n^\e \eta_k \in \wt S(\varsigma_1,\varsigma_2,\e)$ with high probability for every $k$ such that $\lambda_+ - \varsigma \le \gamma_k \le \lambda_+$. 
Then using the spectral decomposition (\ref{spectral1}), we get
\begin{equation}\label{spectraldecomp}
\sum_{k=1}^{n} \frac{n^\e \eta_k  \vert \langle \mathbf v, {\bm \zeta}_k\rangle \vert^2}{(\lambda_k-E)^2+n^{2\e}\eta^2_k }  = \im\, \langle \mathbf v, {G}(z_k)\mathbf v\rangle . 
\end{equation}
Plugging into $E=\lambda_k$ and using \eqref{aniso_law}, we obtain that
\begin{equation}\label{iso_delocal_pf}
\begin{split}
\vert \langle \mathbf v, {\bm \zeta}_k\rangle \vert^2 &\le C n^\e \eta_k \im\, \langle \mathbf v, {G}(z_k)\mathbf v\rangle \\
&\prec  n^\e \eta_k \left[\im m_{2c}(\lambda_k + \ii n^{\e}\eta_k ) + \frac{1}{n^{1+\e}\eta_k} +\phi_n\right]. 
\end{split}
\end{equation} 
With \eqref{eq_estimm}, \eqref{kappalambdak} and $\kappa_{\gamma_k}\sim (k/n)^{2/3}$, we can bound that 
\begin{align*}
\im m_{2c}(\lambda_k + \ii n^{\e}\eta_k ) &\prec \left( \left( \frac{k}{n}\right)^{2/3} + n^\e\eta_k+ n^{-2/3} +  \left(\frac{n}{k}\right)^{2/3} \eta_k^2 \right)^{1/2} \\
&\lesssim \left( \frac{k}{n}\right)^{1/3} + n^{\e/2} \phi_n^{1/2}n^{-1/4} + n^{-1/6}\phi_n \lesssim \left( \frac{k}{n}\right)^{1/3}+n^{\e/2} \phi_n.
\end{align*}
where we used $\phi_n \ge n^{-1/2}$ in the last step. Plugging it into \eqref{iso_delocal_pf}, we obtain that 
\begin{equation}\nonumber
\vert \langle \mathbf v, {\bm \zeta}_k\rangle \vert^2 \prec n^{-1} + n^{3\e/2} \eta_k \left[\phi_n+\left( \frac{k}{n}\right)^{1/3}\right].
\end{equation}
 Since $\epsilon$ is arbitrary, we get \eqref{delocal_claim} for $\vert \langle \mathbf v, {\bm \zeta}_k\rangle \vert^2 $. In a similar way, we can prove \eqref{delocal_claim} for $\left|\langle \mathbf u,{\bm \xi}_k\rangle \right|^2 $. The proof for \eqref{delocals} is the same, except that we can take $z_k:=\lambda_k + \ii n^{-1+\e} \in S_0(\varsigma_1,\varsigma_2,\e)$ in this case. 
\end{proof}

Before concluding this section, we give the proof of Theorem \ref{thm_adaptiveest} of the paper.
\begin{proof}[Proof of Theorem \ref{thm_adaptiveest}] 
By Theorem \ref{thm_outlier} of the paper, under Assumption \ref{assum_supercritical},   we have that 
\begin{equation}\label{rigid_large}
\wt \lambda_{\al(i)}={g}_{2c}(-(\wt\sigma_i^a)^{-1})+\OO_{\prec}(\phi_n).
\end{equation}
Moreover, this shows that $\wt \lambda_{\al(i)} - \lambda_+ \gtrsim 1$ with high probability. 
Together with (\ref{eq_mderivative}) and Theorem \ref{lem_localout}, we obtain from \eqref{rigid_large} that 
\begin{equation}\label{eq_331}
\tsig_i^a=-m_{2c}^{-1}(\wt\lambda_{\al(i)})+\OO_{\prec}(\phi_n) = -m_2^{-1}(\wt\lambda_{\al(i)})+\OO_{\prec}(\phi_n).
\end{equation}
Since $\wt{B}$ is an $l_n$-rank perturbation of the identity matrix, with Theorem \ref{lem_localout} and (\ref{defn_m1m2}) of the paper, we obtain that 
\be\label{local again}
m_2(\wt\lambda_{\al(i)}) = \frac{1}{n}\tr \mathcal G_2(\wt\lambda_{\al(i)}) + \OO_\prec(n^{-1}l_n).
\ee
Finally, using Theorem \ref{thm_outlier} of the paper and the fact that $|\wt\lambda_\nu - \wt\lambda_{\al(i)}| \gtrsim 1$ with high probability for all $\nu\ge r+s+1$, we obtain that 
\be\label{local again2}
\frac{1}{n}\tr \mathcal G_2(\wt\lambda_{\al(i)}) = \frac{1}{n} \sum_{\nu=r+s+1}^{n} \frac{1}{\wt\lambda_\nu(\wt{\mathcal Q}_2)-\wt\lambda_{\al(i)}} + \OO_\prec(n^{-1}).
\ee
Comibing (\ref{eq_331})-(\ref{local again2}), we conclude (\ref{finite rankB}) of the paper. The estimate (\ref{finite rankA}) of the paper can be proved in the same way.
\end{proof}

\section{Outlier eigenvalues}\label{sec:ev}

In this section, we prove Theorems \ref{thm_outlier} and \ref{thm_eigenvaluesticking} of the paper. The argument is an extension of the ones in \cite[Section 4]{principal} and \cite[Section 6]{KY2013}. The proof consists of the following three steps.
\begin{itemize}
\item[(i)] We first find the permissible regions which contain all the eigenvalues of $\ctQ_1$ with high probability. 

\item[(ii)] Then we apply a counting argument to a special case, and show that each connected component of the permissible region contains the right number of eigenvalues of $\ctQ_1$. 

\item[(iii)] Finally we use a continuity argument to extend the result in (ii) to the general case using the gaps in the permissible regions. 
\end{itemize}
Our proof is more complicated than the ones in \cite[Section 4]{principal} and \cite[Section 6]{KY2013}, since we need to keep track of two types of outliers from the spikes of  $\wt A$ and $\wt B$.

\subsection{Outlier locations}\label{sec_strag1}

As in \eqref{linearize_block}, we introduce the following linearization of the spiked separable covariance matrices $\ctQ_{1,2}$:
\begin{equation*}
\widetilde{H}(X,z)= z^{1/2}\begin{pmatrix} 
0 & \wt A^{1/2} X \wt B^{1/2} \\
\wt B^{1/2} X^* \wt A^{1/2} & 0 
\end{pmatrix}, \quad z\in \mathbb C_+ \cup \mathbb R.
\end{equation*}
Note that the non-zero eigenvalues of $z^{-1/2}\wt H$ is given by 
$$\pm \sqrt{\lambda_1(\wt Q_1)}, \ \pm \sqrt{\lambda_2(\ctQ_1)}, \ \cdots \ , \ \pm \sqrt{\lambda_{p\wedge n}(\ctQ_1)}.$$
Hence it is easy to see that $x>0$ is an eigenvalue of $\ctQ_1$ if and only if
\be\label{detH}
\det\left(\wt H(X,x) - x\right)=0.
\ee
With the notations in (\ref{AOBO}) and (\ref{eq_sepamodel}) of the paper, we can write
\begin{equation}\label{Pepsilon0}
\widetilde{H}(X,z) = P H(X,z)P, 
\end{equation}
where
$$P=\begin{pmatrix} 
\left(1+ V_o^a D^a (V_o^a)^*\right)^{1/2}  & 0  \\
0 & \left(1+ V_o^b D^b (V_o^b)^*\right)^{1/2} 
\end{pmatrix}.$$
We introduce the $ (p+n) \times (r+s)$ matrix $\mathbf U$ and the  $(r+s) \times (r+s)$ diagonal matrix $\mathcal{D}$ as 
\begin{equation}\label{Depsilon0}
\mathbf{U}=
\begin{pmatrix}
V_{o}^a & 0 \\
0 & V_{o}^b
\end{pmatrix}, \quad \mathcal{D}=
\begin{pmatrix}
D^a(D^a+1)^{-1} & 0 \\
0 & D^b(D^b+1)^{-1}
\end{pmatrix}.
\end{equation}
The next lemma gives the master equation for the locations of the outlier eigenvalues. 

\begin{lemma}\label{lem_pertubation} 
If $x\ne 0$ is not an eigenvalue of $\mathcal Q_1$, then it is an eigenvalue of $\widetilde{Q}_1$ if and only if
\begin{equation}\label{masterx}
\det\left(\mathcal{D}^{-1}+x \mathbf U^* G(x) \mathbf U\right)=0.
\end{equation} 
\end{lemma}
\begin{proof}
Since $P$ is always invertible, by \eqref{detH} $x \ne 0$ is an eigenvalue of $\wt H=PHP$ if and only if    
\begin{align*}
0&=\det(PHP-x)=\det\Big(P(H-P^{-2}x)P\Big)\\
&=\det(P^2) \det(G(x))\det \left(1+xG(x)(1-P^{-2}) \right) \\
&=\det(P^2) \det(G(x))\det \left(1+xG(x) \mathbf{U} \mathcal{D} \mathbf{U}^* \right) \\
&=\det(P^2) \det(G(x)) \det \Big(1+x \mathbf{U}^* G(x) \mathbf{U} \mathcal{D}  \Big) \nonumber \\
& =\det(P^2) \det(G(x))\det(\mathcal{D})\det(\mathcal{D}^{-1}+x \mathbf{U}^* G(x) \mathbf{U}),
\end{align*}
where in the second step we used $\det(1+AB)=\det(1+BA)$. The claim then follows.
\end{proof}

Heuristically, by \eqref{masterx}, \eqref{aniso_outstrong} and \eqref{defn_pi}, an outlier location $x>\lambda_+$ almost satisfies the equation $\det(\mathcal{D}^{-1}+x \mathbf U^* \Pi(x) \mathbf U)=0 $, which is equivalent to
\begin{align*}
\prod_{i=1}^r \left( \frac{d_i^a + 1}{d_i^a} - \frac{1}{1+m_{2c}(x)\sigma_i^a} \right) \prod_{\mu=1}^s\left( \frac{d_\mu^b + 1}{d_\mu^b} - \frac{1}{1+m_{1c}(x)\sigma_\mu^b} \right)  =0.
\end{align*}
Since $(1+m_{2c}(x)\sigma_i^a)^{-1}$ is a monotonically decreasing function in $x$ for $x>\lambda_+$, the equation 
$$1+(d_i^a)^{-1}- (1+m_{2c}(x)\sigma_i^a)^{-1}=0$$ has a solution on the right of $\lambda_+$ if and only if
$$\frac{d_i^a + 1}{d_i^a} < \frac{1}{1+m_{2c}(\lambda_+)\sigma_i^a} \Leftrightarrow \wt\sigma_i^a > -m_{2c}^{-1}(\lambda_+).$$
We can do a similar calculation for $\wt \sigma_\mu^b$. This explains the conditions in (\ref{spikes}) of the paper.

\begin{proof}[Proof of Theorem \ref{thm_outlier}]
By Theorem \ref{LEM_SMALL}, Theorem \ref{thm_largerigidity} and Theorem \ref{lem_localout}, for any fixed $\epsilon>0$ we can choose a high-probability event $\Xi$ in which the following estimates hold: 
\begin{equation}\label{aniso_lawev}
\mathbf{1}(\Xi) \norm{\mathbf U^* (G(z)-\Pi(z)) \mathbf U} \le n^{\e/2}\left(\phi_n +\Psi(z)\right),\quad \text{for }z\in \wt S(\varsigma_1,\varsigma_2,\e),
\end{equation}
\begin{equation} \label{eq_bound1ev}
\mathbf{1}(\Xi) \norm{\mathbf U^* (G(z)-\Pi(z)) \mathbf U} \leq n^{\epsilon/2}\left(\phi_n + n^{-1/2}\kappa^{-1/4}\right) ,\quad \text{ for } z\in S_{out}(\varsigma_2,\epsilon),
\end{equation}
and
\begin{equation} \label{eq_bound2ev}
\mathbf{1}(\Xi)\left|\lambda_i(\mathcal Q_1)-\lambda_+\right| \leq n^{\e}\left(n^{-1/3}\phi_n^2 + n^{-2/3 }\right), \quad  \text{ for }1\le i \le \varpi . 
\end{equation}
We remark that the randomness of $X$ only comes into play to ensure that $\Xi$ holds with high probability. The rest of the proof is restricted to $\Xi$ only, and will be entirely deterministic. 

For any fixed constant $\e>0$, we define the index sets
\begin{equation} \label{eq_otau}
\mathcal{O}^{(a)}_{\e}:=\left\{i: \wt{\sigma}_i^a+m_{2c}^{-1}(\lambda_+) \geq n^\e (\phi_n + n^{-{1}/{3}}) \right\}, \quad \mathcal{O}^{(b)}_{\e}:=\{\mu: \mu \le \mu_\e\},
\end{equation}
where 
$$ \mu_\e:= \sup\left\{1\le \mu-p\le s^+: \theta_2(\wt\sigma_{\mu}^b)\ge \inf_{i\in \cal O^{(a)}_\e}\theta_1(\wt\sigma_i^a) \right\}.$$ 
Notice that we have
$$ \sup_{\mu\notin \cal O^{(b)}_\e}\left(\wt\sigma_\mu^{b} + m_{1c}^{-1}(\lambda_+) \right) \lesssim n^\e (\phi_n + n^{-{1}/{3}}),$$ 
and
$$\inf_{\mu\in \cal O^{(b)}_\e}\left(\wt\sigma_\mu^{b} + m_{1c}^{-1}(\lambda_+)\right) \gtrsim n^\e (\phi_n + n^{-{1}/{3}}). $$
Here we have defined the set of indices such that 
$$\sup_{i\notin \cal O^{(a)}_\e}\theta_1(\wt\sigma_i^a) \le  \inf_{\mu\in \cal O^{(b)}_\e}\theta_2(\wt\sigma_\mu^b),\quad  \sup_{\mu\notin \cal O^{(b)}_\e}\theta_2(\wt\sigma_\mu^b) \le \inf_{i\in \cal O^{(a)}_\e}\theta_1(\wt\sigma_i^a) .$$
This will simplify the labelling of indices: we can label the largest outliers of the $\ctQ_1$ according to the indices $i\in \cal O_\e^{(a)}$ and $\mu\in \cal O_\e^{(b)}$---the other spikes will only give smaller outliers. 

One can see that to prove Theorem \ref{thm_outlier}, it suffices to prove that for arbitrarily small constant $\e>0$, there exists a constant $C>0$ such that
\begin{equation}\label{eq_spikepf}
\begin{split}
&\mathbf 1(\Xi)\left|\wt\lambda_{\alpha(i)}-\theta_1(\wt\sigma_i^a)\right| \le Cn^{2\e}\left[\phi_n\Delta_1^2(\widetilde{\sigma}_i^a) +n^{-1/2 }\Delta_1(\widetilde{\sigma}_i^a)\right], \\ 
&\mathbf 1(\Xi)\left|\wt\lambda_{\beta(\mu)}-\theta_2(\wt\sigma_\mu^b)\right| \le Cn^{2\e}\left[ \phi_n\Delta_2^2(\widetilde{\sigma}_\mu^b)+n^{-1/2} \Delta_2(\widetilde{\sigma}_\mu^b)\right], 
\end{split}
\end{equation}
for all $i\in \cal O_{4\e}^{(a)}$ and $\mu \in  \cal O_{4\e}^{(b)}$, and 
\begin{equation}\label{eq_nonspikepf}
\begin{split}
& |\wt \lambda_{\al(i)}-\lambda_+| \le Cn^{12\e}\left(\phi_n^2 +n^{-2/3}\right), \\
& |\wt \lambda_{\beta(\mu)}-\lambda_+| \le Cn^{12\e}\left(\phi_n^2 +n^{-2/3}\right),
 \end{split}
\end{equation}
for all $i\in \{1,\cdots, r\}\setminus \cal O_{4\e}^{(a)}$ and $\mu  \in \{p+1,\cdots, p+s\}\setminus \cal O_{4\e}^{(b)}$.

\vspace{10pt}

%
\noindent{\bf Step 1:} Our first step is to show that on $\Xi$, there exist no eigenvalues outside the neighborhoods of the classical outlier locations $\theta_1(\wt\sigma_i^a)$ and $\theta_2(\wt\sigma_\mu^b)$. 
For each $1\le i \le r^+,$ we define the permissible interval  
\begin{equation*}
\begin{split}
& \rI_i^{(a)} \equiv \rI_i^{(a)}(D^a, D^b) :=\left[\theta_1(\wt\sigma_i^a)-n^{\e}w_1 (\widetilde{\sigma}_i^a), \theta_1(\wt\sigma_i^a)+n^{\e}w_1 (\widetilde{\sigma}_i^a)\right].
\end{split}
\end{equation*} 
where for simplicity we denote $w_1 (\widetilde{\sigma}_i^a):= \phi_n\Delta_1^2(\widetilde{\sigma}_i^a) +n^{-1/2 }\Delta_1(\widetilde{\sigma}_i^a)$.    
Similarly for each $1\le \mu-p \le s^+$, we define the permissible interval  
\begin{equation*}
\rI_\mu^{(b)} \equiv \rI_\mu^{(b)}(D^a, D^b):=\left[\theta_2(\wt\sigma_\mu^b)-n^{ \epsilon}w_2 (\widetilde{\sigma}_\mu^b), \theta_2(\wt\sigma_\mu^b)+n^{ \epsilon}w_2 (\widetilde{\sigma}_\mu^b)\right].
\end{equation*} 
where we denote $w_2 (\widetilde{\sigma}_\mu^b):= \phi_n\Delta_2^2(\widetilde{\sigma}_\mu^b)+n^{-1/2} \Delta_2(\widetilde{\sigma}_\mu^b)$. We then define 
\begin{equation}\label{I0}
\rI\equiv \rI({D}^a, D^b):=\rI_0 \cup \Big(\bigcup_{i \in \mathcal{O}^{(a)}_{\epsilon}}\rI_i^{(a)}\Big) \cup  \Big(\bigcup_{\mu \in \mathcal{O}^{(b)}_{\epsilon}}\rI_\mu^{(b)}\Big) ,
\end{equation}
where
$$ \rI_0:=\left[0, \lambda_+ + n^{3\e}\phi_n^2+ n^{-2/3+3\epsilon}\right].$$
We claim the following result.
\begin{lemma}\label{lem_gapI}
The complement of $\rI({D}^a, D^b)$ contains no eigenvalue of $\ctQ_1.$
\end{lemma}
\begin{proof}
By \eqref{masterx}, \eqref{eq_bound2ev} and \eqref{eq_bound1ev}, we see that $x \notin \rI_0$ is an eigenvalue of $\ctQ_1$  if and only if 
\begin{equation}\label{eq_pertubhold}
\begin{split}
&\mathbf{1}(\Xi)(\mathcal{D}^{-1}+x\mathbf U^*G(x) \mathbf U)\\
&=\mathbf{1}(\Xi)\left(\mathcal{D}^{-1}+x\bU^* \Pi(x) \bU+ \OO(\kappa^{-1/4} n^{-1/2+\epsilon/2}+n^{\e/2}\phi_n  ) \right),
\end{split}
\end{equation}
is singular. To prove the claim, it suffices to show that if $x \notin \rI$, then
\be\label{diag_big}
\begin{split}
\min\left\{ \min_{1\le i \le r}  \left| \frac{d_i^a + 1}{d_i^a} - \frac{1}{1+m_{2c}(x)\sigma_i^a} \right| ,\min_{1\le \mu - p \le s}\left| \frac{d_\mu^b + 1}{d_\mu^b} - \frac{1}{1+m_{1c}(x)\sigma_\mu^b} \right|\right\} \\
\gg n^{\e/2} (\phi_n +n^{-1/2 }\kappa_x^{-1/4} ).
\end{split}
\ee
If \eqref{diag_big} holds, then the smallest singular value of $(\mathcal{D}^{-1}+x\bU^* \Pi(x) \bU)$ is much larger than $n^{\e/2}(\phi_n + n^{-1/2 }\kappa_x^{-1/4} )$, and the matrices in \eqref{eq_pertubhold} has to be non-singular. Note that for $x >\lambda_+$, we have
\begin{align*} \left|\frac{d_i^a + 1}{d_i^a} - \frac{1}{1+m_{2c}(x)\sigma_i^a}\right| &= \left|\frac{1}{1+m_{2c}(\theta_1(\wt\sigma_i^a))\sigma_i^a}- \frac{1}{1+m_{2c}(x)\sigma_i^a}\right| \\
&\gtrsim |m_{2c}(x)-m_{2c}(\theta(\wt\sigma_i^a))| .  
\end{align*}

For any $1\le i \le r$, we claim that
\be\label{xnotinI}
\left|x - \theta_1(\wt\sigma_i^a)\right| \ge n^\e w_1(\wt\sigma_i^a) \quad \text{ for all } x\notin \rI .
\ee
 In fact, \eqref{xnotinI} is true for $i\in \cal O_\e^{(a)}$ by definition.  For $i\notin  \cal O_\e^{(a)}$, we have $\wt\sigma_i^a + m_{2c}^{-1}(\lambda_+)\le n^\e (\phi_n + n^{-1/3})$ and by \eqref{eq_derivativebound},
$$\theta_1(\wt\sigma_i(a)) - \lambda_+ \lesssim n^{ 2\e}\phi_n^2 +n^{-2/3+2\e} \ll n^{ 3\e}\phi_n^2 +n^{-2/3+3\e}.$$ 
 Moreover, by the definition of $w_1 (\widetilde{\sigma}_i^a)$ we have
$$w_1 (\widetilde{\sigma}_i^a) \lesssim n^\e \phi_n^2 + n^{-2/3+\e},\quad i\notin  \cal O_\e^{(a)} .$$
The above estimates give \eqref{xnotinI} for $i\notin  \cal O_\e^{(a)}$ by the definition of $\rI_0$.  

Now to prove \eqref{diag_big}, we first assume that there exists a constant $c>0$ such that $\theta_1(\wt\sigma_i^a) \notin [x-c\kappa_x, x+ c\kappa_x]$. Then since $m_{2c}$ is monotonically increasing on $(\lambda_+,+\infty)$, we have that
\be\nonumber
\begin{split}
|m_{2c}(x)-m_{2c}(\theta_1(\wt\sigma_i^a))| &\ge |m_{2c}(x) - m_{2c}(x\pm c\kappa_x)| \sim \kappa_x^{1/2} \\
& \gg n^{\e/2}\phi_n +n^{-1/2+\epsilon/2}\kappa_x^{-1/4},
\end{split}
\ee
where we used \eqref{eq_mderivative0} in the second step, and $\kappa_x\ge n^{3\e}\phi_n^2 + n^{-2/3+3\e}$ for $x\notin \rI_0$ in the last step. On the other hand, suppose $\theta_1(\wt\sigma_i^a) \in [x-c\kappa_x, x+ c\kappa_x]$ such that $\theta_1(\wt\sigma_i^a)-\lambda_+ \sim \kappa_x$. With \eqref{eq_derivativebound} and $\wt{\sigma}_i^a+m_{2c}^{-1}(\lambda_+) \geq n^\e \phi_n + n^{-{1}/{3}+\e}$, it is easy to show that
$$\theta_1(\wt{\sigma}_i^a) - \lambda_+ \sim \left[\Delta_1(\wt{\sigma}_i^a )\right]^4 .$$
Together with \eqref{eq_mderivative}, we have that
$$|m_{2c}'(\xi)| \sim |m_{2c}'(\theta_1(\wt{\sigma}_i^a))| \sim   \left[\Delta_1(\wt{\sigma}_i^a)\right]^{-2}$$
for $\xi \in \rI_i^{(a)}$. Since $m_{2c}$ is monotonically increasing on $(\lambda_+,+\infty)$, we get that for $x\notin \rI_i^{(a)}$,
\be\nonumber
\begin{split}
&|m_{2c}(x)-m_{2c}(\theta_1(\wt\sigma_i^a))| \ge |m_{2c}(\theta(\wt\sigma_i^a)\pm n^{ \epsilon}w_1(\wt\sigma_i^a))-m_{2c}(\theta(\wt\sigma_i^a))| \\
&\gtrsim n^\e \phi_n+n^{-1/2+\epsilon}\left[\Delta_1(\wt{\sigma}_i^a)\right]^{-1} \gtrsim n^\e \phi_n+n^{-1/2+\epsilon}\left(\theta_1(\wt{\sigma}_i^a)-\lambda_+\right)^{-1/4} \\
&\gg  n^{\e/2} \phi_n+n^{-1/2+\epsilon/2}\kappa_x^{-1/4},
\end{split}
\ee
where we used \eqref{eq_derivativebound} in the third step. The $d_\mu^b$ term can be handled in the same way. This proves \eqref{diag_big}.
\end{proof}

\noindent{\bf Step 2:} In this step we will show that each $\rI_i^{(a)}$, $i \in  \mathcal{O}_\epsilon^{(a)}$, or $\rI_\mu^{(b)}$, $\mu\in \cal O_\e^{(b)}$, contains the right number of eigenvalues of $\ctQ_1$, under a {\it special case}; see \eqref{eq_multione} below. For simplicity, we relabel the indices in $\mathcal{O}_\epsilon^{(a)}\cup \mathcal{O}_\epsilon^{(b)}$ as $\wt\sigma_1, \cdots, \wt\sigma_{r_\e}$, and call them $\e$-{\it{spikes}}. Moreover, we assume that they correspond to classical locations of outliers as $x_1,\cdots, x_{r_\e}$ (some of which are determined by $\theta_1$, while others are given by $\theta_2$), such that  
\be
x_1 \ge x_2 \ge \cdots \ge x_{r_\e}.
\ee
The corresponding permissible intervals $\rI_i^{(a)}$ and $\rI_\mu^{(b)}$ are relabelled as $\rI_i$, $1\le i \le r_\e$. In this step, we consider a special configuration $\mathbf x\equiv \mathbf x(0):=(x_1, x_2 , \cdots , x_{r_\e})$ of the outliers that is {\it independent of $n$} and satisfies 
\begin{equation}\label{eq_multione}
x_1>x_2 > \cdots > x_{r_\e}>\lambda_+.
\end{equation}

In this step, we claim that each $\mathbf{\rI}_i(\mathbf x)$, $1\le  i \le r_\e$, contains precisely one eigenvalue of $\ctQ_1$. Fix any $1\le i \le r_\e$ and choose a small $n$-independent positively oriented closed contour $\mathcal{C} \subset \mathbb{C}/[0, \lambda_+]$ that encloses $x_i$ but no other point of the set $\{x_i\}_{i=1}^{r_\e}.$ Define two functions
\begin{equation*}
h(z):=\det(\mathcal{D}^{-1}+z \bU^* G(z) \bU), \quad  l(z)=\det(\mathcal{D}^{-1}+z \bU^*\Pi(z) \bU).
\end{equation*}  
The functions $h,l$ are holomorphic on and inside $\mathcal{C}$ when $n$ is sufficiently large by \eqref{eq_bound2ev}. Moreover, by the construction of $\mathcal{C},$ the function $l$ has precisely one zero inside $\mathcal{C}$ at $x_i.$ By \eqref{eq_bound1ev}, we have  
\begin{equation*}
\min_{z \in \mathcal{C}}|l(z)| \gtrsim 1, \quad |h(z)-l(z)|=\OO(n^{\e/2}\phi_n  ).
\end{equation*}
The claim then follows from Rouch{\' e}'s theorem as long as $\e$ is taken sufficiently small. 

\vspace{10pt}

\noindent{\bf Step 3:} In order to extend the results in Step 2 to arbitrary $n$-dependent configuration $\mathbf x_n$, we shall
employ a continuity argument as in \cite[Section 6.5]{KY2013}. We first choose an $n$-independent $\bx(0)$ that satisfies \eqref{eq_multione}. We then choose a continuous ($n$-dependent) path of the eigenvalues of $D^a$ and $D^b$, which gives a continuous path of the configurations $(\bx(t):0\le t \le 1)$ that connects $\bx(0)$ and $\bx(1)=\bx_n$. Correspondingly, we have a continuous path of eigenvalues $\{\wt \lambda_i(t)\}_{i=1}^n$. We require that $\bx(t)$ satisfies the following properties.
\begin{itemize}
\item[(i)] For all $t\in [0,1]$, the eigenvalues of $D^a(t)$ and $D^b(t)$ are all non-negative.

\item[(ii)] For all $t\in [0,1]$, the number $r_\e$ of $\e$-spikes is unchanged and we denote them by $\wt \sigma_1(t),\cdots, \wt \sigma_{r_\e}(t)$. Moreover, we always have the following order of the outliers: $x_1(t) \ge x_2(t) \ge \cdots \ge x_{r_\e}(t)$.

\item[(iii)] For all $t\in [0,1]$, we denote the permissible intervals as $\rI_i(t)$. If $\rI_i(1)\cap \rI_j(1)=\emptyset$ for $1\le i < j\le r_\e$, then $\rI_i(t)\cap \rI_j(t)=\emptyset$ for all $t\in [0,1]$. The interval $\rI_0$ in \eqref{I0} is unchanged along the path. 
\end{itemize}
It is easy to see that such a path $\bx(t)$ exists. With a bootstrap argument along the path $\bx(t)$, we can prove the following lemma. 

\begin{lemma}\label{lem_cont} 
On the event $\Xi$, the estimate \eqref{eq_spikepf} holds for the configuration $\bx(1)$.
\end{lemma}
\begin{proof}
Along the path, we denote the corresponding separable covariance matrices as $\ctQ_1(t)$, with eigenvalues $\{\wt\lambda_i(t)\}$. We define $\rI(t):= \rI_0\cup (\cup_{1\le i \le r_\e}\rI_i(t))$. Combining Step 1 and Step 2 above, we obtain that on $\Xi$, 
\begin{equation}\label{eq_outliert}
\wt\lambda_i(0) \in \rI_i(0), \quad 1\le i \le r_\e, \quad \text{and} \quad \wt\lambda_i(0) \in \rI_0, \quad  i \geq r_\e .
\end{equation}
To apply a continuity argument, recall that we have shown that all the eigenvalues of $\ctQ_1(t)$ lie in $\rI(t)$ for all $t \in [0,1]$. 
Moreover, since $t \mapsto \ctQ_1(t)$ is continuous, we find that $\wt\lambda_i(t)$ is continuous in $t \in [0,1]$ for all $i$. During the proof, we shall call $i\in \{1, \cdots, r_\e\}$ a type-$a$ index if $\wt \sigma_i= \wt\sigma_{k_i}^{a}$ for some $k_i$. Otherwise, we shall call $i$ a type-$b$ index.
Note that if the $r_{\e}$ intervals are disjoint when $t=1$, then they are disjoint for all $t \in [0,1]$ by property (iii). Together with (\ref{eq_outliert}) and the continuity of $\wt\lambda_i(t),$ we conclude that 
\begin{equation*}
\wt\lambda_i(t) \in \rI_i(t), \quad 1\le i \le r_{\e}, 
\end{equation*}
for all $t \in [0,1]$. 

Now we consider the general case where some of the intervals are not disjoint. Let $\mathcal{B}$ denote the finest partition of $\{1,\cdots, r_{\e}\}$ such that $i$ and $j$ belong to the same block of $\cal B$ if $\rI_i (1) \cap \rI_j(1) \neq  \emptyset$. Denote by $B_i$ the block of $\mathcal{B}$ that contains  $i$. Note that elements of $B_i$ are sequences of consecutive integers. We now pick any $1\le i \le r_{4\e}$ and let $j \in B_i$ such that it is not the smallest index in $B_i$. Our first task is to estimate $x_{j-1}(1)-x_{j}(1)$. We claim that there exists a constant $C>0$ such that
\be\label{deltajj-1} 
x_{j} - x_{j-1} \le Cn^\e w(\wt\sigma_j) ,
\ee
where
$$ w(\wt\sigma_j):= \begin{cases}
w_1(\wt\sigma_{j}) , \ & \text{ if ${j}$ is of type-$a$} \\
w_2(\wt\sigma_{j}) , \ & \text{ if ${j}$ is of type-$b$}
\end{cases}.$$
To prove the claim, without loss of generality, we assume that $j$ is a type-$a$ index. Let $\xi \ge \wt\sigma_j$ be a value such that $\theta_1(\xi)\equiv g_{2c}(-\xi^{-1})\in \rI_j(1)$. Then we have
$$ \min_{\zeta \in [ \wt\sigma_j,\xi]}g'_{2c}(-\zeta^{-1})\left(\tsig_j^{-1}-\xi^{-1}\right) \leq g_{2c}(-\xi^{-1})-g_{2c}(-\wt\sigma_j^{-1}) \leq  Cn^{ \epsilon}w_1(\wt\sigma_j) .$$
By \eqref{eq_gderivative}, this implies that 
\begin{equation}\nonumber
\xi-\wt\sigma_j \lesssim n^\e \phi_n + n^{-1/2+\epsilon} \left( \wt\sigma_j + m_{2c}^{-1}(\lambda_+)\right)^{-1/2}.
\end{equation}
Thus we get that 
\begin{equation*}
\begin{split}
\Delta_1(\xi) &=\Delta_1(\wt\sigma_j)\left(1+\frac{\xi-\wt\sigma_j}{ \wt\sigma_j +m_{2c}^{-1}(\lambda_+)}\right)^{1/2} \\
& \le  \Delta_1(\wt\sigma_j)\left(1+\frac{n^\e \phi_n}{\wt\sigma_j +m_{2c}^{-1}(\lambda_+)}+\frac{n^{-1/2+\e}}{\left(\wt\sigma_j +m_{2c}^{-1}(\lambda_+)\right)^{3/2}}\right)  \lesssim  \Delta_1(\wt\sigma_j), 
\end{split}
\end{equation*}
where in the last step we used that that $\tsig_j \in \mathcal{O}^{(a)}_{\epsilon}$ defined in (\ref{eq_otau}). With the same arguments, we can also prove that for $\xi \le \wt\sigma_{j-1}$,
$$\begin{cases}
\Delta_1(\wt\sigma_{j-1}) \lesssim \Delta_1(\xi) , \ & \text{ if $\wt\sigma_{j-1}$ is of type-$a$ and } \theta_1(\xi)\in I_{j-1}(1),\\
\Delta_2(\wt\sigma_{j-1})  \lesssim \Delta_2(\xi) , \ & \text{ if $\wt\sigma_{j-1}$ is of type-$b$ and } \theta_2(\xi)\in I_{j-1}(1).
\end{cases}$$
Now we pick $x\in I_{j}(1)\cap I_{j-1}(1)$, and denote $\xi_1:=-m_{2c}^{-1}(x)$ and $\xi_2:=-m_{1c}^{-1}(x)$. Note that we have $x=\theta_1(\xi_1)=\theta_2(\xi_2)$, and
\be\label{sim12}
\begin{split}
\Delta_1(\xi_1)&= (m_{2c}^{-1}(\lambda_+)-m_{2c}^{-1}(x))^{1/2}  \sim \kappa_x^{1/4} \\
&\sim  (m_{1c}^{-1}(\lambda_+)-m_{1c}^{-1}(x))^{1/2}=\Delta_2(\xi_2),
\end{split}
\ee
where we used \eqref{eq_mderivative0} in the second and third steps. Then if $(j-1)$ is of type-$a$, we have
$$\Delta_1(\wt\sigma_{j-1}) \lesssim \Delta_1(\xi_1) \lesssim \Delta_1(\wt\sigma_j) .$$
If $(j-1)$ is of type-$b$, then using \eqref{sim12} we can obtain that
$$\Delta_2(\wt\sigma_{j-1}) \lesssim \Delta_2(\xi_2) \lesssim \Delta_1(\xi_1) \lesssim \Delta_1(\wt\sigma_j)  .$$
This proves the claim \eqref{deltajj-1}.

Repeating the estimate \eqref{deltajj-1} for all the remaining $j \in B_i$, since $|B_i|$ is trivially bounded by $r+s$, we obtain that
\begin{equation}\label{diamBi}
\text{diam} \left( \bigcup_{j \in B_i} \rI_j(1) \right) \leq C n^{ \epsilon} w(\wt\sigma_{\max\{j:j\in B_i\} }) \le  C n^{ \epsilon} w(\wt\sigma_i).
\end{equation}
On the other hand, since $i\in \mathcal O_{4\e}^{(a)}\cup \mathcal O_{4\e}^{(b)}$, by \eqref{eq_derivativebound} we have that 
\begin{align*}
\theta_1(\wt\sigma_i)-\lambda_+- \text{diam} \left( \bigcup_{j \in B_i} \rI_j(1) \right)&\ge c\Delta(\wt\sigma_i)^4 - C n^{\epsilon} w(\wt\sigma_i) \\
&\gg n^{ 3\e}\phi_n^2+n^{-2/3+3\e}.
\end{align*}
Hence there is a gap between the right of $\rI_0$ and the left of $ \bigcup_{j \in B_i} \rI_j(1)$. Then by (\ref{eq_outliert}), property (iii) of the path and the continuity of the eigenvalues along the path, we obtain that
\begin{equation}\label{Bit}
\wt\lambda_i(t) \in \bigcup_{j \in B_i} \rI_j(t), \quad 1\le i \le r_{4\e} ,
\end{equation}
for all $t \in [0,1]$. This proves \eqref{eq_spikepf} by \eqref{diamBi}.
\end{proof}

\noindent{\bf Step 4:} Finally, we consider the non-outlier eigenvalues, i.e. eigenvalues corresponding to $i \notin \mathcal O_{\e}^{(a)}\cup \mathcal O_{\e}^{(b)}$. 
First, we fix a configuration $\bx(0)$ satisfying (\ref{eq_multione}). By Step 2, (\ref{eq_bound2ev}) and Lemma \ref{lem_weylmodi}, we have
\begin{equation}\label{eq_upper}
\wt\lambda_i(0) \in \rI_0,\quad \text{and}\quad \wt\lambda_i (0) \geq \lambda_+ - n^\e (n^{-1/3}\phi_n^2 + n^{-2/3}).
\end{equation} 
The above two estimates give that 
$$|\wt\lambda_i (0)-\lambda_+|\le n^{3\e}\phi_n^2 + n^{-2/3+3\e}.$$ 
Next we employ a similar continuity argument as in Step 3. 
For $t \in [0,1],$ by (\ref{eq_bound2ev})  and Lemma \ref{lem_weylmodi}, we always have that 
\begin{equation}\label{iterlacing_t}
\lambda_i(t) \geq \lambda_+-n^\e (n^{-1/3}\phi_n^2 + n^{-2/3}), \quad i \geq r^++s^++1.
\end{equation}
As in the proof of Lemma \ref{lem_cont}, if $\rI_0$ is disjoint from the other $\rI_j$'s, then by the continuity of $\wt\lambda_i(t)$ and Lemma \ref{lem_gapI}, we can conclude that $\wt\lambda_i(t) \in \rI_0(t)$ for all $t \in [0,1]$. Otherwise, we again consider the partition $\mathcal B$ as in the proof of Lemma \ref{lem_cont}, and let $B_0$ be the block of $\mathcal B$ that contains $i$. With the same arguments as in the proof of Lemma \ref{lem_cont}, we can prove that 
$$\rI_0(1) \cup \left( \bigcup_{j\in B_0} \rI_{j}(1)\right)\subset [0, \lambda_+ + Cn^{3\e}(\phi_n^2 +n^{-2/3 })].$$
Then using (\ref{eq_upper}), \eqref{iterlacing_t} and the continuity of the eigenvalues along the path, we obtain that
\begin{equation}\nonumber
\left|\wt\lambda_i(t) - \lambda_+\right| \le Cn^{3\e}(\phi_n^2 + n^{-2/3}), \quad r_{\e}< i \le r+s ,
\end{equation}
for all $t \in [0,1]$. Obviously, we can apply the same arguments to $r_{4\e}< i \le r+s$ by replacing $\rI_0(1)$ with $[0, \lambda_+ + n^{-2/3 + 12\e}]$, and hence conclude \eqref{eq_nonspikepf}. This finishes the proof of Theorem \ref{thm_outlier}.
\end{proof}

\subsection{Eigenvalue Sticking}\label{sec_pf_sticking}
 
In this section, we prove the eigenvalue sticking result, i.e. Theorem \ref{thm_eigenvaluesticking} of the paper. By Theorem \ref{thm_outlier}, Theorem \ref{LEM_SMALL}, Theorem \ref{thm_largerigidity}, Theorem \ref{lem_localout} and Lemma \ref{delocal_rigidity}, for any small constants $\tau>0$ and $\e>0$, we can choose the high-probability event $\Xi$ in which \eqref{aniso_lawev}-\eqref{eq_bound2ev} and the following estimates hold:
\begin{equation}\label{eq_stickingrigi}
\mathbf{1}(\Xi)|\wt\lambda_{i}-\lambda_+| \leq n^{\epsilon/2}\left(\phi_n^2 + n^{-2/3}\right), \quad \text{ for } \ r^++s^+ +1 \le i \le \varpi, 
\end{equation}
for some fixed large integer $\varpi\ge r+s$;  
\begin{equation}\label{eq_stickingrigi2}
\begin{split}
 \mathbf{1}(\Xi)|\lambda_i-\gamma_i| &\leq  n^{-2/3+\e/2}\left( i^{-1/3}+\mathbf 1(i \le n^{1/4} \phi_n^{3/2})\right)  + n^{\e/2} \eta_l(\gamma_i)\\
&+ n^{2/3+\e/2} i^{-2/3} \eta_l^2(\gamma_i);
\end{split}
\end{equation}
for $ i \leq \tau p;$ 
\be\label{delocalav2} \left|\langle \mathbf u,{\bm \xi}_k\rangle \right|^2+\left|\langle \mathbf v ,{\bm \zeta}_k\rangle \right|^2\le n^{\e/20}\left[ n^{-1} + \eta_l(\gamma_k) \left( \frac{k}{n}\right)^{1/3} +\eta_l(\gamma_k) \phi_n\right],\ee
for $ k \leq \tau p$ and $\bu,\bv$ in some given set of deterministic unit vectors of cardinality $n^{\OO(1)}$.  Again the randomness of $X$ only comes into play to ensure that $\Xi$ holds with high probability. The rest of the proof is restricted to $\Xi$ only, and will be entirely deterministic.

Our strategy is similar to the one described at the beginning of Section \ref{sec:ev}. We first find the permissible region. For any $i$, we define the set 
\begin{equation}\label{eq_omega}
\begin{split}
\Omega_i :=
&\Big\{x \in [\lambda_{i-r-s-1}, \lambda_+ + c_0n^{2\epsilon} (\phi_n^2 +n^{-2/3})]: \\
 &\text{dist} \Big (x, \text{Spec}(\mathcal{Q}_1) \Big )>n^{-1+\epsilon} \alpha_+^{-1} + n^{\e}\eta_l(x) \Big\},
\end{split}
\end{equation}
where $\text{Spec}(\mathcal{Q}_1)$ stands for the spectrum of $\mathcal{Q}_1$ and $c_0>0$ is some small constant.

\begin{lemma}\label{prop_eigensticking} 
For $\alpha_+ \geq n^\e (\phi_n + n^{-1/3})$ and $i \leq n^{1-2\epsilon} \alpha_+^3,$ there exists a constant $c_0>0$ such that the set $\Omega_i$ contains no eigenvalue of $\ctQ_1.$
\end{lemma}
\begin{proof}
In the proof, we always use the following parameters
\begin{equation}\label{etax2}
\eta_x:=n^{-1+ \epsilon} \alpha_+^{-1} + n^{\e}\eta_l(x),\quad z_x= x+\ii \eta_x  .
\end{equation}
Suppose $x \in \Omega_i$. We now apply a similar argument as in \eqref{diag_big}. 
We first claim that for any $\bu=\begin{pmatrix}\bu_1\\ \bu_2\end{pmatrix}$ and $\bv=\begin{pmatrix}\bv_1\\ \bv_2\end{pmatrix}$ with $\bu_1,\bv_1\in \mathbb C^{\mathcal I_1}$ and $\bu_2,\bv_2\in \mathbb C^{\mathcal I_2}$, we have
\begin{align}\label{eq_bound2}
|G_{\bu \bv}(z_x)-G_{\bu \bv}(x)| & \lesssim \sum_{i=1}^2 \left[\im G_{\bu_i \bu_i}(z_x)+\im G_{\bv_i \bv_i}(z_x)\right], \ \ x\in \Omega_i.
\end{align}
As in the proof for Theorem \ref{lem_localout}, we identify vectors $\bu_i$ and $\bv_i$ with their natural embeddings in $\mathbb C^{\mathcal I}$.

We prove \eqref{eq_bound2} using \eqref{spectral1} and \eqref{spectral2}. For the terms with $G_{\bu_1\bv_2}(\cdot)$, we have
\begin{align*}
& |G_{\bu_1 \bv_2}(z_x)-G_{\bu_1 \bv_2}(x)|   \\
&\lesssim \eta_x|G_{\bu_1 \bv_2}(z_x)| + \sum_{k = 1}^{p\wedge n} \sqrt{\lambda_k}\left|\langle \bu_1,{\bm \xi}_k\rangle \langle {\bm \zeta}_k,\bv_2\rangle\right| \left|\frac{\eta_x}{(\lambda_k-x-\ii\eta_x)(\lambda_k - x)}\right|\\
&\lesssim \sum_{k} \left(\left|\langle \bu_1,{\bm \xi}_k\rangle \right|^2+\left|\langle {\bm \zeta}_k,\bv_2\rangle\right|^2\right) \frac{\eta_x}{(\lambda_k-x)^2+(\eta_x)^2} \\
&= \operatorname{Im} G_{\bu_1 \bu_1}(z_x)+\operatorname{Im} G_{\bv_2 \bv_2}(z_x),
\end{align*}
where in the second step we used $|x-\lambda_k|\ge \eta_x$ for $x\in \Omega_i$. For the rest of the cases with $G_{\bu_1\bv_1}(\cdot)$, $G_{\bu_2\bv_1}(\cdot)$ and $G_{\bu_2\bv_2}(\cdot)$, the proof is similar. 

Now using (\ref{eq_estimm}), (\ref{aniso_lawev}) and (\ref{eq_bound2}), we obtain that 
\begin{align}
& \mathcal{D}^{-1}+x \mathbf{U}^* G(x) \mathbf{U}  \label{bound_im}\\
=& \mathcal{D}^{-1}+z_x \bU^* G(z_x) \bU+x \bU^*(G(x)-G(z_x)) \bU-\ri \eta_x \bU^* G(z_x) \bU \nonumber\\
=& \mathcal{D}^{-1}+z_x \bU^* \Pi(z_x) \bU+ \OO\left(\eta_x + n^{\e/2}\Psi(z_x)+ n^{\e/2}\phi_n + \im m_{2c}(z_x) \right) \nonumber\\
= &\mathcal{D}^{-1}+z_x \bU^* \Pi(z_x) \bU+ \OO\left( n^{\e/2} \im m_{2c}(z_x) +\frac{n^{\e/2}}{n\eta_x}+ n^{\e/2}\phi_n\right) ,  \nonumber
\end{align}
where in the second step we also used that 
\be\nonumber
\max\left\{ \max_{1\le i \le r}\im \Pi_{\bv^a_i \bv^a_i},\max_{p+1\le \mu \le p+s} \im \Pi_{\bv^b_\mu \bv^b_\mu}\right\}\sim \im m_{2c}(z_x)
\ee
due to \eqref{Piii}, and in the last step we used
$$\Psi(z_x) \lesssim \im m_{2c}(z_x) + ({n\eta_x})^{-1}.$$
Therefore, by Lemma \ref{lem_pertubation}, we conclude that $x$ is not an eigenvalue of $\ctQ_1$ if 
\begin{align} 
\min\left\{ \min_{1\le j \le r}  \left| \frac{d_j^a + 1}{d_j^a} - \frac{1}{1+m_{2c}(x)\sigma_j^a} \right| ,\min_{1\le \mu - p \le s}\left| \frac{d_\mu^b + 1}{d_\mu^b} - \frac{1}{1+m_{1c}(x)\sigma_\mu^b} \right|\right\}  \nonumber \\
\gg n^{\e/2} \im m_{2c}(z_x)  +\frac{n^{\e/2}}{n\eta_x}+ n^{\e/2}\phi_n.\label{diag_big2}
\end{align}

For any $1\le j\le r$, we have
\be\label{sticking_dj}
\frac{d_j^a + 1}{d_j^a} - \frac{1}{1+m_{2c}(x)\sigma_j^a}=\frac{1+m_{2c}(x)\wt\sigma_j^a}{d_j^a(1+m_{2c}(x)\sigma_j^a)} . 
\ee
Since $i \leq n^{1-2\epsilon} \alpha_+^3$, by \eqref{eq_stickingrigi2} we have
\be\label{kappax2}
\begin{split}
 &-c_0n^{2\e}(\phi_n^2 +n^{-2/3}) \le \lambda_+-x \\
 &\lesssim \left(\frac{i}{n}\right)^{2/3} + n^{-2/3+\e/2} + n^{\e/2}\eta_l(\gamma_i) + \frac{n^{2/3+\e/2}}{ i^{2/3}} \eta_l^2(\gamma_i)\lesssim   n^{-4\e/3}\al_+^2  
 \end{split}
 \ee
for $x\in \Omega_i$, where we also used $\gamma_i\sim (i/n)^{2/3}$ and $\al_+\ge n^\e(\phi_n +n^{-1/3})$. Then by (\ref{sqroot4}) of the paper, we have
$$|m_{2c}(x)-m_{2c}(\lambda_+)| \le C n^{-2\e/3}{\al_+} 
\ll \al_+, \quad x\in \Omega_i \cap \{x: x\le \lambda_+\}.$$
and 
$$ |m_{2c}(x)-m_{2c}(\lambda_+)| \le C\sqrt{c_0}n^{ \e}\left(n^{-1/3 }+\phi_n\right) \le C\sqrt{c_0} \al_+, \ \ x\in \Omega_i \cap \{x: x> \lambda_+\}$$
for some constant $C$ independent of $c_0$. Plugging the above two estimates into \eqref{sticking_dj} and using $|\wt\sigma_j^a + m_{2c}^{-1}(\lambda_+)|\ge \al_+$, we obtain that
\be \nonumber 
\left|\frac{d_j^a + 1}{d_j^a} - \frac{1}{1+m_{2c}(x)\sigma_j^a}\right|\gtrsim \al_+ 
\ee
as long as $c_0$ is sufficiently small. On the other hand, using \eqref{eq_estimm}, \eqref{etax2} and \eqref{kappax2}, we can verify that for $x\in \Omega_i$ and $x\le \lambda_+$,
$$n^{\e/2} \im m_{2c}(z_x) +\frac{n^{\e/2}}{n\eta_x} +  n^{\e/2}\phi_n\lesssim n^{\e/2}\sqrt{\kappa_x+\eta_x} +\frac{n^{\e/2}}{n\eta_x} + n^{\e/2}\phi_n\ll \al_+,$$ 
and for $x\in \Omega_i$ and $x> \lambda_+$,
$$n^{\e/2} \im m_{2c}(z_x) +\frac{n^{\e/2}}{n\eta_x}+  n^{\e/2}\phi_n \lesssim n^{\e/2}\frac{\eta_x}{\sqrt{\kappa_x+\eta_x}} +\frac{n^{\e/2}}{n\eta_x} +  n^{\e/2}\phi_n\ll \al_+ .$$
The $d_\mu^b$ terms can be handled in the same way. This proves \eqref{diag_big2}, which further concludes the proof of Lemma \ref{prop_eigensticking}.
\end{proof}

Now we perform a counting argument for a special case. More precisely, we have the following lemma. We postpone its proof until we complete the proof of Theorem \ref{thm_eigenvaluesticking} of the paper.

\begin{lemma}\label{lem_counting} We fix a configuration $\bx\equiv \mathbf x(0):=(x_1, x_2 , \cdots , x_{r^+ + s^+})$ of the outliers that is {\it independent of $n$} and satisfies 
\begin{equation}\label{eq_multione2}
x_1>x_2 > \cdots > x_{r^+ + s^+}>\lambda_+.
\end{equation}
Moreover, we assume $d_i^a=0$ for $r^+ < i \le r$ and $d_\mu^b=0$ for $s^+ < \mu \le s$ (recall (\ref{eq_defnsigmab}) of the paper), so they will not give rise to outliers. 
Then for $\phi_n \le n^{-1/6-20\e}$ and $i \leq n^{1-4\epsilon} \alpha_+^3(0)$, we have   
\begin{equation}\label{eq_counting}
|\wt\lambda_{i+r^++s^+}-\lambda_i| \leq  n^{-1+2\epsilon} \alpha_+^{-1} + n^{3\e}\eta_l(\gamma_i) 
\end{equation}
where $\al_+(0)$ is defined for the configuration $ \mathbf x(0)$.
\end{lemma}

\begin{proof}[Proof of Theorem \ref{thm_eigenvaluesticking}]
 We first consider the case $\phi_n > n^{-1/6-20\e}$. For $i> r+s$, using Lemma \ref{lem_weylmodi} and \eqref{eq_stickingrigi2} we obtain that
\begin{align*}
|\wt\lambda_i - \lambda_{i-r^+-s^+}| &\le  n^{-2/3+\e} + n^{\e} \eta_l(\gamma_i)+  i^{-2/3}n^{-1/3+\e}\phi_n^2 \\
&\le n^{21\e}\eta_l(\gamma_i)+i^{-2/3}n^{-1/3+\e}\phi_n^2 . 
\end{align*}
For $r^+ + s^+ < i\le r+s$, we can use Lemma \ref{lem_weylmodi} and \eqref{eq_stickingrigi2} to obtain a lower bound: 
\begin{align*}
\wt\lambda_i - \lambda_{i-r^+-s^+} &\ge - \left(n^{21\e}\left( n^{-3/4} + n^{-1/2}\phi_n\right)+ n^{-1/3+\e}\phi_n^2\right) . 
\end{align*}
For the upper bound, we use \eqref{eq_stickingrigi2} and Proposition \ref{prop_eigensticking} to get
\begin{align*}
\wt\lambda_i - \lambda_{i-r^+-s^+} &\le (\lambda_1 - \lambda_{i-r^+-s^+}) + n^{-1+\epsilon} \alpha_+^{-1} + n^{\e} \left( n^{-3/4} + n^{-1/2}\phi_n\right) \\
& \le  n^{-1+\epsilon} \alpha_+^{-1} + n^{21\e} \left( n^{-3/4} + n^{-1/2}\phi_n\right)  + n^{-1/3+\e}\phi_n^2  .
\end{align*}
Later we will take $\e$ to be arbitrarily small, hence the above three estimates conclude the proof for the case $\phi_n > n^{-1/6-20\e}$. 

For the rest of the proof, we always assume that $\phi_n \le n^{-1/6-20\e}$. First, we consider the case with $\alpha_+ \geq n^{2\e}(\phi_n + n^{-1/3})$ and $i \leq n^{1-4\epsilon} \alpha_+^3$. We shall apply a similar continuity argument as in Step 4 of the proof  in Section \ref{sec_strag1}. We define 
\begin{align*}
\wt {\rI}_0 :=&\left\{x\in [0, \lambda_+ + c_0n^{2\e} (\phi_n^2+n^{-2/3})]:\right. \\
&\left. \text{dist}\left(x,\text{Spec}(\mathcal{Q}_1)\right) \le n^{-1+\epsilon} \alpha_+^{-1} + n^{\e}\eta_l(x)\right\}.
\end{align*} 
Note that $\wt {\rI}_0$ is a union of connected intervals. We again define a continuous path of configurations $\bx(t)$ such that $\bx(0)$ satisfies \eqref{eq_multione2} and $\bx(1)$ is the configuration we are interested in. Moreover, we can choose the path such that 
$$\inf_{t\in [0,1]}\al_+(t) \ge \al_+\equiv \al_+(1),$$
where $\al_+(t)$ is defined for the configuration $\mathbf x(t)$ at time $t$. 
Note that by interlacing, Lemma \ref{lem_weylmodi}, we have
\begin{equation}\label{interlacing_t2}
\lambda_i\le \wt\lambda_{i}(t) \leq \lambda_{i-r-s}. 
\end{equation}
By Lemma \ref{lem_counting} and Lemma \ref{prop_eigensticking}, we know
\begin{equation}\nonumber
|\wt\lambda_{i+r^++s^+}(0)-\lambda_i| \leq n^{-1+2\epsilon} \alpha_+^{-1}(0) + n^{3\e}\eta_l(\gamma_i),
\ee
and 
\be
\text{dist}\left(\wt \lambda_{i+r^++s^+}(t),\text{Spec}(\mathcal{Q}_1)\right) \le Cn^{-1+\epsilon} \alpha_+^{-1}(1) + n^{3\e/2}\eta_l(\gamma_i),
\end{equation}
where we used that $\al_+(t)\ge \al_+(1)$ and 
\be\label{etali} 
\eta_l(\wt \lambda_{i+r^++s^+}(t)) \ll n^{\e/2} \eta_l(\gamma_i)
\ee
since $\wt \lambda_{i+r^++s^+}(t)$ satisfies \eqref{interlacing_t2} and $\lambda_i$ satisfies \eqref{eq_stickingrigi2}.  In addition, by continuity of the eigenvalues, we know that $\wt \lambda_{i+r^++s^+}(t)$ is in the same connected component of $\wt {\rI}_0$ as $\wt \lambda_{i+r^++s^+}(0)$. Let $B_{i}$ be the set of $1\le j \le p$ such that $\lambda_i$ and $\lambda_j$ are in the same connected component of $\wt {\rI}_0$. Then we conclude that for all $t\in [0,1]$,
\begin{align*}
& \wt \lambda_{i+r^++s^+}(t) \\
& \in \bigcup_{j\in B_{i }:|i +r^++s^+ - j | \le r+s} \left[\lambda_j -\left( n^{-1+2\epsilon} \alpha_+^{-1} + n^{3\e}\eta_l(\gamma_j)\right) , \lambda_j + \left( n^{-1+2\epsilon} \alpha_+^{-1} + n^{3\e}\eta_l(\gamma_j)\right)\right]\\
& \subset  \bigcup_{j\in B_{i }:|i +r^++s^+ - j | \le r+s} \left[\lambda_j -\left( n^{-1+2\epsilon} \alpha_+^{-1} + n^{4\e}\eta_l(\gamma_i)\right) , \lambda_j + \left( n^{-1+2\epsilon} \alpha_+^{-1} + n^{4\e}\eta_l(\gamma_i)\right)\right],
\end{align*}
where we again used estimates that are similar to \eqref{etali}. This gives that
\be \left|\wt \lambda_{i+r^++s^+}(1)-\lambda_i\right| \le 2(r+s)  \left( n^{-1+2\epsilon} \alpha_+^{-1} + n^{4\e}\eta_l(\gamma_i)\right) .\label{case1_extreme}
\ee
when $\alpha_+ \geq n^{2\e}(\phi_n + n^{-1/3})$ and $i \leq n^{1-4\epsilon} \alpha_+^3$.


Finally we consider the cases: $\alpha_+ < n^{2\e}(\phi_n + n^{-1/3})$, or $i >n^{1-4\epsilon} \alpha_+^3$. Suppose first that $\alpha_+<n^{2\e}(\phi_n + n^{-1/3}).$ Then by the assumption $\al_+ \ge n^{c_0} \phi_n$ in Theorem \ref{thm_eigenvaluesticking} of the paper, as long as $\e<c_0/4$ we get
$$ \phi_n =\OO(n^{-1/3-2\e}),\quad \al_+ =\OO(n^{-1/3+2\e}).$$
Then by \eqref{eq_stickingrigi}, (\ref{eq_stickingrigi2}) and Lemma \ref{lem_weylmodi}, we obtain that 
\begin{equation}\label{case1_extreme2}
\begin{split}
|\wt\lambda_{i+r^++s^+}-\lambda_i|  &\leq (r+s)\left( n^{-2/3+\epsilon} + n^{\e}\eta_l(\gamma_i)\right) \\
&\leq C\left( n^{-1+3\e}\al_+^{-1} + n^{\e}\eta_l(\gamma_i)\right). 
\end{split}
\end{equation}
On the other hand, suppose $i >n^{1-4\epsilon} \alpha_+^3$ with $\alpha_+ \ge n^{2\e}(\phi_n + n^{-1/3})$. Then obviously we have $\al_+ > r+ s$, and we can apply (\ref{eq_stickingrigi2}) and Lemma \ref{lem_weylmodi} to get that
\begin{equation}\label{case1_extreme3}
\begin{split}
|\wt\lambda_{i+r^++s^+}-\lambda_i| & \leq C\left( i^{-1/3} n^{-2/3+\epsilon/2} + n^{\e/2}\eta_l(\gamma_i)  \right) \\
&\leq C\left( n^{-1+2\e}\al_+^{-1} + n^{\e/2}\eta_l(\gamma_i)\right). 
\end{split}
\end{equation}
Here in the application of (\ref{eq_stickingrigi2}), we used $\phi_n \le n^{-1/6-20\e}$ to simplify the expression.
Combining \eqref{case1_extreme}-\eqref{case1_extreme3}, we conclude the proof of (\ref{eq_stickingeq}) of the paper for the case $\phi_n \le n^{-1/6-20\e}$.

For (\ref{eq_stickingeq_strong}) of the paper, the proof is exactly the same, except that we can set $\eta_l(E)=n^{-1}$ by using the stronger anisotropic local law \eqref{aniso_law} for $z\in S_0(\varsigma_1, \varsigma_2, \e)$ and the stronger rigidity estimate \eqref{rigidity2}.
\end{proof}

The strategy for the proof of Lemma \ref{lem_counting} is an extension of the one for the proof of \cite[Proposition 6.8]{KY2013}. We remark that in \cite{KY2013}, the results are only proved for the eigenvalues near the edge with $i \leq (\log n)^{C \log \log n},$ for some constant $C>0.$  Here we will prove that the same results hold further into the bulk. 

\begin{proof}[Proof of Lemma \ref{lem_counting}]
 Note that under the condition $\phi_n \le n^{-1/6-20\e}$, \eqref{eq_stickingrigi2} reduces to
\begin{equation}\label{eq_stickingrigi2add}
\begin{split}
 \mathbf{1}(\Xi)|\lambda_i-\gamma_i| &\leq  i^{-1/3} n^{-2/3+\e/2} + 2n^{\e/2} \eta_l(\gamma_i).
\end{split}
\end{equation}
%
First suppose $j$ is large enough such that 
$$j > \min\{n^{1/4-4\e}, n^{-1/2-5\e} \phi_n^{-3} \}. $$
Then by \eqref{eq_stickingrigi2add}, we have for $|i-j|=\OO(1)$,
$$ |\lambda_i-\gamma_i| \leq   i^{-1/3} n^{-2/3+\e/2}+ 2n^{\e/2} \eta_l(\gamma_i) \le n^{5\e/2}\eta_l(\gamma_j). $$ 
Together with interlacing, Lemma \ref{lem_weylmodi}, we immediately obtain \eqref{eq_counting}. Hence in the following proof, we assume that 
\be\label{assum_ind1} 
j \le j_0 \equiv \min\{n^{1/4-4\e}, n^{-1/2-5\e} \phi_n^{-3} \}. 
\ee
Note that for this lemma, we have $\al_+\equiv \al_+(0) \sim 1$.

 In the first step, we group together the eigenvalues that are close to each other. More precisely, let $\mathcal{A}=\{A_k\}$ be the finest partition of $\{1,\cdots, p\}$ such that $i<j$ belong to the same block of $\mathcal A$ if 
$$|\lambda_{i}-\lambda_{j}| \leq \delta(j):=n^{- 1 + 7\epsilon/6} \alpha_+^{-1} + n^{7\e/6}\eta_l(\gamma_j)  . $$ 
Note that each block $A_k$ of $\mathcal A$ consists of a sequence of consecutive integers. We order the blocks in the descending order, i.e. if $k<l$  then $\lambda_{i_k}>\lambda_{i_l}$ for all $i_k \in A_k$ and $i_l \in A_l$. 

We first derive a bound on the sizes of the blocks near the edge with $i\le j_0$. 
We define $k^*$ such that $ j_0\in A_{k^*}$. For any $k\le k^*$, we take $i < j $ such that $i$ and $j$ both belong to the block $A_k$. Then by (\ref{eq_stickingrigi2add}) and Lemma \ref{lem_weylmodi}, we find that for some constants $c,C>0$,
\begin{align*}\nonumber
& c \left[ (j/n)^{2/3}-(i/n)^{2/3}  \right] - C\left(i^{-1/3} n^{-2/3+\epsilon/2} + n^{\e/2}\eta_l(\gamma_j) \right)   \\
&\leq \lambda_{i}-\lambda_{j}  \leq C(j-i)\left(n^{- 1 + 7\epsilon/6} \alpha_+^{-1} + n^{7\e/6}\eta_l(\gamma_j) \right).
\end{align*}
 With the elementary inequalities
\be\label{elementary}
j^{-1/3}(j - i)\le j^{2/3}-i^{2/3} \le i^{-1/3}(j-i), \quad 1\le i\le j,
\ee
we obtain that 
\begin{equation}\nonumber
\left(j^{-1/3}-C \left(n^{- 1/3 + 7\epsilon/6} \alpha_+^{-1} + n^{2/3+7\e/6}\eta_l(\gamma_j) \right)\right) (j-i) \leq C i^{-1/3} n^{\epsilon/2}  .
\end{equation}
Now using \eqref{etalE}, we conclude that if $i$ and $j$ satisfy 
\be\label{extraij}  1\le i \le j \le n^{\e/4}j_0,\ee
then we have
\be\label{boundij} j-i \le C(j/i)^{1/3}n^{\e/2} .\ee
  With this estimate, we claim that 
\begin{equation}\label{eq_sizeeigen}
|A_k| \leq Cn^{3\e/4}  \quad \text{for } \ k=1,\cdots, k^*,
\end{equation}
and for any given $i_k \in A_k$,
\begin{equation}\label{rigidity_size}
|\lambda_{i}-\gamma_{i_k}|\leq i^{-1/3} n^{-2/3+ \epsilon} + n^{\e}\eta_l(\gamma_i)  \quad \text{for all }\ i \in A_k .
\end{equation}
 To prove these two estimates, we first assume that \eqref{extraij} holds. We denote
$$\quad m_k :=\max_{i\in A_k} i, \quad l_k :=\min_{i\in A_k} i .$$
If $i\in A_k$ satisfies $i\ge m_k/2$, then \eqref{boundij} gives that
$m_k-i\le Cn^{\e/2}$. Using \eqref{elementary}, we get that
$$|\gamma_i-\gamma_{m_k}|\le Cn^{\e/2}i^{-1/3}n^{-2/3}.$$
On the other hand, if $i\in A_k$ satisfies $i\le m_k/2$, then \eqref{boundij} gives that
$m_k - i\le m_k\le Cn^{3\e/4}$. Thus we get
$$|\gamma_i-\gamma_{m_k}|\le |\gamma_1-\gamma_{m_k}| \le Cn^{-2/3+\e/2} \le Ci^{-1/3}n^{-2/3+3\e/4}.$$
Together with (\ref{eq_stickingrigi2add}) and \eqref{assum_ind1}, we obtain that
\begin{align*}
& |\lambda_{i}-\gamma_{i_k}|  \le |\lambda_{i}-\gamma_{i}| +|\gamma_{i}-\gamma_{m_k}|+ |\gamma_{m_k}-\gamma_{i_k}|  \\
&\le  C\left[n^{\e/2}\eta_l(\gamma_i) + n^{3\e/4}i^{-1/3}n^{-2/3}\right] \le i^{-1/3}n^{-2/3+\e} + n^{\e}\eta_l(\gamma_i) .
\end{align*}
Combining the two cases, we obtain \eqref{eq_sizeeigen} and \eqref{rigidity_size}. It remains to prove that \eqref{extraij} holds for $i,j \in A_{k^*}$. In fact, if there is $j\in A_{k^*}$ such that $j\ge n^{\e/4}j_0$, then we can find $j'\in A_{k^*}$ such that  $n^\e \le j'- j_0 \le 2n^{\e}$. In other words, we have that $j'$ and $\al$ both satisfy \eqref{extraij}, but $| j'- j_0|\ge n^\e$ which contradicts \eqref{eq_sizeeigen}.  

We are now ready to give the main argument. For any $1\le k \le k^*$, we denote 
\begin{equation*}
a^k:=\min_{i \in A_k} \lambda_{i} = \lambda_{m_k}, \quad b^k:=\max_{i \in A_k} \lambda_{i}=\lambda_{l_k}. 
\end{equation*}
We introduce a continuous path as 
\begin{equation}\label{eq_defnx0x1}
x_t^k=(1-t)\left(a^k - \delta(m_k)/3\right)+t\left(b^k  + \delta(l_k)/3\right)t, \quad t \in [0,1].
\end{equation}
Note that $x_0^k= a^k - \delta(m_k)/3$ and $x_1^k= b^k  + \delta(l_k)/3$. The interval $[x_0^k, x_1^k]$ contains precisely the eigenvalues of $\mathcal Q_1$ that are in $A_k$, and the endpoint $x_0^k$ (or $x_1^k$) is at a distance at least of the orders $\delta(m_k)/3$ (or $\delta(l_k)/3$) from any eigenvalue of $\mathcal Q_1$. 

In order to avoid problems with exceptional events, we add some randomness to $D^a$ and $D^b$. Recall that their eigenvalues satisfy \eqref{eq_multione2}. 
Let $\Delta$ be an $(r+s) \times (r+s)$ Hermitian random matrix, which only has nonzero entries in the upper left $r\times r$ block and the lower right $s\times s$ block. Moreover, we assume the upper triangular entries of $\Delta$ are independent and have an absolutely continuous law supported in the unit disk. Following the notations in \eqref{Pepsilon0} and \eqref{Depsilon0}, for any $\omega>0$, we define $D^{a,\omega}$ and $D^{b,\omega}$ such that
\begin{equation*}
(\widetilde{\mathcal{D}}^{\omega})^{-1}:=\mathcal{D}^{-1}+\omega \Delta .
\end{equation*}
Correspondingly, we define $\ctQ^{\omega}_{1,2}$ and 
$$\wt H^{\omega} = P^\omega HP^\omega, \quad P=\begin{pmatrix} 
\left(1+ V_o^a D^{a,\omega} (V_o^a)^*\right)^{1/2}  & 0  \\
0 & \left(1+ V_o^b D^{b,\omega} (V_o^b)^*\right)^{1/2} 
\end{pmatrix}.$$
We shall take $\omega$ to be sufficiently small, say $\omega\le \wt\e e^{-n}$ for some $\wt\e \to 0$. From now on, we use ``almost surely" to mean almost surely with respect to the randomness of $\Delta.$ Our main goal is to prove the following proposition. 

\begin{proposition} \label{prop_atleat} For each $\omega>0,$ almost surely, there are at least $|A_k|$ eigenvalues of $\ctQ^{\omega}_1$ in $[x_0^k, x_1^k]\setminus \operatorname{Spec}(\mathcal Q_1)$. 
\end{proposition}

Before proving Proposition \ref{prop_atleat}, we first show how to use it to conclude Lemma \ref{lem_counting}. By taking $\omega \rightarrow 0$ and using a standard perturbation argument, we deduce that 
\be\label{counting_omega}
\text{$\ctQ_1$ has at least $|A_k|$ eigenvalues in $[x_0^k, x_1^k]$ for $1\le k \le k^*$}.
\ee 
Next, we will use the standard interlacing argument to show that $\ctQ_1$ has at most $|A_k|$ eigenvalues  in $[x_0^k, x_1^k]$.  By Lemma \ref{lem_weylmodi}, we find that there are at most $|A_1|+r^++s^+$ eigenvalues of $\ctQ_1$ in $[x_0^1, \infty)$ (recall that by the assumption of  Lemma \ref{lem_counting}, we have a rank $(r^+ + s^+)$ perturbation).  Hence, by Theorem \ref{thm_outlier} and \eqref{counting_omega}, there are exactly $|A_1|$ eigenvalues of $\ctQ_1$ in $[x_0^1, x_1^1].$ Repeating this argument, we can show that $\ctQ_1$ has exact $|A_k|$ eigenvalues in $[x_0^k, x_1^k]$ for all $k=2,\cdots, k^*$. Moreover, by (\ref{eq_sizeeigen}), we find that for any $i \in A_k$, 
\begin{align*}
\sup\Big\{ |x-\lambda_{i}|: i \in A_k, x \in [x_0^k, x_1^k]\Big\} & \leq Cn^{3\e/4} \left(n^{- 1 + 7\epsilon/6} \alpha_+^{-1} + n^{7\e/6}\eta_l(\gamma_{m_k})\right)  \\
& \le n^{- 1 + 2\e} \alpha_+^{-1} + n^{2\e}\eta_l(\gamma_{m_k}). 
\end{align*}
Together with $\eta_l(\gamma_{m_k}) \le n^{\e}\eta_l(\gamma_{i})$, we conclude the proof of Lemma \ref{lem_counting}. 
\end{proof}

The proof of Proposition \ref{prop_atleat} is very similar to the argument in \cite[Section 6.4]{KY2013}. We only prove the part that is different from the proof there, and omit the rest of the details. 

\begin{proof}[Proof of Proposition \ref{prop_atleat}] 
For $x \notin \text{spec}(\mathcal Q_1),$ we define
\begin{equation*}
M^{\omega}(x):={\mathcal{D}}^{-1}+\omega\Delta+x \Ub^* G(x) \Ub.
\end{equation*} 
By Lemma \ref{lem_pertubation}, we know that $x \in \operatorname{Spec}(\ctQ_1^{\omega})\setminus\operatorname{Spec}(\mathcal Q_1)$ if and only if $M^\omega (x)$ is singular. 


We split $G$ into ${P}_{A_k} G + {P}_{A_k^c} G$ according to whether 
$i \in A_k$ or $i \notin A_k$ in the spectral decompositions \eqref{spectral1} and \eqref{spectral2}. For example, the upper left blocks of ${P}_{A_k} G$ and ${P}_{A_k^c} G$ are defined as
$$ {P}_{A_k} G_{ij}(x) := \sum_{l \in A_k} \frac{ {\bm \xi}_l(i) {\bm \xi}_l^*(j)}{\lambda_l-x},\ \quad \ {P}_{A_k^c} G_{ij}(x) := \sum_{l \notin A_k} \frac{{\bm \xi}_l(i) {\bm \xi}_l^*(j)}{\lambda_l-x}.$$
Similarly, we can define the other three blocks of ${P}_{A_k} G$ and ${P}_{A_k^c} G$. Let $x \in [x_0^k, x_1^k]$ and
$$z_x= x+ \ii \eta_x, \quad \eta_x:=n^{- 1 + 7\epsilon/6} \alpha_+^{-1} + n^{7\e/6}\eta_l(x).$$ 
Then given any deterministic vectors $\bu=\begin{pmatrix}\bu_1\\ \bu_2\end{pmatrix}$ and $\bv=\begin{pmatrix}\bv_1\\ \bv_2\end{pmatrix}$, similar to (\ref{eq_bound2}) we have
\begin{align}\label{eq_bound22}
|P_{A_k^c}G_{\bu \bv}(z_x)-P_{A_k^c}G_{\bu \bv}(x)| & \lesssim \sum_{i=1}^2 \left[\im G_{\bu_i \bu_i}(z_x)+\im G_{\bv_i \bv_i}(z_x)\right].
\end{align}
For example, for the terms with $G_{\bu_1\bv_2}(\cdot)$, we have
\begin{align*}
& |P_{A_k^c}G_{\bu_1 \bv_2}(z_x)-P_{A_k^c}G_{\bu_1 \bv_2}(x)|   \\
&\lesssim \eta_x|G_{\bu_1 \bv_2}(z_x)| +\sum_{l \notin A_k}\sqrt{\lambda_l}\left|\langle \bu_1,\bxi_l\rangle \langle \bzeta_l,\bv_2\rangle\right| \left|\frac{\eta_x}{(\lambda_l-x-\ii\eta_x)(\lambda_l - x)}\right|\\
&\lesssim \sum_{l\notin A_k} \left(\left|\langle \bu_1,\xi_l\rangle \right|^2+\left|\langle \zeta_l,\bv_2\rangle\right|^2\right) \frac{\eta_x}{(\lambda_l-x)^2+(\eta_x)^2} \\
&\le \operatorname{Im} G_{\bu_1 \bu_1}(z_x)+\operatorname{Im} G_{\bv_2 \bv_2}(z_x),
\end{align*}
where in the second step we used that $|x-\lambda_l| \gtrsim \eta_x$ for any $x \in [x_0^k, x_1^k]$ and $l\notin A_k$. For the rest of the cases with $G_{\bu_1\bv_1}(\cdot)$, $G_{\bu_2\bv_1}(\cdot)$ and $G_{\bu_2\bv_2}(\cdot)$, the proof of \eqref{eq_bound22} is similar. 
Moreover, we claim that 
\begin{equation}\label{eq_rxorder}
\left| P_{A_k}G_{\bu\bv}(z_x) \right| \le n^{-\epsilon/3}.
\end{equation}
 For example, we have
\begin{equation}
\begin{split}
&\left| \sum_{j \in A_k} \frac{\langle {\mathbf{u}}_1, {\bm \xi}_j \rangle \langle {\bm \xi}_j, {\mathbf{v}}_{1} \rangle}{\lambda_{j}-z_x} \right| \\
&\leq  \eta_x^{-1} n^{\e/20}\sum_{j\in A_k}\left[ n^{-1} + \eta_l(\gamma_j) \left( \frac{j}{n}\right)^{1/3} +\eta_l(\gamma_j) \phi_n\right] \ll n^{ - \epsilon/3},
\end{split}
\end{equation}
where in the first step we used \eqref{delocalav2}, and in the second step we used \eqref{eq_sizeeigen} and \eqref{assum_ind1} such that
$$ \eta_l(\gamma_j) \left( {j}/{n}\right)^{1/3} +\eta_l(\gamma_j) \phi_n\le n^{-1/6}\eta_x.$$
For the rest of the cases with $G_{\bu_1\bv_2}(\cdot)$, $G_{\bu_1\bv_2}(\cdot)$ and $G_{\bu_2\bv_2}(\cdot)$, the proof of \eqref{eq_rxorder} is similar. Then by a discussion similar to \eqref{bound_im}, we have
\begin{align}
&M^{\omega}(x)=x\mathbf U^* P_{A_k}G(x) \mathbf U + x \Ub^* (P_{A_k^c}G(x)-P_{A_k^c}G(z_x)) \Ub \nonumber  \\
&\quad + (z_x+(x-z_x))\Ub^* G(z_x)\Ub-x\Ub^* P_{A_k}G(z_x)\Ub +\mathcal D^{-1}+\omega \Delta \nonumber \\
& = x\mathbf U^* P_{A_k}G(x) \mathbf U  +\mathcal D^{-1}+\omega \Delta +z_x \bU^* \Pi(z_x) \bU+ R_0(x) \nonumber\\
& = x\mathbf U^* P_{A_k}G(x) \mathbf U  +\mathcal D^{-1}+\omega \Delta + \lambda_+ \bU^* \Pi(\lambda_+) \bU+ R(x), \label{eq_m1}
\end{align}
where 
$$R_0(x)=\OO\left(\eta_x + n^{\e/2}\Psi(z_x)  + n^{\e/2}\phi_n + \im m_{2c}(z_x) + n^{-\e/3}\right) =\OO\left(n^{-\e/3}\right)$$
and 
$$R(x)=R_0(x) + \OO(\sqrt{\kappa_x + \eta_x}) =\OO\left(n^{-\e/3}\right).$$
Moreover, $R(x)$ is real (since all the other terms in the line \eqref{eq_m1} are real), continuous in $x$ on the extended real line $ \overline{\mathbb R}$, and independent of $\Delta$.

The rest of the proof follows from a continuity argument, which is exactly the same as the proof in \cite[Section 6.4]{KY2013} between (6.27) and (6.28). We remark that the small $\omega \Delta$ is used only in this proof to avoid some problems with exceptional events. We omit the details. This completes the proof of Proposition \ref{prop_atleat}.
\end{proof}

\section{Outlier eigenvectors}\label{sec:eveout}

In this section, we study the outlier eigenvectors. More precisely, we prove Theorem \ref{thm_eveout} of the paper under the following stronger assumption. 

\begin{assumption}\label{assu_strong} For some fixed small constant $\tau>0,$ we assume that for $\al( i) \in S $ and $\beta(\mu)\in S$, 
\begin{equation}\label{eq_spikeassuastr}
\wt{\sigma}^a_{i}+{m_{2c}^{-1}(\lambda_+)} \geq n^{-1/3+\tau}+n^\tau \phi_n,\quad \wt{\sigma}^b_{\mu}+{m_{1c}^{-1}(\lambda_+)} \geq n^{-1/3+\tau}+n^\tau \phi_n.
\end{equation}
\end{assumption}  
 The necessary argument to remove this assumption will be given in Section \ref{sec_nonoutliereve} after we complete the proof of Theorem \ref{thm_noneve}, since we need the delocalization bounds there. Thus the main goal of this section is to prove the following weaker proposition.

\begin{proposition}\label{prop_outev}
Suppose the assumptions in Theorem \ref{thm_eveout} of the paper hold. Then under Assumption \ref{assu_strong}, we have that for all $i,j=1, \cdots, p$,  
\be\label{eq_spikedvector}
\begin{split}
&\left| \langle \bv_i^a, \cal P_S\bv_j^a\rangle- \delta_{ij}\mathbf 1(\al(i)\in S)\frac{1}{\wt{\sigma}_i^a} \frac{g_{2c}'(-(\wt{\sigma}_i^a)^{-1})}{g_{2c}(-(\wt{\sigma}_i)^{-1})} \right|  \prec  \sqrt{\Upsilon(i,S) \Upsilon(j,S)}  \\
&+  \mathbf 1(\al(i)\in S,\al(j)\notin S) \Delta_1(\widetilde{\sigma}_i^a)\left[  \frac{\phi_n}{ \delta^{1/2}_{\al(j)}(S)}+\frac{\psi_{1}(\wt\sigma_j^a)\Delta_1(\widetilde{\sigma}_j^a) }{\delta_{\al(j)}(S)} \right]  \\
&  + (i\leftrightarrow j),
\end{split}
\ee
where $(i\leftrightarrow j)$ denotes the same terms but with $i$ and $j$ interchanged, and \nc
\begin{align*}
  \Upsilon(i,S) := \mathbf 1(\al(i)\in S) \psi_{1}(\wt\sigma_i^a) + \mathbf 1(\al(i)\notin S)   \frac{\phi_n^2}{\delta_{\al(i)}(S)}  +  \frac{\psi_{1}^2(\wt\sigma_i^a)\Delta_1^2(\widetilde{\sigma}_i^a) }{\delta^2_{\al(i)}(S)} .
\end{align*}
\end{proposition}


The rest of this section is devoted to proving Proposition \ref{prop_outev}. Our strategy is an extension of the one in \cite[Section 5]{principal}. But there is additional complication in our case, because we need to simultaneously handle the outliers caused by the spikes of $\wt{B}$.

\subsection{Non-overlapping condition}\label{sec_eve_green}

We first prove Proposition \ref{prop_outev} under the following additional \emph{non-overlapping condition}. We will remove it later in Section \ref{sec_remove}.   
\begin{assumption}\label{ass:nonver}
 For some fixed small constant $\wt\tau>0,$  we assume that for all $\al(i)\in S$ and $\beta(\mu)\in S$, 
\begin{equation*}
 \delta_{\al(i)}(S) \geq \left[\Delta_1(\widetilde{\sigma}_i^a)\right]^{-1} n^{-1/2+\wt\tau} + n^{\wt\tau}\phi_n,  \end{equation*}
and
\begin{equation*}
  \delta_{\beta(\mu)}(S) \geq \left[\Delta_2(\widetilde{\sigma}_\mu^b)\right]^{-1} n^{-1/2+\wt\tau}+ n^{\wt\tau}\phi_n. 
\end{equation*}
\end{assumption}
\begin{remark}\label{remark_nonover}
This condition is actually a generalization of the second condition in (\ref{eq_sepe}) of the paper. Note that for $1\le i\le r^+$, using \eqref{eq_derivativebound}, \eqref{eq_mderivative} and \eqref{eq_gderivative}, we have
$$\delta_{\al(i), \al(j)}^{a}=|\tsig^{a}_j-\tsig^{a}_i| \sim \frac{|\theta_1(\tsig^{a}_i)-\theta_1(\tsig^{a}_j)|}{[\Delta_1(\wt\sigma_i^a)]^{2}}, \quad 1\le j\le r^+ ,$$
and
\begin{align*}
 \delta_{\al(i), \beta(\nu)}^{a}&=\left| \tsig^{b}_\nu + m_{1c}^{-1}(\theta_1( \tsig^{a}_i))\right| \sim \left| m_{1c}(\theta_2(\wt\sigma_\nu^b))- m_{1c}(\theta_1( \tsig^{a}_i))\right| \\
 & \sim \frac{|\theta_1(\tsig^{a}_i)-\theta_2(\tsig^{b}_\nu)|}{[\Delta_1(\wt\sigma_i^a)]^{2}}.
 \end{align*}
Thus under Assumption \ref{ass:nonver}, we have that for $\al(i)\in S$,
$$n^{-1/2+\wt\tau}\Delta_1(\widetilde{\sigma}_i^a) + n^{\wt\tau}\phi_n\Delta_1^2(\widetilde{\sigma}_i^a)\lesssim
\begin{cases}
|\theta_1(\tsig^{a}_i)-\theta_1(\tsig^{a}_j)|, & \ \text{if} \ \al(j)\notin S \\
|\theta_1(\tsig^{a}_i)-\theta_2(\tsig^{b}_\nu)|, & \ \text{if} \  \beta(\nu) \notin S
\end{cases}.$$
With a similar arguments for $\beta(\mu)\in S$, we conclude that the eigenvalues with indices in $S$ do not overlap with any other eigenvalues by Theorem \ref{thm_outlier}. 
\end{remark}

The main estimate for outlier eigenvectors under the non-overlapping assumption is included in the following proposition.

\begin{proposition}\label{prop_outev0}
Suppose the assumptions in Proposition \ref{prop_outev} hold. Then under Assumption \ref{ass:nonver}, we have that for all $i,j=1, \cdots, p$,  
\be\label{eq_spikedvector0}
\begin{split}
&\left| \langle \bv_i^a, \cal P_S\bv_j^a\rangle- \delta_{ij}\mathbf 1(\al(i)\in S)\frac{1}{\wt{\sigma}_i^a} \frac{g_{2c}'(-(\wt{\sigma}_i^a)^{-1})}{g_{2c}(-(\wt{\sigma}_i)^{-1})} \right| \\
&\prec \mathbf 1(\al(i)\in S,\al(j)\in S) \left(\phi_n +    {n^{-1/2}(\Delta_1(\widetilde{\sigma}_i^a)\Delta_1(\widetilde{\sigma}_j^a)})^{-1/2}\right) \\
&+  \frac{1}{n}  \left(\frac1{\delta_{\al(i)}(S)} + \frac{\mathbf 1(\al(i)\in S)}{\Delta^2_1(\widetilde{\sigma}_i^a)} \right)\left(\frac1{\delta_{\al(j)}(S)} + \frac{\mathbf 1(\al(j)\in S)}{\Delta_1^2(\widetilde{\sigma}_j^a)} \right)\\
& + \phi_n^2  \left[ \left(\frac{\Delta^2_1(\widetilde{\sigma}_i^a)}{\delta_{\al(i)}(S)}+1\right) \left(\frac1{\delta_{\al(j)}(S)} + \frac{\mathbf 1(\al(j)\in S)}{\Delta^2_1(\widetilde{\sigma}_j^a)} \right)\wedge  \left(i\leftrightarrow j \right)\right] \\
&+ \mathbf 1(\al(i)\in S,\al(j)\notin S)  \frac{ \psi_1(\wt\sigma_i^a)\Delta^2_1(\widetilde{\sigma}_i^a)}{\delta_{\al(i),\al(j)}^a}  \\
&+ \mathbf 1(\al(i)\notin S,\al(j)\in S) \frac{\psi_1(\wt\sigma_j^a)\Delta^2_1(\widetilde{\sigma}_j^a)}{\delta_{\al(i),\al(j)}^a},
\end{split}
\ee
 where $\delta^a$ is defined in (\ref{eq_nu}) of the paper, and  $(i\leftrightarrow j)$ denotes the same term  but with $i$ and $j$ interchanged:
 $$(i\leftrightarrow j):=  \left(\frac{\Delta_1(\widetilde{\sigma}_j^a)^2}{\delta_{\al(j)}(S)} +1 \right)\left(\frac1{\delta_{\al(i)}(S)} + \frac{\mathbf 1(\al(i)\in S)}{\Delta_1(\widetilde{\sigma}_i^a)^2} \right) .$$
\end{proposition}

 
The rest of this subsection is devoted to proving Proposition \ref{prop_outev0}. Suppose that Assumptions \ref{assu_strong} and \ref{ass:nonver} hold. Let $\omega<\tau/2$ and $0<\epsilon<\min\{\tau, \wt\tau\}/10$ be small positive constants to be chosen later. 
By Theorem \ref{thm_largerigidity}, Theorem \ref{lem_localout}, and Theorem \ref{thm_outlier}, we can choose a high-probability event $\Xi_1 \equiv \Xi_1(\epsilon,\omega, \tau,\wt\tau)$ in which the following estimates hold. 
\begin{itemize}
\item[(i)] For all
\begin{equation}\label{eq_parameterset}
\begin{split}
&z\in S_{out}(\omega):=\left\{ E+ \ii\eta:  \right. \\
&\left. \lambda_+ + n^\omega(n^{- 2/3}+ n^{-1/3}  \phi_n^2) \le E \le \omega^{-1},\eta\in [0,1]\right\},
\end{split}
\ee
we have the anisotropic local law
\begin{equation}\label{eq_bb1}
\mathbf{1}(\Xi_1) \norm{\bU^* (G(z)-\Pi(z)) \bU} \leq n^\e \phi_n +n^{-1/2+\epsilon}(\kappa+\eta)^{-1/4} .
\end{equation}

\item[(ii)] For all $1\le i \le r^+ $ and $1\le \mu - p\le s^+$, we have 
\begin{equation}\label{rigid_outXi1}
\begin{split}
& \mathbf{1}(\Xi_1)\left|\wt\lambda_{\alpha(i)}-\theta_1(\wt\sigma_i^a)\right| \le n^{-1/2+\e}\Delta_1(\widetilde{\sigma}_i^a)+ n^{ \e}\phi_n\Delta_1^2(\widetilde{\sigma}_i^a),\\
& \mathbf{1}(\Xi_1)\left|\wt\lambda_{\beta(\mu)}-\theta_2(\wt\sigma_\mu^b)\right| \le n^{-1/2+\e} \Delta_2(\widetilde{\sigma}_\mu^b) + n^{ \e}\phi_n\Delta_2^2(\widetilde{\sigma}_\mu^b).
\end{split}
\end{equation}

\item[(iii)] For any fixed integer $\varpi>r+s$ and all $r^++s^+ < i \leq \varpi$, we have 
\begin{equation}\label{rigid_outXi2}
\mathbf{1}(\Xi_1)\left(|\lambda_1-\lambda_+| + |\wt \lambda_{i}-\lambda_+|\right) \leq n^\e (\phi_n^2 + n^{-2/3} ).
\end{equation}
\end{itemize}
As in the proof in Section \ref{sec:ev}, the randomness of $X$ only comes into play to ensure that $\Xi_1$ holds with high probability. The rest of the proof is restricted to the event $\Xi_1$ only, and will be entirely deterministic. 

Given any $1\le i \le r^+$, our first step is to give a contour integral representation of the generalized components $\langle \bv_i^a, \cal P_S\bv_j^a\rangle$ using resolvents. We define the radius 
\begin{equation}\label{radius_rhoia}
\rho^a_{i}=c_i \left[{\delta_{\al(i)}(S) \wedge (\tsig_i^a+m_{2c}^{-1}(\lambda_+))}\right], \quad \al(i) \in S ,
\end{equation}
and
\begin{equation}\label{radius_rhomu}
\rho_{\mu}^b = c_\mu \left[{\delta_{\beta(\mu)}(S) \wedge (\tsig_\mu^b+m_{1c}^{-1}(\lambda_+))}\right], \quad \beta(\mu) \in S,
\ee
for some sufficiently small constants $0<c_i,c_\mu <1$. Define the contour $\Gamma:=\partial \mathsf{C}$ as the boundary of the union of open discs 
\begin{equation}\label{eq_defnc}
 \mathsf{C}:= \bigcup_{\al(i)\in S} \mathsf C_i  \cup  \bigcup_{\beta(\mu)\in S} \mathsf C_\mu , 
 \end{equation}
 where 
 $$\mathsf C_i :=B_{\rho_i^a}\left(-(\wt\sigma_i^a)^{-1}\right), \ \ \mathsf C_\mu:=B_{\rho_\mu^b}\left(m_{2c}(\theta_2(\wt\sigma_\mu^b))\right).$$
Here $B_r(x)$ denotes an open disc of radius $r$ around $x$. By choosing sufficiently small $c_i$ and $c_\mu$, we can assume that $\mathsf C\subset \mathbf D_2(\tau_2,\varsigma)$ in Lemma \ref{lem_complexderivative}. In the following lemma, we shall show that: (i) $\overline{g_{2c}(\mathsf{C})}$ is a subset of the parameter set in (\ref{eq_parameterset}) so that we can use the estimate (\ref{eq_bb1});  
(ii) $\partial g_{2c}(\mathsf{C})=g_{2c}(\Gamma)$ only encloses the outliers with indices in $S$. 

\begin{lemma}\label{lem_contourpropo} 
Suppose that Assumptions \ref{assu_strong} and \ref{ass:nonver} hold true. Then the set $\overline{g_{2c}(\mathsf{C})}$ lies in the parameter set $S_{out}(\omega)$ in (\ref{eq_parameterset}) as long as the $c_i$'s and $c_\mu$'s are sufficiently small. Moreover, we have $\{\wt\lambda_{\mathfrak a}\}_{\mathfrak a\in S} \subset g_{2c}(\mathsf{C})$ and all the other eigenvalues lie in the complement of $\overline{g_{2c}(\mathsf{C})}$. 
\end{lemma}
\begin{proof} 
Our proof is similar to the one for \cite[Lemmas 5.4 and 5.5]{principal}. We first show that each $g_{2c}(\mathsf C_i) $ is a subset of $S_{out}(\omega)$. By (\ref{eq_gcomplex}), it is easy to see that $|g_{2c}(\zeta)| \leq \omega^{-1}$ for all $\zeta \in \mathsf{C}$ as long as $\omega$ is sufficiently small. For the lower bound on $\re g_{2c}(\zeta)$, we claim that for any constant $\wt C>0$ and sufficiently small constant $0< \wt c_0<1$, there exists a constant $\wt c_1 \equiv \wt c_1(\wt c_0, \wt C)$ such that 
\begin{equation}\label{claim_reg2c}
\re g_{2c}(\zeta) \geq \lambda_+ + \wt c_1 (\re \zeta-m_{2c}(\lambda_+))^2,
\end{equation}
for $ \re \zeta \ge m_{2c}(\lambda_+)$, $|\text{Im} \zeta | \leq \wt c_0 (\re \zeta-m_{2c}(\lambda_+)),$ and $|\zeta| \leq \wt C$. In fact, if $0\le \re \zeta-m_{2c}(\lambda_+) \le c_0 $ for some sufficiently small constant $c_0>0$, then \eqref{claim_reg2c} follows from (\ref{sqroot4}) of the paper that
$$\re g_{2c}(\zeta) -\lambda_+ \sim \re( \zeta-m_{2c}(\lambda_+))^2 \sim (\re \zeta-m_{2c}(\lambda_+))^2$$
for $|\text{Im} \zeta | \leq \wt c_0 (\re \zeta-m_{2c}(\lambda_+))$. On the other hand, if $\re \zeta-m_{2c}(\lambda_+) \ge c_0 $, then using \eqref{eq_gcomplexd} we get
$$ \re g_{2c}(\zeta) -\lambda_+ \ge g_{2c}(\re \zeta) -\lambda_+ - C|\zeta-m_{2c}(\lambda_+)|\im \zeta \ge c,$$
for some constants $C>0$ and $c\equiv c(c_0,\wt c_0, \wt C, C)>0$ as long as $\wt c_0$ is small enough. The claim \eqref{claim_reg2c} then follows by first choosing a sufficiently small constant $ \wt c_0$ and then choosing an appropriate constant $\wt c_1$. 

Now as long as $c_i$ is sufficiently small, we conclude that $g_{2c}(\mathsf C_i) \subset S_{out}(\varsigma_2,\e)$ using \eqref{claim_reg2c}, $\im \zeta \le c_i \left(\wt\sigma_i^{a} + m^{-1}_{2c}(\lambda_+)\right)$,
$$\re \zeta-m_{2c}(\lambda_+) \geq \left(-\frac{1}{\wt\sigma_i^a m_{2c}^{-1}(\lambda_+)} - c_i\right)(\wt\sigma_i^a + m_{2c}^{-1}(\lambda_+)),$$
and $(\wt\sigma_i^{-1} + m_{2c}(\lambda_+)) \gtrsim n^\tau ( \phi_n +n^{-1/3})$. Similarly, for $\zeta\in \mathsf C_\mu$, using \eqref{eq_mdiff} and \eqref{eq_derivativebound} we get that
$$\re \zeta-m_{2c}(\lambda_+) \geq m_{2c}(\theta_2(\wt\sigma_\mu^b)) -m_{2c}(\lambda_+) - c_\mu(\wt\sigma^b_\mu + m_{1c}^{-1}(\lambda_+)) \ge \wt c_2(\wt\sigma^b_\mu + m_{1c}^{-1}(\lambda_+)) ,$$
and 
$$\im \zeta \le \wt C_2c_\mu(\wt\sigma_i^{-1} + m_{2c}(\lambda_+))$$ 
for some constants $\wt c_2,\wt C_2>0$ that are independent of $c_\mu$. Then using \eqref{claim_reg2c} and \eqref{eq_spikeassuastr}, we obtain that $g_{2c}(\mathsf C_\mu)\subset  S_{out}(\varsigma_2,\e)$ as long as $c_\mu$ is sufficiently small. This finishes the proof of the first statement.

Next, we prove the second statement. If suffices to show that: 
\begin{itemize}
\item[(i)] $\wt\lambda_{\al(i)} \in g_{2c}(\mathsf{C}_i)$ and $\wt\lambda_{\beta(\mu)} \in g_{2c}(\mathsf{C}_\mu)$ for all  $ \al(i) \in  S$ and $\beta(\mu)\in S$; 

\item[(ii)] all the other eigenvalues $\wt\lambda_j$ satisfies $\wt\lambda_j \notin  g_{2c}(\mathsf{C}_i)$ and $\wt\lambda_j \notin  g_{2c}(\mathsf{C}_\mu)$ for all  $ \al(i) \in  S$ and $\beta(\mu)\in S$. 
\end{itemize}
To prove (i), we notice that under Assumptions \ref{assu_strong} and \ref{ass:nonver}, 
\begin{equation*}\label{eq_rhobound0}
\rho^a_i \geq \left[\Delta_1(\widetilde{\sigma}_i^a)\right]^{-1}n^{-1/2+2\e}+n^{2\e}\phi_n, \quad \rho^b_\mu \geq \left[\Delta_2(\widetilde{\sigma}_\mu^b)\right]^{-1}n^{-1/2+2\e}+n^{2\e}\phi_n,
\end{equation*}
where we recall that $\epsilon < \min \{\tau, \widetilde{\tau}\}/10.$ 
Together with \eqref{eq_gderivative}, we get that
$$\left|g_{2c}\left(-(\wt\sigma_i^a)^{-1} \pm \rho^a_i\right) - g_{2c}\left(-(\wt\sigma_i^a)^{-1} \right) \right| \gtrsim \Delta_1(\widetilde{\sigma}_i^a)n^{-1/2+2\e}+n^{2\e}\phi_n\Delta_1^2(\widetilde{\sigma}_i^a)$$
for $\al(i)\in S,$ and
$$\left|g_{2c}\left(m_{2c}(\theta_2(\wt\sigma_\mu^b)) \pm \rho^b_\mu\right) - \theta_2(\wt\sigma_\mu^b) \right| \gtrsim \Delta_2(\widetilde{\sigma}_\mu^b)n^{-1/2+2\e}+n^{2\e}\phi_n\Delta_2^2(\widetilde{\sigma}_\mu^b)$$
for $\beta(\mu)\in S.$ Then we conclude (i) using \eqref{rigid_outXi1}.
In order to prove (ii), we consider the two cases: (1) $ j\in \mathcal O^+ \setminus S$; (2) $j\notin \mathcal O^+$. In case (1), the claim follows from Assumption \ref{ass:nonver}, \eqref{rigid_outXi1} and (\ref{eq_gderivative}); see Remark \ref{remark_nonover}. 
In case (2), the claim follows from \eqref{rigid_outXi2} and the first statement of this lemma. This concludes the proof. 
\end{proof}

For the proof of Proposition \ref{prop_outev0}, we shall use a contour integral representation of $\cal P_S$.  As in \eqref{spectral1} and \eqref{spectral2}, we have the following spectral decompositions for $\wt G$:
\begin{equation}\label{eq_greenexpan}
\begin{split}
\wt G_{ij} = \sum_{k = 1}^{p} \frac{\wt{\bm \xi}_k(i) \wt{\bm \xi}_k^*(j)}{\wt\lambda_k-z},\  \ & \wt G_{\mu\nu} = \sum_{k = 1}^{n} \frac{\wt{\bm \zeta}_k(\mu) \wt{\bm \zeta}_k^*(\nu)}{\wt\lambda_k-z}, \\
\wt G_{i\mu} = z^{-1/2}\sum_{k = 1}^{p\wedge n} \frac{\sqrt{\wt\lambda_k}\wt{\bm \xi}_k(i) \wt{\bm \zeta}_k^*(\mu)}{\wt\lambda_k-z}, \  \ &\wt G_{\mu i} =  z^{-1/2}\sum_{k = 1}^{p\wedge n} \frac{\sqrt{\wt\lambda_k}\wt{\bm \zeta}_k(\mu) \wt{\bm \xi}_k^*(i)}{\wt\lambda_k-z}. 
\end{split}
\end{equation}
By (\ref{eq_greenexpan}), Lemma \ref{lem_contourpropo} and Cauchy's integral formula, we have 
\begin{equation}\label{eq_greenrepresent}
\langle \mathbf{v}_i^a, \cal P_S\bv_j^a \rangle =-\frac{1}{2 \pi \mathrm{i}}  \oint_{g_{2c}(\Gamma)} \langle \mathbf{v}_i, \widetilde{G}(z) \mathbf{v}_j \rangle dz,
\end{equation} 
where $\bv_{i/j}$ is the natural embedding of $\bv^a_{i/j}$ in $\mathbb C^{\mathcal I}$. We next provide a representation for  $\langle \mathbf{v}_i, \widetilde{G}(z) \mathbf{v}_j \rangle$ for $1\le i,j \le r$. Using \eqref{Pepsilon0} and the Woodbury matrix identity in Lemma \ref{lem_woodbury},  we obtain that 
\begin{equation} \label{eq_gexpan}
\begin{split}
&\bU^*\widetilde{G}(z)\bU=\bU^*P^{-1}\left(H-z+z(1-P^{-2})\right)^{-1}P^{-1}\bU \\
&=\bU^*P^{-1}\left(G^{-1}(z)+z\bU\mathcal D\bU^*\right)^{-1}P^{-1}\bU\\
&=\bU^*P^{-1} \left[G(z)-z G(z)\mathbf{U}\frac{1}{\mathcal D^{-1}+ z\bU^*G(z)\bU} \mathbf{U}^*G(z)\right]P^{-1}\bU\\
&=\wt{\mathcal D}^{\frac12}\left[\bU^*G(z)\bU-z \bU^*G(z)\mathbf{U}\frac{1}{\mathcal D^{-1}+ z\bU^*G(z)\bU} \mathbf{U}^*G(z)\bU\right] \wt{\mathcal D}^{\frac12} ,
\end{split}
\end{equation}
where 
$$\wt{\mathcal D}:=\begin{pmatrix} (1+D^a)^{-1} & 0\\ 0& (1+D^b)^{-1}\end{pmatrix} .$$
With \eqref{eq_greenrepresent} and \eqref{eq_gexpan}, we now give the proof of Proposition \ref{prop_outev0}. 

\begin{proof}[Proof of Proposition \ref{prop_outev0}] 
We denote $\mathcal E(z)=z\mathbf{U}^*(\Pi(z) -G(z))\mathbf{U}.$ Then we can write 
\begin{equation*}
z\mathbf{U}^*G(z) \mathbf{U}= z \mathbf{U}^* \Pi(z) \mathbf{U}-\mathcal E(z). 
\end{equation*} 
We now perform a resolvent expansion for the denominator in (\ref{eq_gexpan}) as
\be\label{resolvent_3rd}
\begin{split}
&\frac{1}{\mathcal D^{-1}+ z\bU^*G(z)\bU} = \frac{1}{\mathcal D^{-1}+ z\bU^*\Pi(z)\bU}  \\
&\quad+ \frac{1}{\mathcal D^{-1}+ z\bU^*\Pi(z)\bU} \mathcal E \frac{1}{\mathcal D^{-1}+ z\bU^*\Pi(z)\bU}  \\
&\quad+ \frac{1}{\mathcal D^{-1}+ z\bU^*\Pi(z)\bU}\mathcal E \frac{1}{\mathcal D^{-1}+ z\bU^*G(z)\bU}\mathcal E \frac{1}{\mathcal D^{-1}+ z\bU^*\Pi(z)\bU} .
\end{split}
\ee
Inserting it into (\ref{eq_greenrepresent}) and using that $\Gamma$ does not enclose any pole of $G$ by \eqref{rigid_outXi2}, we obtain that 
\begin{align*}
\langle \mathbf{v}_i^a,\cal P_S\bv_j^a\rangle=\frac{\sqrt{(1+d_i^a)(1+d_j^a)}}{d_i^a d_j^a}(s_0+s_1+s_2),
\end{align*}
where $s_0$, $s_1$ and $s_2$ are defined as 
\begin{align*}
& s_0=\frac{\delta_{ij}}{2 \pi \mathrm{i}} \oint_{g_{2c}(\Gamma)} \frac{1}{(d_i^a)^{-1}+1-(1+m_{2c}(z)\sigma_i^a)^{-1}} \frac{\dd z}{z}, \\
& s_1=\frac{1}{2 \pi \ri}  \oint_{g_{2c}(\Gamma)} \frac{\mathcal E_{ij}(z)}{\left( (d_i^a)^{-1}+1-(1+m_{2c}(z) \sigma_i^a)^{-1}\right)\left( (d_j^a)^{-1}+1-(1+m_{2c}(z) \sigma_j^a)^{-1}\right) }   \frac{\dd z}{z},  \end{align*}
and
\begin{align*}
& s_2= \frac{1}{2 \pi \ri} \oint_{g_{2c}(\Gamma)} \left( \frac{1}{\mathcal D^{-1}+ z\bU^*\Pi(z)\bU} \mathcal E(z) \frac{1}{\mathcal D^{-1}+ z\bU^*G(z)\bU} \mathcal E(z)  \frac{1}{\mathcal D^{-1}+ z\bU^*\Pi(z)\bU} \right)_{ij} \frac{\dd z}{z}.
\end{align*}

First of all, the zeroth order limit $s_0$ can be calculated using Cauchy's theorem as
\be\label{estimate_s0finish}
\begin{split}
&\frac{\sqrt{(1+d_i^a)(1+d_j^a)}}{d_i^a d_j^a} s_0 =\frac{1+d_i^a}{d_i^a}\frac{\delta_{ij}}{2 \pi \mathrm{i}} \oint_{g_{2c}(\Gamma)} \frac{1+m_{2c}(z)\sigma_i^a}{1+m_{2c}(z)\wt\sigma_i^a } \frac{\dd z}{z} \\
& = \frac{d_i^a+1}{d_i^a\wt{\sigma}_i^a}\frac{\delta_{ij}}{2 \pi \mathrm{i} } \oint_{\Gamma} \frac{g_{2c}'(\zeta)}{g_{2c}(\zeta)} \frac{1+\zeta \sigma_i^a}{\zeta+(\wt{\sigma}_i^a)^{-1}} d\zeta=\delta_{ij} \frac{1}{\wt{\sigma}_i^a} \frac{g_{2c}'(-(\wt{\sigma}_i^a)^{-1})}{g_{2c}(-(\wt{\sigma}_i^a)^{-1})}.
\end{split}
\ee

For the first order error $s_1$, we can further write it as 
\begin{equation}\label{contour_s1}
s_1=\frac{d_i^a d_j^a}{\tsig_i^a\tsig_j^a } \frac1{2 \pi \ri }\oint_{\Gamma} \frac{h_{ij}(\zeta)}{(\zeta+(\tsig_i^a)^{-1})(\zeta+(\tsig_j^a)^{-1})} d \zeta, 
\end{equation} 
where $h_{ii}(\zeta)$ is defined as 
$$h_{ij}(\zeta):=(1+\zeta \sigma_i^a)(1+\zeta \sigma_j^a) \mathcal E_{ij}(g_{2c}(\zeta)) \frac{g_{2c}'(\zeta)}{g_{2c}(\zeta)}.$$
With (\ref{eq_bb1}), (\ref{eq_gcomplex}) and \eqref{eq_gcomplexd}, we find that 
\begin{equation}\label{eq_hjjbound}
\begin{split}
|h_{ij}(\zeta)| & \lesssim n^\e \left(\phi_n + n^{-1/2 }|g_{2c}(\zeta)-\lambda_+|^{-1/4}\right) |\zeta-m_{2c}(\lambda_+)| \\
&\lesssim n^\e \left( \phi_n |\zeta-m_{2c}(\lambda_+)| + n^{-1/2}|\zeta-m_{2c}(\lambda_+)|^{1/2} \right)
\end{split}
\end{equation}
for $\zeta \in \Gamma$, where we used that $(\kappa+\eta)|_{z=g_{2c}(\zeta)} \gtrsim |g_{2c}(\zeta)-\lambda_+|.$ Moreover, $h_{ij}(\xi)$ is holomorphic on $\{\zeta\in \mathbb C: \re\zeta - m_{2c}(\lambda_+)\ge n^\e(\phi_n +n^{-1/3})\}$ by \eqref{rigid_outXi2}. Hence using Cauchy's differentiation formula, we obtain that
\begin{equation*}
h_{ij}'(\zeta)=\frac{1}{2 \pi \ri} \oint_{\mathcal{C}} \frac{h_{ij}(\xi)}{(\xi-\zeta)^2} d \xi,
\end{equation*}
where $\mathcal{C}$ is the disc of radius $|\zeta-m_{2c}(\lambda_+)|/2$ centered at $\zeta.$ Together with (\ref{eq_hjjbound}), we obtain that
\begin{equation}\label{eq_hjjbound'}
|h'_{ij}(\zeta)| \leq Cn^\e\left(\phi_n + |\zeta-m_{2c}(\lambda_+)|^{-1/2} n^{-1/2}\right).
\end{equation}
Next we consider three different cases. First suppose that $\al(i)\in S$ and $\al(j)\in S$. If $\wt\sigma_i^a \ne \wt\sigma_j^a$, we have
\be\label{estimate_s1finish1}
\begin{split}
|s_1| & \le C\left| \frac{h_{ij}(-(\wt\sigma_i^a)^{-1})-h_{ij}((-(\wt\sigma_j^a)^{-1}))}{(\wt\sigma_i^a)^{-1}-(\wt\sigma_j^a)^{-1}}\right| \\
&\le \frac{C}{|(\wt\sigma_i^a)^{-1}-(\wt\sigma_j^a)^{-1}|}\left|\int_{-(\wt\sigma_i^a)^{-1}}^{-(\wt\sigma_j^a)^{-1}}|h'_{ij}(\zeta)|d\zeta \right| \\
& \le Cn^\e\phi_n +\frac{Cn^{-1/2+\e}}{\sqrt{\Delta_1(\wt\sigma_i^a)\Delta_1(\wt\sigma_j^a)}},
\end{split}
\ee
where we used \eqref{eq_hjjbound'} in the last step. If $\wt\sigma_i^a = \wt\sigma_j^a$, then a simple application of the residue's theorem gives the same bound.
Next we suppose that $\al(i)\in S$ and $\al(j)\notin S$. Then we get from \eqref{eq_hjjbound} that
\be\label{estimate_s1finish2}
|s_1|\le C\frac{|h_{ij}(-(\wt\sigma_i^a)^{-1})|}{|\wt\sigma_i^a-\wt\sigma_j^a|} \le Cn^\e \frac{n^{-1/2 }\Delta_1(\wt\sigma_i^a) + \phi_n \Delta^2_1(\wt\sigma_i^a)}{\delta_{\al(i),\al(j)}^a}.
\ee
We have a similar estimate if $\al(i)\notin S$ and $\al(j)\in S$. Finally, if $\al(i)\notin S$ and $\al(j)\notin S$, we have $s_1=0$ by Cauchy's residue theorem. 

It remains to estimate the second order error $s_2$. We decompose the contour into 
\be\label{decompose_contour}
\Gamma=\bigcup_{\al(i)\in S} \Gamma_i \cup \bigcup_{\beta(\mu)\in S} \Gamma_\mu, \quad \Gamma_i:= \Gamma\cap \partial\mathsf C_i, \quad  \Gamma_\mu:= \Gamma\cap \partial\mathsf C_\mu.
\ee
We have the following basic estimates on each of these components.

\begin{lemma}\label{lem_distance}
For any $\al(i)\in S$, $1\le j \le r$, $1\le \nu-p\le s$ and $\zeta\in \partial \mathsf C_i$, we have 
\begin{equation}\label{zeta_rhoi}
|\zeta+(\tsig_j^a)^{-1}| \sim \rho_i^a+\delta^a_{\al(i),\al(j)} , 
\ee
and 
\be\label{zeta_rhomu}
\left|m_{1c}(g_{2c}(\zeta)) + (\wt\sigma_\nu^b)^{-1}\right| \sim  \rho_i^a +\delta^a_{\al(i),\beta(\nu)}.
\ee
For any $\beta(\mu)\in S$, $1\le j \le r$, $1\le \mu-p\le s$ and $\zeta\in \partial \mathsf C_\mu$, we have 
\begin{equation}\label{zeta_rhoi2}
|\zeta+(\tsig_j^a)^{-1}| \sim \rho_\mu^b+\delta_{\beta(\mu), \al(j)}^b, 
\ee
and 
\be\label{zeta_rhomu2}
\left|m_{1c}(g_{2c}(\zeta)) + (\wt\sigma_\nu^b)^{-1}\right| \sim \rho_\mu^b +\delta^b_{\beta(\mu),\beta(\nu)}. 
\ee
\end{lemma}

\begin{proof} 
The proof is similar to but a little more complicated than the one for \cite[Lemma 5.6]{principal}. The upper bound in \eqref{zeta_rhoi} follows from the triangle inequality:
\begin{equation*}
|\zeta+(\tsig_j^a)^{-1}|  \leq \rho_i^a+|(\tsig_i^a)^{-1}-(\tsig_j^a)^{-1}| \lesssim \rho_i^a+\delta^a_{\al(i),\al(j)} .
\end{equation*}
It remains to prove a lower bound. For $\al(j)\notin S$, by Assumptions \ref{assu_strong} and \ref{ass:nonver}, we trivially have $|(\tsig^a_i)^{-1}-(\tsig_j^a)^{-1}| \geq 2\rho_i^a$, from which we obtain that
\begin{equation*}
|\zeta+(\tsig_k^a)^{-1}| \geq |(\tsig_i^a)^{-1}-(\tsig_j^a)^{-1}|-\rho_i^a \gtrsim \rho_i^a+|(\tsig_i^a)^{-1}-(\tsig_j^a)^{-1}|.
\end{equation*} 

Next we consider the case $\al( j)\in S$. Define $\delta:=|(\tsig_i^a)^{-1}-(\tsig_j^a)^{-1}|-\rho_j^a-\rho_i^a$.  First suppose that $C_0 \delta >|(\tsig_i^a)^{-1}-(\tsig_j^a)^{-1}|$ for some constant $C_0>1$. It then follows that $\rho_i^a+\rho_j^a \leq \frac{C_0-1}{C_0} |(\tsig^a_i)^{-1}-(\tsig^a_j)^{-1}|.$ 
As a consequence, we obtain that 
\begin{equation*}
|\zeta+(\tsig^a_j)^{-1}| \geq \left|(\tsig^a_i)^{-1}-(\tsig^a_j)^{-1}\right|-\rho_i^a \geq \frac{1}{C_0}|(\tsig^a_i)^{-1}-(\tsig^a_j)^{-1}| \gtrsim  \rho_i^a+\delta^a_{\al(i),\al(j)}.
\end{equation*}
Suppose now that $C_0 \delta \leq |(\tsig^a_i)^{-1}-(\tsig^a_j)^{-1}|$. Then we have
\begin{equation*}
|(\tsig^a_i)^{-1}-(\tsig^a_j)^{-1}| \leq \frac{C_0}{C_0-1}(\rho_i^a+\rho_j^a) .
\end{equation*}
We claim that for large enough constant $C_0>0$, there exists a constant $\wt C(c_i,c_j,C_0)>0$ such that   
\be\label{comparablE_rho}
\wt C^{-1}\rho_i^a \leq \rho_j^a \leq \wt C \rho_i^a.
\ee
If \eqref{comparablE_rho} holds, then we have
\begin{equation*}
|\zeta+(\tsig^a_j)^{-1}| \geq \rho_j^a \gtrsim \rho_i^a+\rho_j^a \gtrsim \rho_i^a + \delta^a_{\al(i),\al(j)} .
\end{equation*} 
This concludes \eqref{zeta_rhoi}. 

It remains to prove \eqref{comparablE_rho}. Recall the definitions of $\rho_i^a$ in \eqref{radius_rhoia}. We consider the following three cases. (i) If $\rho_i^a = c_i \delta^a_{\al(i),\al(k)}$ for some $k$ such that $\al(k)\notin S$, then we have
\be\label{proof_example}
\frac{\rho_j^a}{c_j}  \le \delta^a_{\al(j),\al(k)} \le \delta^a_{\al(i),\al(k)} + |\wt\sigma_i^a-\wt\sigma_j^a| \le \frac{\rho_i^a}{c_i} + \frac{C_0\wt\sigma^a_i \wt\sigma^a_j}{C_0-1}(\rho_i^a+\rho_j^a) .
\ee
Thus as long as $c_j$ and $C_0$ is chosen such that $c_j^{-1} > \frac{C_0\wt\sigma^a_i \wt\sigma^a_j}{C_0-1}$, we can obtain the upper bound in \eqref{comparablE_rho}. (ii) If $\rho_i^a = c_i (\wt\sigma_i^a + m_{2c}^{-1}(\lambda_+))$, the proof is the same as in case (i). (iii) If $\rho_i^a = c_i \delta^a_{\al(i),\beta(\nu)}$ for some $\nu$ such that $\beta(\nu)\notin S$, then there exists a constant $C>0$ independent of $c_i,c_j,C_0$ such that 
\begin{align*}
\frac{\rho_j^a}{c_j} & \le \left|m_{1c}^{-1}(\theta_1( \tsig^{a}_i))-m_{1c}^{-1}(\theta_1( \tsig^{a}_j))\right|  +\delta^a_{\al(i),\beta(\nu)}\\
& \le \frac{\rho_i^a}{c_i} + C|\wt\sigma_i^a-\wt\sigma_j^a|\le \frac{\rho_i^a}{c_i}+ \frac{CC_0\wt\sigma^a_i \wt\sigma^a_j}{C_0-1}(\rho_i^a+\rho_j^a) ,
\end{align*}
where in the second step we used \eqref{isometry2}. Again we obtain the upper bound in \eqref{comparablE_rho} by choosing appropriate $c_j$ and $C_0$. Finally, the lower bound in \eqref{comparablE_rho} follows immediately by switching the roles of $i$ and $j$.

The proof for \eqref{zeta_rhomu}, \eqref{zeta_rhoi2} and \eqref{zeta_rhomu2} is similar; the only difference is that we need to use the approximate isometry properties in \eqref{isometry} and \eqref{isometry2}. 
\end{proof}

 Now we finish the estimate of $s_2$. First with (\ref{eq_bb1}), (\ref{eq_gcomplex}) and \eqref{eq_gcomplexd}, we can estimate that
\begin{equation}\label{eq_s3bound}
\begin{split}
|s_2| &\leq C \oint_{\Gamma} \frac{n^{2\e}\phi_n^2+n^{-1+2\epsilon}|\zeta-m_{2c}(\lambda_+)|^{-1}}{|\zeta+(\tsig_i^a)^{-1}||\zeta+(\tsig_j^a)^{-1}|} {|g_{2c}'(\zeta)|}  \\
&\quad \times \left\| \left(\mathcal D^{-1}+g_{2c}(\zeta) \mathbf{U}^* G(g_{2c}(\zeta)) \mathbf{U})^{-1}\right)^{-1} \right\| |d \zeta| \\
&\le  C \oint_{\Gamma} \frac{n^{-1+2\epsilon}+n^{2\e}\phi_n^2|\zeta-m_{2c}(\lambda_+)| }{|\zeta+(\tsig_i^a)^{-1}||\zeta+(\tsig_j^a)^{-1}|} \frac{1}{\mathfrak d(\zeta)-\norm{ \mathcal E(g_{2c}(\zeta))} } |d \zeta|,
\end{split}
\end{equation}
where 
\begin{align*}
\mathfrak d(\zeta) : = & \min\left\{\min_{1\le j \le r}\left|(d_j^a)^{-1}+1-(1+\zeta \sigma_j^a)^{-1}\right|,\right. \\
&\left. \min_{1\le \mu-p \le s}\left|(d_\mu^b)^{-1}+1-(1+m_{1c}(g_{2c}(\zeta)) \sigma_\mu^b)^{-1}\right|\right\}.
\end{align*}
We can bound $\norm{ \mathcal E(g_{2c}(\zeta))} $ using (\ref{eq_bb1}), (\ref{eq_gcomplex}) and the Hilbert-Schmidt norm as
\begin{equation}\label{eq_localcontrol}
\norm{\mathcal E(g_{2c}(\zeta))} \leq C \sqrt{rs} n^\e \left[\phi_n +n^{-1/2 } \left(\zeta-m_{2c}(\lambda_+)\right)^{-1/2} \right].
\end{equation}
For $\mathfrak d(\zeta)$, we have for $1\le j \le r,$
\begin{align}\label{eq_matrixentry}
&\frac{d_j^a+1}{d_j^a}-\frac{1}{1+\zeta \sigma_j^a} = \frac{1+ \wt\sigma_j^a\zeta}{d_j^a \sigma_j^a \left(\zeta+(\sigma_j^a)^{-1}\right)},
\end{align}
 and for $1\le \mu-p\le s,$
\begin{align}\label{eq_matrixentry2}
&\frac{d_\mu^b + 1}{d_\mu^b}-\frac1{1+m_{1c}(g_{2c}(\zeta)) \sigma_\mu^b} = \frac{1+\wt\sigma_\mu^b m_{1c}(g_{2c}(\zeta)) }{d_\mu^b \sigma_\mu^b \left(m_{1c}(g_{2c}(\zeta)) + (\sigma_\mu^b)^{-1}\right)}.
\end{align}

Note that we have $|\zeta+(\sigma_j^a)^{-1}|\sim 1$ and $|m_{1c}(g_{2c}(\zeta)) + (\sigma_\mu^b)^{-1}|\sim 1$ by \eqref{Piii}. On the other hand, we can use Lemma \ref{lem_distance} to bound the numerators from below. Thus we obtain that
\begin{align*}
\norm{\mathcal E(g_{2c}(\zeta))}&\ll \left(\wt{\sigma}^a_{i}+{m_{2c}^{-1}(\lambda_+)}\right) \wedge \left[n^{\wt \tau}\phi_n + \left(\tsig_i^a+m_{2c}^{-1}(\lambda_+)\right)^{-1/2} n^{-1/2+\wt\tau} \right]\\
&\lesssim \begin{cases}\rho_i^a \lesssim \mathfrak d(\zeta), \ & \text{for } \zeta\in \Gamma_i\\
\rho_\mu^b \lesssim \mathfrak d(\zeta), \ & \text{for } \zeta\in \Gamma_\mu
\end{cases},
\end{align*}
where we used Assumption \ref{assu_strong}, Assumption \ref{ass:nonver} and \eqref{eq_localcontrol}.
Thus we have  
\begin{equation}\label{eq_boundsim1}
 \frac{1}{\mathfrak d(\zeta)-\norm{ \mathcal E(g_{2c}(\zeta))} } \lesssim \begin{cases}(\rho_i^a)^{-1}, \ & \text{for } \zeta\in \Gamma_i\\
(\rho_\mu^b)^{-1}  , \ & \text{for } \zeta\in \Gamma_\mu
\end{cases} .
\end{equation}
Decomposing the integral contour in \eqref{eq_s3bound} as in \eqref{decompose_contour}, using \eqref{eq_boundsim1} and Lemma \ref{lem_distance}, and recalling that the length of $\Gamma_i$ (or $\Gamma_\mu$) is at most $2 \pi \rho_i^a$ (or $2\pi \rho_\mu^b$), we get that 
\begin{equation}\label{estimate_s2}
\begin{split}
|s_2| & \le  C \sum_{\al(k) \in S} \frac{n^{-1+2\epsilon}+n^{2\e}\phi_n^2\Delta_1^2(\wt\sigma_k^a) }{(\rho_k^a + \delta^a_{\al(k),\al(i)})(\rho_k^a + \delta^a_{\al(k),\al(j)}) } \\
&+C \sum_{\beta(\mu) \in S} \frac{n^{-1+2\epsilon}+n^{2\e}\phi_n^2\Delta_2^2(\wt\sigma_\mu^b)}{(\rho_\mu^b + \delta^b_{\beta(\mu),\al(i)})(\rho_\mu^b + \delta^b_{\beta(\mu),\al(j)})} .
\end{split}
\end{equation}
 Finally, we estimate the RHS of \eqref{estimate_s2}. We have
$$ \Delta_1^2(\wt\sigma_k^a)  \lesssim \Delta_1^2(\wt\sigma_i^a) + \delta^a_{\al(k),\al(i)},\quad \Delta_2^2(\wt\sigma_\mu^b)  \lesssim \Delta_1^2(\wt\sigma_i^a) + \delta^b_{\beta(\mu), \al(i)}.$$
For $\al(i)\notin S$, $\al(k)\in S$ and $\beta(\mu)\in S$, we have
\begin{align*}
& \frac{1}{(\rho_k^a + \delta^a_{\al(k),\al(i)})^2 } +\frac{1}{(\rho_\mu^b + \delta^b_{\beta(\mu),\al(i)})^2 }  \le \frac{1}{(\delta^a_{\al(k),\al(i)})^2 }+ \frac{1}{ (\delta^b_{\beta(\mu),\al(i)})^2 }  \le \frac{C}{\delta_{\al(i)}(S)^2} .
\end{align*}
For $\al(i)\in S$, we have $\rho_k^a + \delta^a_{\al(k),\al(i)}\gtrsim\rho_i^a $ for $\al(k)\in S$, and $\rho_\mu^b + \delta^b_{\beta(\mu),\al(i)}\gtrsim \rho_i^a$ for $\beta(\mu)\in S$ (which follow from arguments that are similar to the first two inequalities in \eqref{proof_example}). Then we have 
$$ \frac{1}{(\rho_k^a + \delta^a_{\al(k),\al(i)})^2 }+ \frac{1}{(\rho_\mu^b + \delta^b_{\beta(\mu),\al(i)})^2 }  \le  \frac{C}{(\rho_i^a)^2}\le \frac{C}{\delta_{\al(i)}(S)^2} + \frac{C}{\Delta_1(\wt\sigma_i^a)^4} .$$
Plugging the above estimates into \eqref{estimate_s2}, we get that 
\begin{equation}\label{estimate_s2finish}
\begin{split}
|s_2| \lesssim n^{-1+2\e}\left(\frac1{\delta_{\al(i)}(S)} + \frac{\mathbf 1(\al(i)\in S)}{\Delta_1(\widetilde{\sigma}_i^a)^2} \right)\left(\frac1{\delta_{\al(j)}(S)} + \frac{\mathbf 1(\al(j)\in S)}{\Delta_1(\widetilde{\sigma}_j^a)^2} \right)\\
+ n^{2\e}\phi_n^2  \left[\left(\frac{\Delta_1(\widetilde{\sigma}_i^a)^2}{\delta_{\al(i)}(S)} +1 \right)\left(\frac1{\delta_{\al(j)}(S)} + \frac{\mathbf 1(\al(j)\in S)}{\Delta_1(\widetilde{\sigma}_j^a)^2} \right)\wedge  \left(i\leftrightarrow j \right)\right] .
\end{split}
\end{equation}
Combining \eqref{estimate_s0finish}, \eqref{estimate_s1finish1}, \eqref{estimate_s1finish2} and \eqref{estimate_s2finish}, we obtain (\ref{eq_spikedvector0}) for $1\le i,j \le r$ since $\e$ can be arbitrarily small.

We can easily extend the above arguments to the general case. For any $i,j \in \{1,\cdots, p\}$, we define $\mathcal R:=\{1,\cdots, r\}\cup \{i,j\}$. Then we define a perturbed model with (recall (\ref{AOBO}) and (\ref{eq_sepamodel}) of the paper)
\begin{equation}\nonumber
\widehat A = A\Big(1+{\widehat V_o^a} {\wh D}^{a} ({\wh V_o}^a)^*\Big), \quad \wh{D}^{a}=\text{diag}(d_k^{a})_{k\in \mathcal R}, \quad V_o^a=(\bv_k^a)_{k\in \mathcal R},
\end{equation}
where 
$$\quad d_{k}^a:=\begin{cases} d_k^a, \ &\text{if } 1\le k\le r\\
\wt\e, \ &\text{if } k\in \mathcal R \text{ and } k>r\end{cases}.
$$
Then all the previous proof goes through for the perturbed model as long as we replace the $\mathbf U$ and $\mathcal D$ in \eqref{eq_gexpan} with
\begin{equation}\label{Depsilon0_general}
\wh{\mathbf{U}}=
\begin{pmatrix}
\wh V_{o}^a & 0 \\
0 & V_{o}^b
\end{pmatrix}, \quad \wh{\mathcal{D}}=
\begin{pmatrix}
\wh D^a(\wh D^a+1)^{-1} & 0 \\
0 & D^b(D^b+1)^{-1}
\end{pmatrix}.
\end{equation}
Note that in the proof, only the upper bound on the $d_k^a$'s were used. Moreover, the proof does not depend on the fact that $\wt\sigma_i^a$ or $\wt\sigma_j^a$ satisfy (\ref{spikes}) of the paper (we only need the indices in $S$ to satisfy Assumptions \ref{assu_strong} and \ref{ass:nonver}). Finally, taking $\wt\e\downarrow 0$ and using continuity, we get (\ref{eq_spikedvector0}) for general $i,j \in \{1,\cdots, p\}$.
\end{proof}

\subsection{Removing the non-overlapping condition}\label{sec_remove}

In this subsection, we prove Proposition \ref{prop_outev} by removing the non-overlapping Assumption \ref{ass:nonver} in Proposition \ref{prop_outev0}. The proof is an extension of the one in \cite[Section 5.2]{principal}


\begin{proof}[Proof of Proposition \ref{prop_outev}] 
Recall the constants $\tau$ in Assumption \ref{assu_strong} and $\wt\tau$ in Assumption \ref{ass:nonver}. Let $\wt\tau<\tau/4$. We define the index set (recall \eqref{eq_otau})
$$\mathcal O^+_{\tau/2}:= \left\{\al(i): i \in \mathcal{O}^{(a)}_{\tau/2}\right\}\cup \left\{\beta(\mu): \mu \in \mathcal{O}^{(b)}_{\tau/2}\right\}.$$
For simplicity, we denote
$$\delta_{\mathfrak a, \cdot} := \begin{cases}\delta^a_{\mathfrak a,\cdot}, \ & \text{if } \fa=\al(i)\\
\delta^b_{\mathfrak a,\cdot}, \ & \text{if } \fa=\beta(\mu)
\end{cases}$$
for any $\mathfrak a\in \mathcal O^+$. We say that $\fa \neq \fb\in \mathcal{O}_{\tau/2}^+$ overlap if $$\delta_{\fa,\fb}\wedge \delta_{\fb,\fa}\leq \begin{cases}\left[\Delta_1(\widetilde{\sigma}_i^a)\right]^{-1} n^{-1/2+\wt\tau} + n^{\wt\tau}\phi_n, \ & \text{if } \fa=\al(i)\\
\left[\Delta_2(\widetilde{\sigma}_\mu^b)\right]^{-1} n^{-1/2+\wt\tau}+ n^{\wt\tau}\phi_n, \ & \text{if } \fa=\beta(\mu)
\end{cases}, $$
 \text{or} $$ \quad  \delta_{\fa,\fb}\wedge \delta_{\fb,\fa}\leq \begin{cases}\left[\Delta_1(\widetilde{\sigma}_j^a)\right]^{-1} n^{-1/2+\wt\tau}+ n^{\wt\tau}\phi_n, \ & \text{if } \fb=\al(j)\\
\left[\Delta_2(\widetilde{\sigma}_\nu^b)\right]^{-1} n^{-1/2+\wt\tau}+ n^{\wt\tau}\phi_n, \ & \text{if } \fb=\beta(\nu)
\end{cases}.$$
\begin{definition}
For $S$ satisfying Assumption \ref{assu_strong}, we define sets $L_1(S)\subset S\subset L_2(S)$ such that 
$L_1(S)$ is the largest subset of $S$ that do not overlap with its complement, and $L_2(S)$ is the smallest subset of $\mathcal{O}^+_{\tau/2}$ that do not overlap with its complement. 
\end{definition}

It is easy to see that $L_1(S)$ and $L_2(S)$ exist and are unique. For an illustration of these two sets, we refer the reader to Fig. 4 of \cite{principal}. The main reason for defining these two sets is that Proposition \ref{prop_outev0} now holds for $(\tau/2,L_1(S))$ or $(\tau/2,L_2(S))$. 
Now we are ready to prove \eqref{eq_spikedvector}. As discussed at the end of Section \ref{sec_eve_green}, without loss of generality, we can assume that $1\le i, j \le r$. There are four cases to consider.

\vspace{5pt}

\noindent {\bf Case (a)}: $\al(i)=\al(j)\notin S$. If $\al(i)\notin L_2(S)$, then using $r+s=\OO(1)$ we see that $\delta_{\al(i)}(S)\sim \delta_{\al(i)}(L_2(S))$. Then Proposition \ref{prop_outev0} gives that
\be\label{proofout_a1}
\begin{split}
 &\langle \bv_i^a, \cal P_S\bv_i^a\rangle \le\langle \bv_i^a, \cal P_{L_2(S)}\bv_i^a\rangle \\
 &\prec  \frac1{n\delta^2_{\al(i)}(L_2(S))} + \phi_n^2 \frac{\Delta_1^2(\widetilde{\sigma}_i^a)+ \delta_{\al(i)}(L_2(S))}{\delta_{\al(i)}^2(L_2(S))} \\
&\lesssim  \frac{\psi_1^2(\wt\sigma_i^a)\Delta_1^2(\widetilde{\sigma}_i^a)}{\delta^2_{\al(i)}(S) } +  \frac{\phi_n^2}{\delta_{\al(i)}(S)} .
\end{split}
\ee
If $\al(i)\in L_2(S)$, an easy argument gives that 
\be\label{proofout_a2}
\delta_{\al(i)}(S)\le Cn^{\wt\tau}\psi_1(\wt\sigma_i^a) 
\le C\delta_{\al(i)}(L_2(S)).
\ee
Then Proposition \ref{prop_outev0} gives that
\be\label{proofout_a3}
\begin{split}
&\langle \bv_i^a, \cal P_S\bv_i^a\rangle\le \langle \bv_i^a, \cal P_{L_2(S)}\bv_i^a\rangle \\
&\prec \frac{1}{\wt{\sigma}_i^a} \frac{g_{2c}'(-(\wt{\sigma}_i^a)^{-1})}{g_{2c}(-(\wt{\sigma}_i)^{-1})} + \phi_n +   \frac{1}{n^{1/2} \Delta_1(\widetilde{\sigma}_i^a)}+ \frac1{n\delta^2_{\al(i)}(L_2(S))} \\
&+ \frac{\phi_n^2 \Delta_1^2(\wt\sigma_i^a)}{\delta_{\al(i)}^2(L_2(S))} + \frac{\phi_n^2}{\Delta_1(\wt\sigma_i^a)^2}\le C\Delta_1^2(\widetilde{\sigma}_i^a) \le \frac{Cn^{2\wt\tau}\psi_1^2(\wt\sigma_i^a)\Delta^2_1(\widetilde{\sigma}_i^a)}{\delta^2_{\al(i)}(S)} ,
\end{split}
\ee
where we also used Assumption \ref{assu_strong} and \eqref{eq_gderivative} in the third step.

\vspace{5pt}

\noindent {\bf Case (b)}: $\al(i)=\al(j)\in S$. We first consider the case $\al(i)\in L_1(S)$. We can write
\be\label{proofout_b1}
\langle \bv_i^a, \cal P_S\bv_i^a\rangle = \langle \bv_i^a, \cal P_{L_1(S)}\bv_i^a\rangle +  \langle \bv_i^a, \cal P_{S\setminus L_1(S)}\bv_i^a\rangle .
\ee
Using Proposition \ref{prop_outev0} and the fact that $\delta_{\al(i)}(S)\sim \delta_{\al(i)}(L_1(S))$, we can estimate the first term as 
\be\label{proofout_b2}
\begin{split}
&\left| \langle \bv_i^a, \cal P_{L_1(S)}\bv_i^a\rangle- \frac{1}{\wt{\sigma}_i^a} \frac{g_{2c}'(-(\wt{\sigma}_i^a)^{-1})}{g_{2c}(-(\wt{\sigma}_i)^{-1})} \right| \\
&\prec  \psi_1(\wt\sigma_i^a) + \psi^2_1(\wt\sigma_i^a) \Delta^2_1(\wt\sigma_i^a) \left(\frac1{\delta^2_{\al(i)}(S)} + \frac{1}{\Delta^4_1(\widetilde{\sigma}_i^a)} \right) \\
&\prec  \psi_1(\wt\sigma_i^a) +  \frac{\psi^2_1(\wt\sigma_i^a) \Delta^2_1(\wt\sigma_i^a) }{\delta^2_{\al(i)}(S)} ,
\end{split}
\ee
where we used that $\psi_1(\wt\sigma_i^a)\le \Delta^2_1(\wt\sigma_i^a) $ in the last step. For the second term in \eqref{proofout_b1}, it suffices to assume that $S\setminus L_1(S)\ne \emptyset$ (otherwise it is equal to zero). Then we observe that $\delta_{\al(i)}(S) \sim \delta_{\al(i)}(S\setminus L_1(S))$. Applying \eqref{proofout_a1} with $S$ replaced by $S\setminus L_1(S)$, we obtain that
\be\label{proofout_b3}
\langle \bv_i^a, \cal P_{S\setminus L_1(S)}\bv_i^a\rangle \prec  \frac{\psi_1^2(\wt\sigma_i^a)\Delta_1^2(\widetilde{\sigma}_i^a)}{\delta^2_{\al(i)}(S) } +  \frac{\phi_n^2}{\delta_{\al(i)}(S)} \le  \phi_n +  \frac{\psi^2_1(\wt\sigma_i^a) \Delta^2_1(\wt\sigma_i^a) }{\delta^2_{\al(i)}(S)} .
\ee
Next, for the case $\al(i)\notin L_1(S)$, it is easy to show that \eqref{proofout_a2} holds, and as in \eqref{proofout_a3}, we get
\be\label{proofout_b4}
\begin{split}
&\left|\langle \bv_i^a, \cal P_S\bv_i^a\rangle- \frac{1}{\wt{\sigma}_i^a} \frac{g_{2c}'(-(\wt{\sigma}_i^a)^{-1})}{g_{2c}(-(\wt{\sigma}_i)^{-1})} \right|\\
&\le \langle \bv_i^a, \cal P_{L_2(S)}\bv_i^a\rangle  +  \frac{1}{\wt{\sigma}_i^a} \frac{g_{2c}'(-(\wt{\sigma}_i^a)^{-1})}{g_{2c}(-(\wt{\sigma}_i)^{-1})} \prec \frac{n^{2\wt\tau}\psi_1^2(\wt\sigma_i^a)\Delta^2_1(\widetilde{\sigma}_i^a)}{\delta^2_{\al(i)}(S)}.
\end{split}
\ee

\vspace{5pt}

Combining \eqref{proofout_a1}, \eqref{proofout_a3} and \eqref{proofout_b2}-\eqref{proofout_b4}, we conclude that 
\be\label{conclude_i=j}
\begin{split}
&\left| \langle \bv_i^a, \cal P_S\bv_i^a\rangle- \mathbf 1(\al(i)\in S)\frac{1}{\wt{\sigma}_i^a} \frac{g_{2c}'(-(\wt{\sigma}_i^a)^{-1})}{g_{2c}(-(\wt{\sigma}_i)^{-1})} \right| \prec n^{2\wt\tau}\Upsilon(i, S).
\end{split}
\ee
This concludes \eqref{eq_spikedvector} for the $i=j$ case since $\wt \tau$ can be chosen arbitrarily small. 

\vspace{5pt}

\noindent {\bf Case (c)}: $i\ne j$ and $\al(i)\notin S$ or $\al(j)\notin S$. Using \eqref{conclude_i=j} and the basic estimate
\be\label{CS new}
\left| \langle \bv_i^a, \cal P_S\bv_j^a\rangle\right|^2 \le \langle \bv_i^a, \cal P_S\bv_i^a\rangle\langle \bv_j^a, \cal P_S\bv_j^a\rangle ,
\ee
we find that in this case, \eqref{eq_spikedvector} holds with an additional $n^{2\wt\tau}$ factor multiplying the RHS. 

\vspace{5pt}

\noindent {\bf Case (d)}: $i\ne j$ and $\al(i),\al(j)\in S$. Our goal is to prove that  
\be\label{eq_pfgoal}
\begin{split}
& \left| \langle \bv_i^a, \cal P_S\bv_j^a\rangle \right| \\
&\prec n^{2\wt\tau} \left[\psi_{1}^{1/2}(\wt\sigma_i^a)+  \frac{\psi_{1}(\wt\sigma_i^a)\Delta_1(\widetilde{\sigma}_i^a) }{\delta_{\al(i)}(S)} \right] \left[\psi_{1}^{1/2}(\wt\sigma_j^a)+  \frac{\psi_{1}(\wt\sigma_j^a)\Delta_1(\widetilde{\sigma}_j^a) }{\delta_{\al(j)}(S)}  \right] .
\end{split}
\ee  
We again split $\cal P_S$ into
\be\label{proofout_d1}
\langle \bv_i^a, \cal P_S\bv_j^a\rangle = \langle \bv_i^a, \cal P_{L_1(S)}\bv_j^a\rangle +  \langle \bv_i^a, \cal P_{S\setminus L_1(S)}\bv_j^a\rangle .
\ee
There are four cases: (i) $\al(i),\al(j)\in L_1(S)$; (ii) $\al(i)\in L_1(S)$ and $\al(j)\notin L_1(S)$; (iii) $\al(i)\notin L_1(S)$ and $\al(j)\in L_1(S)$; (iv) $\al(i),\al(j)\notin L_1(S)$. 

In case (i), we can bound the first term in \eqref{proofout_d1} using Proposition \ref{prop_outev0} and the estimates that $\delta_{\al(i)}(S)\sim \delta_{\al(i)}(L_1(S))$ and $\delta_{\al(j)}(S)\sim \delta_{\al(j)}(L_1(S))$. The second term in \eqref{proofout_d1} can be bounded as in case (c) above (with $S$ replaced by $S\setminus L_1(S)$) together with the estimates $\phi_n\le \delta_{\al(i)}(S) \le C\delta_{\al(i)}(S\setminus L_1(S))$ and $\phi_n\le \delta_{\al(j)}(S) \le C\delta_{\al(j)}(S\setminus L_1(S))$.

In case (ii), we have
\begin{equation}\label{revisionadd}
\begin{split}
&\delta_{\al(i)}(S)\sim \delta_{\al(i)}(L_1(S)),\quad \delta_{\al(i)}(S) \le C\delta_{\al(i),\al(j)}^a, \\
 &\delta_{\al(j)}(S)\le C n^{\wt\tau}\psi_1(\wt\sigma_j^a) 
 \le C\delta_{\al(j)}(L_1(S)).
 \end{split}
 \ee
 Then with Proposition \ref{prop_outev0}, we can bound the first term in \eqref{proofout_d1} as
\be\nonumber
\begin{split}
&\left| \langle \bv_i^a, \cal P_{L_1(S)}\bv_j^a\rangle\right| \prec  \frac{1}{n\delta_{\al(i)}(L_1(S))\delta_{\al(j)}(L_1(S))} + \frac{1 }{n\delta_{\al(j)}(L_1(S))\Delta_1^2(\widetilde{\sigma}_i^a)} \\
& + \phi_n^2 \Delta_1(\widetilde{\sigma}_i^a)\Delta_1(\widetilde{\sigma}_j^a) \left[\left(\frac{1}{\delta_{\al(i)}(L_1(S)) } +\frac1{\Delta_1^2(\widetilde{\sigma}_i^a)}   \right) \left(\frac{1}{\delta_{\al(j)}(L_1(S)) } +\frac1{\Delta_1^2(\widetilde{\sigma}_j^a)}   \right) \right] \\
&+ \frac{\psi_1(\wt\sigma_i^a)\Delta^2_1(\widetilde{\sigma}_i^a)}{\delta_{\al(i),\al(j)}^a}\\
&\lesssim  \left[\psi_{1}^{1/2}(\wt\sigma_i^a)+ \frac{\psi_{1}(\wt\sigma_i^a)\Delta_1(\widetilde{\sigma}_i^a) }{\delta_{\al(i)}(S)}   \right]  \left[\psi_{1}^{1/2}(\wt\sigma_j^a)+\frac{\psi_{1}(\wt\sigma_j^a)\Delta_1(\widetilde{\sigma}_j^a) }{\delta_{\al(j)}(S)} \right] + \frac{\psi_1(\wt\sigma_i^a)\Delta^2_1(\widetilde{\sigma}_i^a)}{\delta_{\al(i),\al(j)}^a}.
\end{split}
\ee
For the last term, we first assume that $\wt\sigma_j^a\le \wt\sigma_i^a$ and $ \wt\sigma_i^a + m_{2c}^{-1}(\lambda_+) \le 2|\wt\sigma_i^a-\wt\sigma_j^a|$. Then
\begin{align*}
  \frac{\psi_1(\wt\sigma_i^a)\Delta^2_1(\widetilde{\sigma}_i^a)}{\delta_{\al(i),\al(j)}^a}\lesssim  \psi_1(\wt\sigma_i^a) \le \sqrt{\psi_1(\wt\sigma_i^a)\psi_1(\wt\sigma_j^a)}. 
 \end{align*}
On the other hand, if $\wt\sigma_j^a\ge \wt\sigma_i^a$ or $ \wt\sigma_i^a + m_{2c}^{-1}(\lambda_+) \ge 2|\wt\sigma_i^a-\wt\sigma_j^a|$, we have $\Delta_1(\wt\sigma_i^a) \lesssim \Delta_1(\wt\sigma_j^a)$. Hence using \eqref{revisionadd}, we get
\begin{align*}
  \frac{\psi_1(\wt\sigma_i^a)\Delta^2_1(\widetilde{\sigma}_i^a)}{\delta_{\al(i),\al(j)}^a}\lesssim n^{\wt\tau} \frac{\psi_1(\wt\sigma_i^a)\Delta_1(\widetilde{\sigma}_i^a)\psi_1(\wt\sigma_j^a)\Delta_1(\widetilde{\sigma}_j^a) }{\delta_{\al(i)}(S)\delta_{\al(j)}(S)}. 
 \end{align*}
 The above estimates show that $ | \langle \bv_i^a, \cal P_{L_1(S)}\bv_j^a\rangle | $ can be bounded by the right-hand side of \eqref{eq_pfgoal}. 
The second term in \eqref{proofout_d1} can be bounded as in case (c) above (with $S$ replaced by $S\setminus L_1(S)$) together with the estimates in \eqref{revisionadd} and
$$\delta_{\al(i)}(S)\sim \delta_{\al(i)}(S\setminus L_1(S))\gtrsim n^{\wt\tau}\phi_n, \ \ \delta_{\al(j)}(S) \lesssim \delta_{\al(j)}(S\setminus L_1(S)) \le Cn^{\wt\tau}\psi_1(\wt\sigma_j^a). $$
Then we get that
\be\nonumber
\begin{split}
&\left| \langle \bv_i^a, \cal P_{S\setminus L_1(S)}\bv_j^a\rangle\right|  \\
&\prec n^{2\wt\tau} \left[  \frac{\phi_n}{\delta^{1/2}_{\al(i)}(S\setminus L_1(S))}+\frac{\psi_{1}(\wt\sigma_i^a)\Delta_1(\widetilde{\sigma}_i^a) }{\delta_{\al(i)}(S\setminus L_1(S))} \right] \left[\psi_{1}^{1/2}(\wt\sigma_j^a)+  \frac{\psi_{1}(\wt\sigma_j^a)\Delta_1(\widetilde{\sigma}_j^a) }{\delta_{\al(j)}(S\setminus L_1(S))}  \right] \\
&\prec  n^{2\wt\tau} \left[ \psi_{1}^{1/2}(\wt\sigma_i^a) +\frac{\psi_{1}(\wt\sigma_i^a)\Delta_1(\widetilde{\sigma}_i^a) }{\delta_{\al(i)}(S )} \right]\left[\psi_{1}^{1/2}(\wt\sigma_j^a)+  \frac{\psi_{1}(\wt\sigma_j^a)\Delta_1(\widetilde{\sigma}_j^a) }{\delta_{\al(j)}(S )} \right].
\end{split}
\ee
This concludes the proof of \eqref{eq_pfgoal} for case (ii). The case (iii) can be handled in the same way by interchanging $i$ and $j$.

Finally, we deal with case (iv). For the first term in \eqref{proofout_d1}, we have
$$\delta_{\al(i)}(S)\lesssim \delta_{\al(i)}(L_1(S)), \quad \delta_{\al(i)}(L_1(S)) \gtrsim \psi_1(\wt\sigma_i^a), $$ 
and similar estimates for the $\al(j)$ case. Then using Proposition \ref{prop_outev0}, we can obtain that 
\be\nonumber
\begin{split}
&\left| \langle \bv_i^a, \cal P_{L_1(S)}\bv_j^a\rangle \right| \prec  \frac{1}{n\delta_{\al(i)}(L_1(S))\delta_{\al(j)}(L_1(S))}  \\
&+  \phi_n^2  \left[\left(\frac{\Delta_1^2(\widetilde{\sigma}_i^a)}{\delta_{\al(i)}(L_1(S))} +1 \right) \frac1{\delta_{\al(j)}(L_1(S))} \right] \wedge  \left[\left(\frac{\Delta_1^2(\widetilde{\sigma}_j^a)}{\delta_{\al(j)}(L_1(S))} +1 \right) \frac1{\delta_{\al(i)}(L_1(S))} \right]\\ 
&\lesssim   \frac{\psi_1(\wt\sigma_i^a)\psi_1(\wt\sigma_j^a)\Delta_1(\widetilde{\sigma}_i^a) \Delta_1(\widetilde{\sigma}_j^a)}{ \sqrt{\delta_{\al(i)}(L_1(S))\delta_{\al(j)}(L_1(S))}}   \left[\left(\frac{1}{\delta_{\al(i)}(L_1(S))} +\frac1{\Delta_1^2(\widetilde{\sigma}_i^a)} \right) \left(\frac{1}{\delta_{\al(j)}(L_1(S))} +\frac1{\Delta_1^2(\widetilde{\sigma}_j^a)} \right)   \right] ^{1/2} \\
&\lesssim \left[ \psi_{1}^{1/2}(\wt\sigma_i^a)+ \frac{\psi_{1}(\wt\sigma_i^a)\Delta_1(\widetilde{\sigma}_i^a) }{\delta_{\al(i)}(S)} \right]  \left[ \psi_{1}^{1/2}(\wt\sigma_j^a)+\frac{\psi_{1}(\wt\sigma_j^a)\Delta_1(\widetilde{\sigma}_j^a) }{\delta_{\al(j)}(S)}  \right] .
\end{split}
\ee
For the second term in \eqref{proofout_d1}, we use the estimate
$$\delta_{\al(i)}(S)\le C\delta_{\al(i)}(S\setminus L_1(S))\le Cn^{\wt\tau}\psi_1(\wt\sigma_i^a) $$
and case (b) to get that
\begin{align*}
 \langle \bv_i^a, \cal P_{S\setminus L_1(s)}\bv_i^a\rangle & \prec \Delta^2 _1(\wt\sigma_i^a)+\psi_{1}(\wt\sigma_i^a)+ n^{2\wt\tau}\left( \psi_{1} (\wt\sigma_i^a)+\frac{\psi_{1}^2(\wt\sigma_i^a)\Delta_1^2(\widetilde{\sigma}_i^a) }{\delta^2_{\al(i)}(S\setminus L_1(S))}  \right)\\
&\lesssim  n^{2\wt\tau}\left(\psi_{1} (\wt\sigma_i^a)+\frac{\psi_{1}^2(\wt\sigma_i^a)\Delta_1^2(\widetilde{\sigma}_i^a) }{\delta^2_{\al(i)}(S )} \right).
\end{align*}
A similar estimate holds for $\langle \bv_j^a, \cal P_{S\setminus L_1(s)}\bv_j^a\rangle$. Then we conclude that 
\begin{align*}
&\left| \langle \bv_i^a, \cal P_S\bv_j^a\rangle\right| \le \langle \bv_i^a, \cal P_S\bv_i^a\rangle^{1/2}\langle \bv_j^a, \cal P_S\bv_j^a\rangle^{1/2}  \\
& \prec n^{2\wt\tau}  \left[\psi_{1}^{1/2}(\wt\sigma_i^a)+  \frac{\psi_{1}(\wt\sigma_i^a)\Delta_1(\widetilde{\sigma}_i^a) }{\delta_{\al(i)}(S)} \right] \left[\psi_{1}^{1/2}(\wt\sigma_j^a)+  \frac{\psi_{1}(\wt\sigma_j^a)\Delta_1(\widetilde{\sigma}_j^a) }{\delta_{\al(j)}(S)} \right].
\end{align*}
This proves \eqref{eq_pfgoal} for case (iv), and hence concludes the proof for case (d).

\vspace{5pt}

Combining cases (c) and (d), we conclude \eqref{eq_spikedvector} for the $i\ne j$ case since $\wt \tau$ can be chosen arbitrarily small. This finishes the proof of Proposition \ref{prop_outev} together with \eqref{conclude_i=j}.
\end{proof}


\section{Non-outlier eigenvectors} \label{sec_nonoutliereve}

In this section, we first prove Theorem \ref{thm_noneve} of the paper, which will then be used to complete the proof of Theorem \ref{thm_eveout} of the paper. In other words, we will remove Assumption \ref{assu_strong} in Proposition \ref{prop_outev}.

Our first goal of this section is to prove the following proposition, from which the Theorem \ref{thm_noneve} of the paper follows.
 
\begin{proposition}\label{prop_nonspike} 
Fix a constant $\wt\tau \in (0, 1/3)$. For $\al( i)\notin \mathcal O^+$ and $ i \leq {\tau}p$, where $\tau >0$ is as given in Theorem \ref{thm_noneve} of the paper, we have
\begin{equation}\label{eq_nonspikeeg1}
 |\langle \bv_j^a, \wt\bxi_{\al(i)}  \rangle|^2 \prec 
\frac{n^{-1}+  \eta_l (\gamma_i) \sqrt{\kappa_{\gamma_i}}+\phi_n^3}{|\tsig_j^a+m_{2c}^{-1}(\lambda_+)|^2+\phi_n^2+\kappa_{\gamma_i}}, 
\end{equation} 
where we recall the definitions \eqref{KAPPA} and \eqref{etalE}. Moreover, if $\al(i) \in \mathcal O^+$ satisfies
\begin{equation}\label{eq_propohold}
  \tsig_i^a + m_c^{-1}(\lambda_+) \leq (\phi_n + n^{-1/3})n^{\wt\tau},  
\end{equation}
then we have 
\begin{equation}\label{eq_nonspikeeg2}
  |\langle \bv_j^a, \wt\bxi_{\al(i)}  \rangle|^2 \prec n^{4\wt\tau}\left( \frac{n^{-1}+  \eta_l (\gamma_i)  \sqrt{\kappa_{\gamma_i}}+\phi_n^3}{|\tsig_j^a+m_{2c}^{-1}(\lambda_+)|^2+\phi_n^2 + \kappa_{\gamma_i}}\right).  
\end{equation}
\end{proposition}

\begin{proof} 
 By Theorems \ref{thm_outlier}, \ref{thm_eigenvaluesticking},  \ref{LEM_SMALL},  \ref{thm_largerigidity}, \ref{lem_localout} and Lemma \ref{delocal_rigidity}, for any fixed $\e>0$, we can choose a high-probability event $\Xi_2$ in which \eqref{aniso_lawev}-\eqref{eq_bound2ev}, \eqref{eq_stickingrigi}-\eqref{eq_stickingrigi2}, \eqref{rigid_outXi1} and the following estimate hold:  
\begin{equation}\label{eq_tildelambda}
\begin{split}
 \mathbf{1}(\Xi_2)|\wt\lambda_i-\gamma_i| & \leq Cn^{\e/2}\left(i^{-1/3} n^{-2/3} + \eta_l(\gamma_i) + i^{-2/3} n^{-1/3}\phi_n^2 \right)\\
& +   Cn^{\e/2}\phi_n^2\mathbf 1_{|\al(i)|\le r+s } 
 \end{split}
\end{equation}
for $\al(i)\notin \mathcal O^+ \text{ and } i \leq \tau p.$   In fact, \eqref{eq_tildelambda} follows from \eqref{eq_stickingrigi} and \eqref{eq_stickingrigi2} combined with the interlacing, Lemma \ref{lem_weylmodi}.

Now we fix an $\al(i)\notin \mathcal O^+$ or $\al(i)\in \mathcal O^+$ satisfying \eqref{eq_propohold}, and some $1\le j \le {\tau}p$. As discussed at the end of Section \ref{sec_eve_green}, we may define $\mathcal R:=\{1,\cdots,r\}\cup \{j\}$ and can assume without loss of generality that $\wt\sigma_j^a$ also has a nonzero perturbation $d_j^a$ (even though it may not cause any outlier). For simplicity, we still use the unperturbed notations and denote $\wh A$ as $A$. 

We choose a specific spectral parameter as 
$z_i=\wt\lambda_i+\ri \eta_i $.
Here $\eta_i : = \wh \eta_i \vee n^{\e}\eta_{l}(\gamma_i)$, where $\wh\eta_i$ is defined as the 
solution of  
\begin{equation}\label{eq_defneta}
 \im m_{2c}(\wt\lambda_i + \ii\wh\eta_i)=n^{2\e}\phi_n + n^{-1+6\epsilon} \wh\eta_i^{-1}. 
\end{equation}
In fact, the solution exists and is unique since $\eta\im m_c(\wt\lambda_i + \ii\eta)$ is a strictly monotonically increasing function of $\eta$. With \eqref{eq_estimm}, one can check that
\be\label{eq_eta}
\wh\eta_i \sim \begin{dcases}
  n^{4 \epsilon}\left( \phi_n^2 + n^{-2/3}\right)  , \ &\text{if } \ |\wt\lambda_i - \lambda_+|\le n^{4 \epsilon}\left( \phi_n^2 + n^{-2/3}\right) \\
 n^{2\e}\phi_n \sqrt{\kappa_{\wt\lambda_i}}+ n^{-1/2+3 \epsilon} \kappa_{\wt\lambda_i}^{1/4},  \ &\text{if }\ \wt\lambda_i \geq  \lambda_++n^{4 \epsilon}\left( \phi_n^2 + n^{-2/3}\right)
\end{dcases},
\ee 
and if $\wt\lambda_i \leq \lambda_+-n^{4 \epsilon}\left( \phi_n^2 + n^{-2/3}\right)$, we have 
\be\label{eq_eta2}
\wh\eta_i \sim \begin{dcases}
  n^{4\e}\left( \phi_n^2 + n^{-2/3}\right) , \ &\text{if } \ \kappa_{\wt\lambda_i}\le n^{-2+4\e}\phi_n^{-4} \\
 n^{-1+6 \epsilon}\kappa_{\wt\lambda_i}^{-1/2}  ,  \ &\text{if }\ \kappa_{\wt\lambda_i}> n^{-2+4\e}\phi_n^{-4}
\end{dcases}.
\ee 
Note that by \eqref{eq_stickingrigi}, in order to have $\wt\lambda_i \geq  \lambda_++n^{4 \epsilon}\left( \phi_n^2 + n^{-2/3}\right)$, we must have $\al(i)\in \cal O^+$.
\nc
Moreover, with \eqref{etalE} and \eqref{eq_tildelambda}, we obtain that
\begin{align}
&\kappa_{\wt\lambda_i} - \kappa_{\gamma_i}  \nonumber\\
&\lesssim n^{\e/2}\left(  i^{-1/3}n^{-2/3} + \eta_l(\gamma_i)+i^{-2/3} n^{-1/3}\phi_n^2  +  \phi_n^2\mathbf 1_{|\al(i)|\le r+s}  \right) ,\label{kappaj}
\end{align}
and
\begin{align*}\eta_l(\wt\lambda_i) & \lesssim n^{-3/4} + n^{-1/2}\phi_n + n^{-1/2+\e/4}\left( \kappa_{\gamma_i} + \eta_l(\gamma_i)+ n^{-1/3}\phi_n^2  +  \phi_n^2\mathbf 1_{|\al(i)|\le r+s}  \right) ^{1/2} \\
&\lesssim n^{\e/4}\eta_l(\gamma_i).\end{align*}
In particular, 
we see that $z_i\in \wt S(\varsigma_1,\varsigma_2,\e)$ and 
(\ref{aniso_lawev}) can be applied at $z_i$. We consider two cases: (i) $\wh\eta_i \ge n^\e \eta_l(\gamma_i)$, and (ii) $\wh\eta_i < n^\e \eta_l(\gamma_i)$. In case (i), (\ref{aniso_lawev}) gives that
\begin{equation}\label{ll1}
\begin{split}
\norm{\bU^*(G(z_i)-\Pi(z_i))\bU} &\leq n^{\e/2}\phi_n + n^{\e/2} \Psi(z_i)  \lesssim n^{-3\e/2}\im m_{2c}(z_i).
\end{split}
\end{equation}
 In case (ii), with \eqref{eq_eta} and \eqref{kappaj} we can readily check that $\wt\lambda_i \leq \lambda_+-n^{4 \epsilon}\left( \phi_n^2 + n^{-2/3}\right)$ and $\kappa_{\wt\lambda_i} \gtrsim n^{-1/2+3\e} + n^{4\e}\phi_n^2,$ 
which further imply that 
\be\label{gammaisim}
\kappa_{\wt\lambda_i}=\kappa_{\gamma_i}(1+\oo(1)),\ \  n^{\e}\left(\eta_l(\gamma_i)+\frac{\phi_n^2}{i^{2/3} n^{1/3}} \right) \ll \kappa_{\gamma_i},\ \ \frac{n^{6\e}}{n\eta_l(\gamma_i)} \lesssim \sqrt{\kappa_{\gamma_i}}.
\ee
Together with \eqref{eq_estimm}, we get that
\be\label{ll2}
\begin{split}
&\norm{\bU^*(G(z_i)-\Pi(z_i))\bU} \leq n^{\e/2}\phi_n +n^{\e/2}\Psi(\wt\lambda_i + \ii n^{\e}\eta_l(\gamma_i)) \\
& \lesssim n^{\e/2}\phi_n +\sqrt{\frac{\sqrt{\kappa_{\gamma_i}}}{n\eta_l(\gamma_i)}} + \frac{1}{n^{1+\e/2}\eta_l(\gamma_i)} \\
& \le n^{\e/2}\phi_n + n^{-\e}\sqrt{\kappa_{\gamma_i}} \lesssim n^{-\e}\im m_{2c}(z_i).
\end{split}
\ee
 
%

After these preparations, we are ready to give the proof. As in \eqref{iso_delocal_pf}, with the spectral decomposition \eqref{eq_greenexpan}, we have the following bound
\begin{equation}\label{eq_nonspike1}
\left|\langle \bv^a_j, \wt\bxi_{\al(i)} \rangle\right|^2 \leq \eta_i \im \langle \bv_j^a, \widetilde{G}(z_i) \bv_j^a\rangle, 
\end{equation}
Applying \eqref{Huaineq} to \eqref{eq_gexpan}, we obtain another identity
\begin{equation} \label{eq_evgenerralreprest}
\begin{split}
\bU^*\widetilde{G}(z)\bU&=z^{-1}\wt{\mathcal D}^{1/2}\left({\mathcal D}^{-1}-{\mathcal D}^{-1}\frac{1}{\mathcal D^{-1}+ z\bU^*G(z)\bU} {\mathcal D}^{-1}\right) \wt{\mathcal D}^{1/2} .
\end{split}
\end{equation}
In particular, we have
\begin{align} \label{eq_expan_nonspike}
\begin{split}
&z \langle \bv_j^a, \widetilde{G}(z) \bv_j^a \rangle=\frac{1}{d^a_j}-\frac{1+d^a_j}{(d^a_j)^2} \left( \frac{1}{\mathcal D^{-1}+z \bU^* G(z) \bU} \right)_{jj}  \\
& =\frac{1}{d^a_j}-\frac{1+d^a_j}{(d^a_j)^2} \left[ \Phi_j (z)+ \Phi_j^2(z) \left( \mathcal E(z)+\mathcal E(z) \frac{1}{\mathcal D^{-1}+z \bU^* G(z) \bU} \mathcal E(z) \right)_{jj} \right],
\end{split}
\end{align}
where we used the resolvent expansion in \eqref{resolvent_3rd} and abbreviated 
\begin{equation*}
\Phi_j(z):=\frac{1}{(d_j^a)^{-1}+1-(1+m_{2c}(z) \sigma_j^a )^{-1}}.
\end{equation*}
By \eqref{ll1} and \eqref{ll2}, we have that 
\begin{equation*}
\min_j \left|(d_j^a)^{-1}+1-\frac{1}{1+m_{2c}(z_i) \sigma_j^a}\right|   \gtrsim  \im m_{2c}(z_i) \gg \|\mathcal E(z_i)\|. 
\end{equation*}
Thus as in (\ref{eq_boundsim1}), we conclude that 
\begin{equation*}
\left\| \frac{1}{\mathcal D^{-1}+z \bU^* G(z_i) \bU}\right\| \le \frac{C}{\text{Im} \ m_{2c}(z_i)} \ll \|\mathcal E(z_i)\|^{-1}.
\end{equation*} 
Inserting it into (\ref{eq_expan_nonspike}) and using \eqref{eq_matrixentry}, we obtain that
\begin{equation}\label{eq_reduceim}
z \langle \bv_j^a, \widetilde{G}(z_i) \bv_j^a \rangle=-(1+m_{2c}(z_i) \tsig_j^a)^{-1}+\OO\left(\frac{\|\mathcal E(z_i)\|}{ |1+m_{2c}(z_i) \tsig_j^a |^2} \right). 
\end{equation}

The next lemma provides a lower bound for $(1+m_{2c}(z) \tsig^a_j)^{-1}$. Its proof is the same as the one for (6.10) in \cite{principal}, where the only input is Lemma \ref{lem_mplaw}.

\begin{lemma}\label{lem_denominator} 
For any fixed $\delta \in [0, 1/3-\epsilon),$ there exists a constant $c>0$ such that 
\begin{equation*}
|1+m_{2c}(z_i) \tsig^a_j|  \geq c\left[n^{-2\delta}|\tsig_j^{a}+m_{2c}^{-1}(\lambda_+)|+\operatorname{Im} m_{2c}(z_i)\right]
\end{equation*}
holds whenever $\wt\lambda_i \in [0, \theta_1(-m_{2c}^{-1}(\lambda_+)+(\phi_n+n^{-1/3})n^{\delta+\e})].$
\end{lemma}
Now we fix the $\delta$ in Lemma \ref{lem_denominator}. By (\ref{eq_nonspike1}) and (\ref{eq_reduceim}), we have that
\begin{equation}\label{eq_finalbound}
\begin{split}
&\left|\langle \bv^a_j, \wt\bxi_{\al(i)} \rangle\right|^2   \leq -\eta_i \im \left[ z_i^{-1} (1+m_{2c}(z_i) \tsig^a_j)^{-1} \right]+\frac{C\eta_i\|\mathcal E(z_i)\| } {|1+m_{2c}(z_i) \tsig^a_j|^{2}} \\
& = -\frac{\eta_i^2}{|z_i|^2} \re\left(1+m_{2c}(z_i) \tsig^a_j\right)^{-1}- \frac{\eta_i \wt\lambda_i}{|z_i|^2}\im \left(1+m_{2c}(z_i) \tsig^a_j\right)^{-1} \\
&+\frac{C \eta_i \|\mathcal E(z_i)\|}{|1+m_{2c}(z_i) \tsig^a_j|^2}.
\end{split}
\ee
We next estimate the terms in (\ref{eq_finalbound}) one by one. First, $|z_i|\sim 1$ by \eqref{eq_tildelambda}
and hence we have 
\begin{align}\label{eq_firstone}
 -\frac{\eta_i^2}{|z_i|^2} \re\left(1+m_{2c}(z_i) \tsig^a_j\right)^{-1} \leq  \frac{C\eta_i^2}{|1+m_{2c}(z_i) \tsig^a_j|} \leq  \frac{C\eta_i^2}{\im m_{2c}(z_i)},  
\end{align}
where we used Lemma \ref{lem_denominator} in the second step. If $\wh\eta_i \ge n^\e \eta_l(\gamma_i)$, then with (\ref{eq_defneta}) we get
\be
  \eqref{eq_firstone} \le \frac{C\wh \eta_i^2}{n^{2\e}\phi_n + n^{-1+6\epsilon} \wh\eta_i^{-1}} \leq C n^{-1+6 \epsilon + 3\delta} + Cn^{2\e+\delta}\phi_n\wh \eta_i ,\nonumber
\ee
where we used that $\wh\eta_i \le (\phi_n^2+n^{-2/3})n^{\delta+4\e}$, as follows from \eqref{eq_eta}. If $\wh\eta_i < n^\e \eta_l(\gamma_i)$, by \eqref{eq_estimm} and \eqref{gammaisim} we get
\be
\eqref{eq_firstone} \le \frac{C \eta_i^2}{\sqrt{\kappa_{\gamma_i}}} \le n^{\e}\eta_l(\gamma_i)\sqrt{\kappa_{\gamma_i}}.\nonumber 
\ee
%
%
Similarly, for the second item of (\ref{eq_finalbound}), we have
\begin{equation*}
\begin{split}
- \frac{\eta_i \wt\lambda_i}{|z_i|^2}\im \left(1+m_{2c}(z_i) \tsig^a_j\right)^{-1} &\leq  \frac{C\eta_i \im m_{2c}(z_i)}{|1+m_{2c}(z_i) \tsig^a_j|^2}  \\
& \leq  \begin{dcases}\frac{C(n^{2\e}\phi_n \wh \eta_i+n^{-1+6 \epsilon})}{|1+m_{2c}(z_i) \tsig^a_j|^2}, \ &\text{if } \wh\eta_i \ge n^\e \eta_l(\gamma_i)\\
\frac{Cn^{\e}\eta_l(\gamma_i)\sqrt{\kappa_{\gamma_i}}}{|1+m_{2c}(z_i) \tsig^a_j|^2}, \ &\text{if } \wh\eta_i < n^\e \eta_l(\gamma_i)
\end{dcases}.
\end{split}
\end{equation*}
Finally, the third term of (\ref{eq_finalbound}) can be estimated using \eqref{ll1} and \eqref{ll2} by
\begin{equation*}
\frac{C \eta_i \|\mathcal E(z_i)\|}{|1+m_{2c}(z_i) \tsig^a_j|^2} \leq  \begin{dcases}\frac{ n^{\e}\phi_n \wh \eta_i+n^{-1+5 \epsilon}}{|1+m_{2c}(z_i) \tsig^a_j|^2}, \ &\text{if } \wh\eta_i \ge n^\e \eta_l(\gamma_i)\\
\frac{C\eta_l(\gamma_i)\sqrt{\kappa_{\gamma_i}}}{|1+m_{2c}(z_i) \tsig^a_j|^2}, \ &\text{if } \wh\eta_i < n^\e \eta_l(\gamma_i)
\end{dcases} . 
\end{equation*}
 Combining all the above estimates, we conclude that 
\begin{equation} \label{final_estimate}
\begin{split}
|\langle \bv^a_j, \wt\bxi_{\al(i)} \rangle|^2  &\lesssim n^{-1+6 \epsilon + 3\delta}+ n^{6\e+2\delta}\phi_n^3\\
&+\frac{ n^{-1+6\e+\delta} + n^{6\e+\delta}\phi_n^3 + n^\e \eta_l(\gamma_i)\sqrt{\kappa_{\gamma_i}}}{|1+m_{2c}(z) \tsig^a_j|^2},
\end{split}
\end{equation}
where we used that  for $\wh\eta_i \ge n^\e \eta_l(\gamma_i)$,
\begin{align*}
\phi_n \wh \eta_i &\lesssim n^{4\e+\delta}\phi_n (\phi_n^2 + n^{-2/3}) \lesssim n^{4\e+\delta} (\phi_n^3 + n^{-1})
\end{align*}
by \eqref{eq_eta} and \eqref{eq_eta2}.

We still need to estimate the denominator of \eqref{final_estimate} from below using Lemma \ref{lem_denominator}, which requires a lower bound on $\im m_{2c}(z_i).$ For $\al(i)\notin \mathcal O^+$, with (\ref{eq_estimm}), \eqref{eq_tildelambda}, (\ref{eq_eta}) and \eqref{eq_eta2}, we find that $\im m_{2c}(z_i) \gtrsim  \phi_n + \sqrt{\kappa_{\gamma_i}}$. Together with \eqref{final_estimate}, this concludes the proof of (\ref{eq_nonspikeeg1}) by choosing $\delta=0$ in Lemma \ref{lem_denominator}. On the other hand, when $\al(i) \in \mathcal O^+$ such that (\ref{eq_propohold}) holds, with \eqref{rigid_outXi1} and \eqref{eq_eta} we can verify that 
\begin{equation*}
\wt\lambda_i \leq \theta_1\left(-m_{2c}^{-1}(\lambda_+)+n^{\wt\tau+\epsilon}(n^{-1/3}+\phi_n)\right),
\end{equation*}
and
\begin{equation*}
\im m_{2c}(z_i) \ge (\phi_n +n^{-1/3})n^{2\epsilon-\wt\tau}\ge n^{-\wt\tau}(\phi_n + \sqrt{\kappa_{\gamma_i}}). 
\end{equation*}
We can therefore conclude the proof of (\ref{eq_nonspikeeg2}) with \eqref{final_estimate} by letting $\delta=\wt\tau$ in Lemma \ref{lem_denominator}. \nc
\end{proof}

\begin{proof}[Proof of Theorem \ref{thm_noneve} ]
We decompose 
\be\label{para+perp} \bv = \bv_{\parallel} + \bv_{\perp}, \quad \bv_{\parallel} := \sum_{i=1}^{r}v_i \bv_i^a ,\quad \bv_{\perp}:= \sum_{i>r} v_i \bv_i^a.\ee
Then the bound on $|\langle \bv_{\parallel}, \wt\bxi_{\al(i)} \rangle|^2$ is an easy corollary of (\ref{eq_nonspikeeg1}) using \eqref{etalE}.  For $|\langle \bv_{\perp}, \wt\bxi_{\al(i)} \rangle|^2$, we repeat the previous proof: applying similar arguments as below \eqref{eq_evgenerralreprest}, we get 
\begin{equation}\label{eq_reduceim2}
 \langle \bv_{\perp}, \widetilde{G}(z_i) \bv_{\perp} \rangle\prec  \|\mathcal E(z_i)\| ,\quad \mathcal E(z_i):=\bv_{\perp}^* (G(z_i)-\Pi(z_i))\bU
\end{equation} 
which is a similar version as in \eqref{eq_reduceim}. Then using \eqref{eq_nonspike1}, we get that
\begin{equation}\label{eq_nonspike1111111}
\left|\langle \bv_{\perp}, \wt\bxi_{\al(i)} \rangle\right|^2 \prec \eta_i \left( \phi_n + \Psi(z_i)\right) \lesssim n^{4\e} (n^{-1}+\phi_n^3), 
\end{equation}
where we used \eqref{aniso_law} in the first step, and \eqref{eq_defneta}-\eqref{eq_eta2} in the second step. This concludes the bound on $|\langle \bv_{\perp}, \wt\bxi_{\al(i)} \rangle|^2$.

If we have (a) (\ref{assm_3rdmoment}) of the paper holds, or (b) either $A$ or $B$ is diagonal, then we can remove the $\eta_i$ term and prove the stronger estimate (\ref{eq_evebulka_strong}) of the paper by using the stronger versions of Theorem \ref{thm_outlier}, Theorem \ref{LEM_SMALL} and Theorem \ref{thm_largerigidity}.
\end{proof}

Finally, we can prove Theorem \ref{thm_eveout} without the Assumption \ref{assu_strong}. 

\begin{proof}[Proof of Theorem \ref{thm_eveout}]
Suppose we have proved that \eqref{eq_spikedvector} holds for $S\subset \mathcal O^+$, where for all $\al( i) \in S $ and $\beta(\mu)\in S$, 
\begin{equation}\label{eq_spikeweak}
\wt{\sigma}^a_{i}+{m_{2c}^{-1}(\lambda_+)} \geq n^{-1/3}+\phi_n,\quad \wt{\sigma}^b_{\mu}+{m_{1c}^{-1}(\lambda_+)} \geq n^{-1/3}+\phi_n.
\end{equation}
Again we consider the decomposition \eqref{para+perp}. Since 
$$\langle \bv_{\parallel},  \cal Z_S \bv_{\perp}\rangle =\langle \bv_{\perp},  \cal Z_S \bv_{\perp}\rangle=0,$$ 
we have that
$$\left| \langle \bv, \left(\cal P_S   -  \cal Z_S\right) \bv\rangle\right| = \left|\langle \bv_{\parallel}, \left(\cal P_S-\cal Z_S\right) \bv_{\parallel}\rangle\right| + 2 \left|\langle \bv_{\parallel}, \cal P_S \bv_{\perp}\rangle\right| +  \left|\langle \bv_{\perp}, \cal P_s \bv_{\perp}\rangle\right| .$$
Now using \eqref{eq_spikedvector}, we obtain from Cauchy-Schwarz inequality that
\begin{align*}
& \left|\langle \bv_{\parallel}, \left(\cal P_s-\cal Z_S\right) \bv_{\parallel}\rangle\right| \\
&\prec  \sum_{1\le i\le r:\al(i)\in S}|v_i|^2  \psi_{1}(\wt\sigma_i^a)  +\sum_{1\le i\le r:\al(i)\notin S}|v_i|^2  \frac{\phi_n^2}{\delta_{\al(i)}(S)} + \sum_{1\le i \le r}{|v_i|^2} \frac{\psi_{1}^2(\wt\sigma_i^a)\Delta_1^2(\widetilde{\sigma}_i^a) }{\delta^2_{\al(i)}(S)}  \nonumber \\
&+ \langle \bv , \cal Z_S\bv \rangle^{1/2}\left[ \sum_{1\le i \le r:\al(i)\notin S}{|v_i|^2} \left(\frac{\psi_{1}^2(\wt\sigma_i^a)\Delta_1^2(\widetilde{\sigma}_i^a) }{\delta^2_{\al(i)}(S)}+ \frac{\phi_n^2}{\delta_{\al(i)}(S)}\right) \right]^{1/2} ,
\end{align*}
where we also used the fact that
$$ \frac{1}{\wt{\sigma}_i^a} \frac{g_{2c}'(-(\wt{\sigma}_i^a)^{-1})}{g_{2c}(-(\wt{\sigma}_i)^{-1})} \sim  \Delta^2_1(\widetilde{\sigma}_i^a) , \quad 1\le i \le r, $$
since $g_{2c}(-(\wt{\sigma}_i)^{-1}) \sim 1$, $\wt{\sigma}_i^a\sim 1$ and $g_{2c}'(-(\wt{\sigma}_i^a)^{-1})\sim \Delta_1^2(\widetilde{\sigma}_i^a)$ by \eqref{eq_gderivative}. 
For the term $\left|\langle \bv_{\perp}, \cal P_s \bv_{\perp}\rangle\right|$, using Theorem \ref{thm_noneve} and the estimate $|\wt\sigma_i^a+m_{2c}^{-1}(\lambda_+)|\sim 1 $ for $i>r$, we get 
\begin{align*}
\left|\langle \bv_{\perp}, \cal P_s \bv_{\perp}\rangle\right|&\prec \sum_{i> r} |v_i|^2 \left( n^{-1} + \eta_i \sqrt{\kappa_i} + \phi_n^3\right)\lesssim \sum_{i> r} |v_i|^2 \left( n^{-1} + n^{-1/2}\kappa_i+ \phi_n^3\right) ,
\end{align*}
where we used the definition of $\eta_i$ and $\kappa_i$ in the second step. For the term $\left|\langle \bv_{\parallel}, \cal P_S \bv_{\perp}\rangle\right| $, we use Cauchy-Schwarz inequality to get that 
\begin{align*}
&\left|\langle \bv_{\parallel}, \cal P_S \bv_{\perp}\rangle\right| \le \left|\langle \bv_{\parallel}, \cal P_S \bv_{\parallel}\rangle\right|^{1/2} \left|\langle \bv_{\perp}, \cal P_S \bv_{\perp}\rangle\right|^{1/2} \\
&\le \left|\langle \bv_{\parallel}, (\cal P_S-\cal Z_s) \bv_{\parallel}\rangle\right| + \left|\langle \bv_{\perp}, \cal P_S \bv_{\perp}\rangle\right|+ \left|\langle \bv , \cal Z_s \bv \rangle\right|^{1/2} \left|\langle \bv_{\perp}, \cal P_S \bv_{\perp}\rangle\right|^{1/2} .
\end{align*}
Combining the above estimates, we conclude (\ref{eq_spikePA}) of the paper using $\delta_{\al(i)}(S) \sim \Delta_1^2(\wt\sigma_i^a) \sim 1 $ for $i\ge r$.  If we have (a) (\ref{assm_3rdmoment}) of the paper holds, or (b) either $A$ or $B$ is diagonal, then we can remove the $n^{-1/2}\kappa_i$ term by using the stronger versions of Theorem \ref{thm_noneve}, Theorem \ref{LEM_SMALL} and Theorem \ref{thm_largerigidity}.


The rest of the proof is devoted to showing that \eqref{eq_spikedvector} holds for $S\subset \mathcal O^+$ where \eqref{eq_spikeweak} holds. Fix a constant $\epsilon>0.$ Note that it is easy to check by contradiction that there exists some $x_0 \in [1, r+s]$ satisfying the following gap property: for all $k$ such that $\tsig_k^a> -m_{2c}^{-1}(\lambda_+)+x_0 n^{\epsilon}(n^{-1/3}+\phi_n)$, we have $\tsig_k^a> -m_{2c}^{-1}(\lambda_+)+(x_0+1)n^{\epsilon}(n^{-1/3}+\phi_n)$. Following the idea in \cite[Section 6.2]{principal}, for such $x_0$, we split $S=S_{0} \cup S_1$ such that $\tsig_k^a \le -m_{2c}^{-1}(\lambda_+)+x_0 n^{\epsilon}(n^{-1/3}+\phi_n)$ for $\al(k) \in S_0$, and $\tsig_k^a> -m_{2c}^{-1}(\lambda_+)+(x_0+1)n^{\epsilon}(n^{-1/3}+\phi_n)$ for $\al(k)\in S_1$. Without loss of generality, we assume that $S_0\ne \emptyset$, since otherwise the claim already follows from Proposition \ref{prop_outev}.


There are totally six cases: (a) $\al(i),\al(j)\in S_0$; (b) $\al(i)\in S_0$ and $\al(j)\in S_1$; (c) $\al(i)\in S_0$ and $\al(j)\notin S$; (d) $\al(i),\al(j)\in S_1$; (e) $\al(i)\in S_1$ and $\al(j)\notin S$; (f) $\al(i),\al(j)\notin S$.  

\vspace{5pt}

\noindent{\bf Case (a)}: $\al(i),\al(j)\in S_0$. We have the splitting 
\be\label{split_general}
\langle \bv_i^a, \cal P_S\bv_j^a\rangle = \langle \bv_i^a, \cal P_{S_0}\bv_j^a\rangle +  \langle \bv_i^a, \cal P_{S_1}\bv_j^a\rangle.
\ee
Applying \eqref{CS new} and (\ref{eq_nonspikeeg2}) to the first term, and Proposition \ref{prop_outev} to the second term, we get that
\begin{align*}
&\left |\langle \bv_i^a, \cal P_S\bv_j^a\rangle-\delta_{ij}\frac{1}{\wt{\sigma}_i^a} \frac{g_{2c}'(-(\wt{\sigma}_i^a)^{-1})}{g_{2c}(-(\wt{\sigma}_i)^{-1})} \right|  \prec \delta_{ij}\Delta_1^2(\wt\sigma_i^a)+\frac{n^{4\e}\left( n^{-1} + \phi_n^3\right)}{\Delta_1^2(\widetilde{\sigma}_i^a) \Delta_1^2(\widetilde{\sigma}_j^a)} \\
&+\left[  \frac{\phi_n}{\delta_{\al(i)}^{1/2}(S_1)}+\frac{\psi_{1}(\wt\sigma_i^a)\Delta_1(\widetilde{\sigma}_i^a) }{\delta_{\al(i)}(S_1)} \right]   \left[  \frac{\phi_n}{\delta^{1/2}_{\al(j)}(S_1)}+\frac{\psi_{1}(\wt\sigma_j^a)\Delta_1(\widetilde{\sigma}_j^a) }{\delta_{\al(j)}(S_1)} \right]   \\
&\lesssim 
n^{4\e}\psi_1^{1/2}(\wt \sigma_i^a)\psi_1^{1/2}(\wt \sigma_j^a),
\end{align*}
where we used that $\eta_l (\gamma_i) \sqrt{\kappa_{\gamma_i}} \lesssim n^{-1} +\phi_n n^{-5/6}\lesssim n^{-1}+\phi_n^3$ for $k=\OO(1)$ in the first step, and $(n^{-1/3} +\phi_n)  \le \Delta_1^2(\wt\sigma_{i/j}^a) \lesssim n^\e(n^{-1/3} +\phi_n) \lesssim \delta_{\al(i/j)}(S_1) $ in the second step.

\vspace{5pt}

\noindent{\bf Case (b)}: $\al(i)\in S_0$ and $\al(j)\in S_1$. First suppose that Assumption \ref{ass:nonver} holds for some constant $0<\wt\tau<\e$. Applying Cauchy-Schwarz and Proposition \ref{prop_nonspike} to the first term in \eqref{split_general}, we get that
\begin{align*}
 &|\langle \bv_i^a, \cal P_{S_0}\bv_j^a\rangle|\prec \frac{n^{4\e}(n^{-1}+\phi_n^3)}{\Delta_1^2 (\widetilde{\sigma}_i^a)\Delta_1^2(\widetilde{\sigma}_j^a)} \lesssim n^{4\e} \psi_1^{1/2}(\wt\sigma_i^a)\psi_1^{1/2}(\wt\sigma_j^a).
\end{align*}
Applying \eqref{eq_spikedvector0} to the second term in \eqref{split_general}, we get that
\begin{align*}
 &|\langle \bv_i^a, \cal P_{S_1}\bv_j^a\rangle|\\
 &\prec \frac{\psi_1(\wt\sigma_j^a)\Delta^2_1(\widetilde{\sigma}_j^a)}{\delta_{\al(i),\al(j)}^a} +\psi_1^2(\wt\sigma_i^a)\Delta^2_1(\widetilde{\sigma}_i^a)  \left(\frac{1}{\delta_{\al(i)}(S_1)}+ \frac{1}{\Delta_1^2(\widetilde{\sigma}_i^a)}\right)\left(\frac1{\delta_{\al(j)}(S_1)} + \frac{1}{\Delta_1^2(\widetilde{\sigma}_j^a)} \right)\\
 & \lesssim \left[\psi_{1}^{1/2}(\wt\sigma_i^a)+  \frac{\psi_{1}(\wt\sigma_i^a)\Delta_1(\widetilde{\sigma}_i^a) }{\delta_{\al(i)}(S)}   \right] \left[\psi_{1}^{1/2}(\wt\sigma_j^a)+  \frac{\psi_{1}(\wt\sigma_j^a)\Delta_1(\widetilde{\sigma}_j^a) }{\delta_{\al(j)}(S)}  \right] ,
\end{align*}
where we used 
$$\delta_{\al(i)}(S_1) \gtrsim \Delta_1^2(\wt\sigma_{i}^a),\quad \delta_{\al(j)}(S_1)\gtrsim \Delta_1^2(\wt\sigma_{j}^a) \wedge \delta_{\al(j)}(S),\quad \psi_1(\wt\sigma_j^a) \lesssim \psi_1(\wt\sigma_i^a),$$ 
and
$$\delta_{\al(i),\al(j)}^a\gtrsim \Delta_1^2(\widetilde{\sigma}_j^a) \gtrsim \Delta_1^2(\widetilde{\sigma}_i^a),\quad \psi_1(\wt\sigma_j^a) \Delta_1(\wt\sigma_j^a)\gtrsim \psi_1(\wt\sigma_i^a)\Delta_1(\wt\sigma_i^a).$$ 
This concludes the proof of case (b) if the non-overlapping Assumption \ref{ass:nonver} holds. Otherwise, the argument is similar to the one in Section \ref{sec_remove} by using the set $L_1(S_1)$, and we ignore the details. 

\vspace{5pt}

\noindent{\bf Cases (c), (e) and (f)}: We use the splitting \eqref{split_general}, where we will apply \eqref{CS new} and Proposition \ref{prop_nonspike} to the first term, and Proposition \ref{prop_outev} to the second term.  Note that in all cases, we have $\delta_{\al(j)}(S)\le \delta_{\al(j)}(S_1)$ and $\delta_{\al(j)}(S)\lesssim n^\e (\Delta_1^2(\wt\sigma_{j}^a)+\phi_n + n^{-1/3} )$. In case (c) with $\al(i)\in S_0$ and $\al(j)\notin S$, we obtain that
\begin{align*}
 |\langle \bv_i^a, \cal P_{S_0}\bv_j^a\rangle| &\prec \frac{n^{4\e}(n^{-1}+\phi_n^3)}{\Delta_1^2(\widetilde{\sigma}_i^a) \left(\Delta_1^2(\widetilde{\sigma}_j^a)+\phi_n+n^{-1/3}\right)} \\
&  \lesssim n^{5\e}  \psi_1^{1/2}(\wt\sigma_i^a)  \left[  \frac{\phi_n}{ \delta^{1/2}_{\al(j)}(S)}+\frac{\psi_{1}(\wt\sigma_j^a)\Delta_1(\widetilde{\sigma}_j^a) }{\delta_{\al(j)}(S)} \right] ,
\end{align*}
where we also used $\delta_{\al(i)}(S_1) \gtrsim\Delta^2_1(\wt\sigma_i^a)\ge n^{-1/3}+\phi_n$ in the second step, and  
\begin{equation*}
\begin{split}
 |\langle \bv_i^a, \cal P_{S_1}\bv_j^a\rangle| &\prec \left[  \frac{\phi_n}{ \delta^{1/2}_{\al(i)}(S_1)}+\frac{\psi_{1}(\wt\sigma_i^a)\Delta_1(\widetilde{\sigma}_i^a) }{\delta_{\al(i)}(S_1)} \right] \left[  \frac{\phi_n}{ \delta^{1/2}_{\al(j)}(S_1)}+\frac{\psi_{1}(\wt\sigma_j^a)\Delta_1(\widetilde{\sigma}_j^a) }{\delta_{\al(j)}(S_1)} \right] \\
&  \lesssim    \psi_1^{1/2}(\wt\sigma_i^a) \left[  \frac{\phi_n}{ \delta^{1/2}_{\al(j)}(S)}+\frac{\psi_{1}(\wt\sigma_j^a)\Delta_1(\widetilde{\sigma}_j^a) }{\delta_{\al(j)}(S)} \right] .
\end{split}
\end{equation*}
In case (e) with $\al(i)\in S_1$ and $\al(j)\notin S$, the $ |\langle \bv_i^a, \cal P_{S_0}\bv_j^a\rangle|$ can be bounded in the same way as case (c). On the other hand,
\begin{align*}
&|\langle \bv_i^a, \cal P_{S_1}\bv_j^a\rangle|  \prec  \Delta_1(\widetilde{\sigma}_i^a)\left[  \frac{\phi_n}{ \delta^{1/2}_{\al(j)}(S_1)}+\frac{\psi_{1}(\wt\sigma_j^a)\Delta_1(\widetilde{\sigma}_j^a) }{\delta_{\al(j)}(S_1)} \right] \\
&+ \left[   \psi^{1/2} (\wt\sigma_i^a)+\frac{\psi_{1}(\wt\sigma_i^a)\Delta_1(\widetilde{\sigma}_i^a) }{\delta_{\al(i)}(S_1)}  \right] \left[  \frac{\phi_n}{ \delta^{1/2}_{\al(j)}(S_1)}+\frac{\psi_{1}(\wt\sigma_j^a)\Delta_1(\widetilde{\sigma}_j^a) }{\delta_{\al(j)}(S_1)} \right] \\
& \lesssim   \Delta_1(\widetilde{\sigma}_i^a)\left[  \frac{\phi_n}{ \delta^{1/2}_{\al(j)}(S)}+\frac{\psi_{1}(\wt\sigma_j^a)\Delta_1(\widetilde{\sigma}_j^a) }{\delta_{\al(j)}(S)} \right]  \\
&+\left[   \psi^{1/2} (\wt\sigma_i^a)+\frac{\psi_{1}(\wt\sigma_i^a)\Delta_1(\widetilde{\sigma}_i^a) }{\delta_{\al(i)}(S)} \right] \left[  \frac{\phi_n}{ \delta^{1/2}_{\al(j)}(S)}+\frac{\psi_{1}(\wt\sigma_j^a)\Delta_1(\widetilde{\sigma}_j^a) }{\delta_{\al(j)}(S)} \right] ,
\end{align*}
where we used $ \delta_{\al(i)}(S_1)\gtrsim \Delta_1^2(\wt\sigma_i^a) \wedge \delta_{\al(i)}(S)$ in the second step. In case (f) with $\al(i),\al(j)\notin S$, we obtain that
\begin{align*}
 |\langle \bv_i^a, \cal P_{S_0}\bv_j^a\rangle| &\prec \frac{n^{4\e}(n^{-1}+\phi_n^3)}{\left(\Delta_1^2(\widetilde{\sigma}_i^a)+\phi_n+n^{-1/3}\right) \left(\Delta_1^2(\widetilde{\sigma}_j^a)+\phi_n+n^{-1/3}\right)} \\
&  \lesssim n^{6\e}  \left[  \frac{\phi_n}{\delta_{\al(i)}^{1/2}(S)}+\frac{\psi_{1}(\wt\sigma_i^a)\Delta_1(\widetilde{\sigma}_i^a) }{\delta_{\al(i)}(S)} \right]   \left[  \frac{\phi_n}{\delta^{1/2}_{\al(j)}(S)}+\frac{\psi_{1}(\wt\sigma_j^a)\Delta_1(\widetilde{\sigma}_j^a) }{\delta_{\al(j)}(S)} \right] , 
\end{align*}
where in the second step we used 
$$ \delta_{\al(i/j)}(S)\lesssim n^\e (\Delta_1^2(\tsig_{i/j}^a) +\phi_n + n^{-1/3} ) .$$
For the $\cal P_{S_1}$ term, we have
\begin{align*}
 |\langle \bv_i^a, \cal P_{S_1}\bv_j^a\rangle| &\prec \left[  \frac{\phi_n}{\delta_{\al(i)}^{1/2}(S_1)}+\frac{\psi_{1}(\wt\sigma_i^a)\Delta_1(\widetilde{\sigma}_i^a) }{\delta_{\al(i)}(S_1)} \right]   \left[  \frac{\phi_n}{\delta^{1/2}_{\al(j)}(S_1)}+\frac{\psi_{1}(\wt\sigma_j^a)\Delta_1(\widetilde{\sigma}_j^a) }{\delta_{\al(j)}(S_1)} \right]  \\
& \le  \left[  \frac{\phi_n}{\delta_{\al(i)}^{1/2}(S)}+\frac{\psi_{1}(\wt\sigma_i^a)\Delta_1(\widetilde{\sigma}_i^a) }{\delta_{\al(i)}(S)} \right]   \left[  \frac{\phi_n}{\delta^{1/2}_{\al(j)}(S)}+\frac{\psi_{1}(\wt\sigma_j^a)\Delta_1(\widetilde{\sigma}_j^a) }{\delta_{\al(j)}(S)} \right] ,
\end{align*}
where we used $\delta_{\al(i/j)}(S)\le \delta_{\al(i/j)}(S_1)$ in the second step.

\vspace{5pt}

\noindent{\bf Case (d)}: $\al(i),\al(j)\in S_1$.  Again using \eqref{split_general}, Proposition \ref{prop_nonspike} and Proposition \ref{prop_outev}, we get that 
\begin{align*}
& \left|\langle \bv_i^a, \cal P_S\bv_j^a \rangle - \delta_{ij}\frac{1}{\wt{\sigma}_i^a} \frac{g_{2c}'(-(\wt{\sigma}_i^a)^{-1})}{g_{2c}(-(\wt{\sigma}_i)^{-1})}\rangle\right| \prec \frac{n^{4\e}(n^{-1}+\phi_n^3)}{\Delta^2_1(\widetilde{\sigma}_i^a) \Delta^2_1(\widetilde{\sigma}_j^a)}  \\
&+\left[   \psi^{1/2} (\wt\sigma_i^a)+\frac{\psi_{1}(\wt\sigma_i^a)\Delta_1(\widetilde{\sigma}_i^a) }{\delta_{\al(i)}(S_1)} \right]\left[   \psi^{1/2} (\wt\sigma_j^a)+\frac{\psi_{1}(\wt\sigma_j^a)\Delta_1(\widetilde{\sigma}_j^a) }{\delta_{\al(j)}(S_1)}\right] \\
&\prec n^{4\e} \left[   \psi^{1/2} (\wt\sigma_i^a)+\frac{\psi_{1}(\wt\sigma_i^a)\Delta_1(\widetilde{\sigma}_i^a) }{\delta_{\al(i)}(S)}\right]\left[   \psi^{1/2} (\wt\sigma_j^a)+\frac{\psi_{1}(\wt\sigma_j^a)\Delta_1(\widetilde{\sigma}_j^a) }{\delta_{\al(j)}(S)} \right] .
\end{align*}
where we used $ \delta_{\al(i/j)}(S_1)\gtrsim \Delta_1^2(\wt\sigma_{i/j}^a)\wedge \delta_{\al(i/j)}(S)$ in the second step. 

\vspace{5pt}

Combining all the above six cases, we conclude that even without the Assumption \ref{assu_strong}, the estimate \eqref{eq_spikedvector} still holds with an additional factor $n^{6\e}$ multiplying with the RHS. Since $\e$ can be arbitrarily small, we conclude the proof.
\end{proof}

\end{appendix}

\bibliographystyle{abbrv}
\bibliographystyle{plain}
\bibliography{sepcov}

\end{document}